\documentclass{amsart}

\usepackage[normalem]{ulem}
\usepackage{amsmath,amsthm, amssymb,mathrsfs,  xypic, url, verbatim,stmaryrd}
\usepackage{xr-hyper}
\usepackage{hyperref}
\usepackage{stackengine,scalerel}
\usepackage{todonotes}
\presetkeys{todonotes}{inline, size=\small}{}
\usepackage[capitalize]{cleveref}
\usepackage{graphicx}
\usepackage{quiver}
\usepackage{tikz}
\usetikzlibrary{matrix, arrows}
\usetikzlibrary{cd}

\crefname{maintheorem}{Theorem}{Theorems}
\crefname{maincorollary}{Corollary}{Corollaries}

\begin{document}
%macros
\newcommand{\actsonr}{\mathrel{\reflectbox{$\righttoleftarrow$}}}
\newcommand{\actsonl}{\mathrel{\reflectbox{$\lefttorightarrow$}}}

\newcommand{\floor}[1]{\lfloor #1 \rfloor}

\newcommand{\isoeq}{\cong}
\newcommand{\cech}{\vee}
\newcommand{\dsum}{\mathop{\oplus}}
\newcommand{\End}{\mathrm{End}}

\newcommand{\calQ}{\mathcal{Q}}
\newcommand{\calO}{\mathcal{O}}
\newcommand{\calM}{\mathcal{M}}

\newcommand{\frakm}{\mathfrak{m}}

\newcommand{\bbA}{\mathbb{A}}
\newcommand{\bbB}{\mathbb{B}}
\newcommand{\bbC}{\mathbb{C}}
\newcommand{\bbD}{\mathbb{D}}
\newcommand{\bbE}{\mathbb{E}}
\newcommand{\bbF}{\mathbb{F}}
\newcommand{\bbG}{\mathbb{G}}
\newcommand{\bbH}{\mathbb{H}}
\newcommand{\bbI}{\mathbb{I}}
\newcommand{\bbJ}{\mathbb{J}}
\newcommand{\bbK}{\mathbb{K}}
\newcommand{\bbL}{\mathbb{L}}
\newcommand{\bbM}{\mathbb{M}}
\newcommand{\bbN}{\mathbb{N}}
\newcommand{\bbO}{\mathbb{O}}
\newcommand{\bbP}{\mathbb{P}}
\newcommand{\bbQ}{\mathbb{Q}}
\newcommand{\bbR}{\mathbb{R}}
\newcommand{\bbS}{\mathbb{S}}
\newcommand{\bbT}{\mathbb{T}}
\newcommand{\bbU}{\mathbb{U}}
\newcommand{\bbV}{\mathbb{V}}
\newcommand{\bbW}{\mathbb{W}}
\newcommand{\bbX}{\mathbb{X}}
\newcommand{\bbY}{\mathbb{Y}}
\newcommand{\bbZ}{\mathbb{Z}}

\newcommand{\bc}{\mathbf{c}}

\newcommand{\Spec}{\mathrm{Spec}}\
\newcommand{\Triv}{\mathrm{Triv}}
\newcommand{\Loc}{\mathrm{Loc}}
\newcommand{\Et}{\mathrm{Et}}

\newcommand{\ul}[1]{\underline{#1}}

\newcommand{\Vect}{\mathrm{Vect}}
\newcommand{\FilVect}{\mathrm{FilVect}}
\newcommand{\FilVB}{\mathrm{FilVB}}

\newcommand{\trFil}{\mathrm{trFil}}
\newcommand{\gr}{\mathrm{gr}}
\newcommand{\LattFilt}[1]{\mr{Latt}_{#1}\mr{Vect}}

\newcommand{\dR}{\mathrm{dR}}

\newcommand{\HS}{\mathrm{HS}}

\newcommand{\Hodge}{\mathrm{Hdg}}

\newcommand{\BC}{\mathrm{BC}}
\newcommand{\CM}{\mathrm{CM}}

\newcommand{\goodred}{\textrm{good-red}}

% Typesetting
\newcommand{\dash}{\mathrm{-}}
\newcommand\+{\mkern3mu}

% Revision commands
\newcommand{\rev}[1]{{\color{teal} #1}}
\newcommand{\strike}[1]{{\color{gray} \sout{#1}}}
\newcommand{\spc}[1]{{\color{blue} \textsf{$\blacktriangle\blacktriangle\blacktriangle$ Comment: [#1]}}}

\newcommand{\triv}{\mr{triv}}
\newcommand{\crys}{\mathrm{crys}}
\newcommand{\Fil}{F}
\newcommand{\calL}{\mathcal{L}}
\newcommand{\Stab}{\mathrm{Stab}}
\newcommand{\Spd}{\mathrm{Spd}}
\newcommand{\mbb}[1]{\mathbb{#1}}
\newcommand{\HT}{\mathrm{HT}}
\newcommand{\GH}{\mathrm{GH}}
\newcommand{\GM}{\mathrm{GM}}
\newcommand{\Gr}{\mathrm{Gr}}
\newcommand{\Fl}{\mathrm{Fl}}
\newcommand{\LT}{\mathrm{LT}}
\newcommand{\colim}{\mathrm{colim}}
\newcommand{\cris}{\mathrm{cris}}
\renewcommand{\l}{\left}
\renewcommand{\r}{\right}
\newcommand{\GL}{\mathrm{GL}}
\newcommand{\Sh}{\mathrm{Sh}}

\newcommand{\Perf}{\mathrm{Perf}}
\newcommand{\Perfd}{\mathrm{Perfd}}
\newcommand{\mc}[1]{\mathcal{#1}}
\renewcommand{\mbb}[1]{\mathbb{#1}}
\newcommand{\mr}[1]{\mathrm{#1}}
\newcommand{\mf}[1]{\mathfrak{#1}}
\newcommand{\ms}[1]{\mathscr{#1}}

\newcommand{\ol}[1]{\overline{#1}}

\newcommand{\Kt}{\mathrm{Kt}}
\newcommand{\Isoc}{\mr{Isoc}}

\newcommand{\basic}{\mr{basic}}
\newcommand{\ab}{\mr{ab}}
\newcommand{\Spa}{\mathrm{Spa}}
\newcommand{\Spf}{\mathrm{Spf}}
\newcommand{\perf}{\mathrm{perf}}
\newcommand{\Spr}{\mr{Spr}}
\newcommand{\Hom}{\mr{Hom}}
\newcommand{\Ker}{\mr{Ker}}
\newcommand{\adm}{\mr{adm}}
\newcommand{\bs}{\backslash}
\newcommand{\BB}{\mathrm{BB}}
\newcommand{\an}{\mathrm{an}}
\newcommand{\AdmFil}{\mathrm{AdmFil}} 
\newcommand{\Id}{\mr{Id}}

%Notational shortcuts
\newcommand{\QQ}{\mathbb{Q}}
\newcommand{\ZZ}{\mathbb{Z}}

\newcommand{\FF}{\mathrm{FF}}

\newcommand{\proet}{\mr{pro\acute{e}t}}
\newcommand{\et}{\mr{\acute{e}t}}
\newcommand{\fet}{\mr{f\acute{e}t}}

\newcommand{\badm}{{\textrm{b-adm}}}
\newcommand{\Qpbreve}{{\breve{\mbb{Q}}_p}}
\newcommand{\Qpbrevebar}{{\overline{\breve{\mbb{Q}}}_p}}

\newcommand{\Hdg}{\mr{Hdg}}
\newcommand{\Hdggen}{\mr{gen}}
\newcommand{\MT}{\mr{MT}}
\newcommand{\Rep}{\mathrm{Rep}}

\newcommand{\FilPhiMod}{\mathrm{MF}^{\varphi}}
\newcommand{\waFilPhiMod}{\mathrm{MF}^{\varphi, \mr{wa}}}

\newcommand{\Isom}{\mr{Isom}}
\newcommand{\Gal}{\mr{Gal}}
\newcommand{\Aut}{\mr{Aut}}
\newcommand{\AdmPair}{\mr{AdmPair}}
\newcommand{\bAdmPair}{\AdmPair^\basic}
\newcommand{\Mot}{\mr{Mot}}
\newcommand{\Forget}{\mr{Forget}}

\newcommand{\Cont}{\mr{Cont}}
\newcommand{\AP}{\mathbf{A}}
\newcommand{\GAP}{\mathbf{A}}
\newcommand{\ev}{\mathrm{ev}}
\newcommand{\std}{\mathrm{std}}
\newcommand{\Lie}{\mr{Lie}}
\newcommand{\spck}[1]{{\color{teal} \textsf{$\dagger\dagger\dagger$ CK: [#1]}}}

\newcommand{\bl}{\mr{bl}}

\numberwithin{equation}{subsubsection}
%theorem environments
\theoremstyle{plain}

%the following are labelled without section numbers
\newtheorem{maintheorem}{Theorem} 
\renewcommand{\themaintheorem}{\Alph{maintheorem}} 
\newtheorem{maincorollary}[maintheorem]{Corollary}
\newtheorem{mainconjecture}[maintheorem]{Conjecture}

%unlabelled 
\newtheorem*{theorem*}{Theorem}

%theorems, etc., labeled in same counter as subsubsection
\newtheorem{theorem}[subsubsection]{Theorem}
\newtheorem{corollary}[subsubsection]{Corollary}
\newtheorem{conjecture}[subsubsection]{Conjecture}
\newtheorem{proposition}[subsubsection]{Proposition}
\newtheorem{lemma}[subsubsection]{Lemma}

\theoremstyle{definition}

\newtheorem{example}[subsubsection]{Example}
\newtheorem{assumption}[subsubsection]{Assumption}
\newtheorem{definition}[subsubsection]{Definition}
\newtheorem{remark}[subsubsection]{Remark}
\newtheorem{question}[subsubsection]{Question}

\renewcommand{\Vect}{\mathrm{Vect}}

\title[Admissible pairs and $p$-adic Hodge structures I]{Admissible pairs and $p$-adic Hodge structures I: transcendence of the de Rham lattice}

\author{Sean Howe}
\email{sean.howe@utah.edu}
\author{Christian Klevdal}
\email{cklevdal@ucsd.edu}

\begin{abstract}
For an algebraically closed non-archimedean extension $C/\mbb{Q}_p$, we define a Tannakian category of \emph{$p$-adic Hodge structures} over $C$ that is a local, $p$-adic structural analog of the global, archimedean category of $\mbb{Q}$-Hodge structures in complex geometry. In this setting the filtrations of classical Hodge theory must be enriched to lattices over a complete discrete valuation ring, Fontaine's integral de Rham period ring $B^+_\dR$, and a pure $p$-adic Hodge structure is then a $\mbb{Q}_p$-vector space equipped with a $B^+_\dR$-lattice satisfying a natural condition analogous to the transversality of the complex Hodge filtration with its  conjugate. We show $p$-adic Hodge structures are equivalent to a full subcategory of \emph{basic} objects in the category of \emph{admissible pairs}, a toy category of cohomological motives over $C$ that is equivalent to the isogeny category of rigidified Breuil-Kisin-Fargues modules and closely related to Fontaine's $p$-adic Hodge theory over $p$-adic subfields. As an application, we characterize basic admissible pairs with complex multiplication in terms of the transcendence of $p$-adic periods. This generalizes an earlier result for one-dimensional formal groups and is an unconditional, local, $p$-adic analog of a global, archimedean  characterization of CM motives over $\mbb{C}$ conditional on the standard conjectures, the Hodge conjecture, and the Grothendieck period conjecture (known unconditionally for abelian varieties by work Cohen and Shiga and Wolfart). 
\end{abstract}

\maketitle

\tableofcontents

\section{Introduction}\label{s.introduction}

The de Rham comparison isomorphism equips the rational singular cohomology of a complex algebraic variety $X$ with a Hodge filtration defined over $\mbb{C}$. After choosing a basis for the singular cohomology, the modulus of this Hodge filtration is a complex point in a partial flag variety defined over $\mbb{Q}$, and a fundamental question in algebraic geometry is to relate algebraic and arithmetic properties of these period moduli to algebraic and arithmetic properties of the equations defining $X$ (see, e.g, \cite{cattani-deligne-kaplan}). In particular, one would like to understand which algebraic conditions on the flag variety correspond to algebraic conditions on the coefficients of the equations defining $X$, and vice versa. The goal of this paper and its sequel is to explore a local, $p$-adic analog of this question, where complex algebraic varieties are replaced with rigid analytic varieties over a complete algebraically closed extension of $\mbb{Q}_p$ (eventually embedded in the more versatile geometric framework of diamonds). In this first part we focus purely on the theory over a point, which can be handled  algebraically without knowledge of the modern foundations of $p$-adic geometry. 

A fundamental example over $\mbb{C}$ is Schneider's \cite{schneider:j} 1937 result on the transcendence of the $j$-invariant. Recall that a complex elliptic curve $E$ can be presented analytically as a quotient $\mbb{C}/(\mbb{Z}+\mbb{Z}\tau)$ for $\tau \in \mbb{H}^\pm \subset \mbb{P}^1(\mbb{C})$ or algebraically as the solution set of a Weierstrass equation $y^2=4x^3+ax+b$. The $j$-invariant
\[ j=1728\frac{a^3}{a^3-27b^2} = q^{-1}+744+196884q + \ldots (q=e^{2\pi i\tau})\]
is independent of choices and controls the field of definition of $E$ as an algebraic variety. Schneider's theorem says that if $j$ and $\tau$ are both algebraic over $\mbb{Q}$ then $E$ has complex multiplication (CM). Conversely, if $E$ has CM then $\tau$ is quadratic imaginary and $j$ is contained in an abelian extension of $\mbb{Q}(\tau)$. In \cite{howe:transcendence}, one author proved an analog where elliptic curves are replaced with one-dimensional $p$-divisible formal groups, singular homology is replaced with the $p$-adic Tate module, and the Hodge filtration is replaced with the Hodge-Tate filtration: let $C/\mbb{Q}_p$ be an algebraically closed non-archimedean\footnote{A non-archimedean field is a field that is complete for a non-archimedean absolute value.} extension, let $\mc{O}_C$ be the valuation ring in $C$ with maximal ideal $\mf{m}_C$ and residue field $\kappa=\mc{O}_C/\mf{m}_C$, and let $\overline{\mbb{F}}_p$ be the algebraic closure of $\mbb{F}_p$ in $\kappa$. For $W$ the $p$-typical Witt vectors, let $C_0=W(\kappa)[1/p]$, which we identify with the maximal complete absolutely unramified sub-extension of $C$. Let $\breve{\mbb{Q}}_p=W(\overline{\mbb{F}}_p)[1/p] \subseteq C_0$. For a subfield $K \subseteq C$, $\overline{K}$ is its algebraic closure in $C$. 

For $G/\mc{O}_C$ a one-dimensional $p$-divisible formal group, we say $G$ is $\overline{C}_0$-analytic if there is a finite extension $K/C_0$ in $C$ and a choice of formal coordinate such that the power series defining the group law has coefficients in $\mc{O}_K$. We say the Hodge-Tate filtration $\Lie\+ G(1) \subseteq T_p G \otimes C$ is $\overline{C}_0$-analytic if it is defined over $\overline{C}_0$, i.e. there is a basis of $\Lie\+ G(1)$ consisting of vectors in $T_p G \otimes \overline{C}_0$. We say $G$ has complex multiplication (CM) if it admits isogenies by a semisimple $\mbb{Q}_p$-algebra of dimension equal to the height of $G$ (i.e. the rank of the free $\mbb{Z}_p$-module $T_p G$).

\begin{theorem}\label{theorem.transcendence-ht-filt-one-dim}{\cite{howe:transcendence}} Let $G/\mc{O}_C$ be a one-dimensional $p$-divisible formal group. If $G$ and the Hodge-Tate filtration are both $\overline{C}_0$-analytic, then $G$ has CM. Conversely, if $G$ has CM, then the Hodge-Tate filtration is defined over a finite extension $K/\mbb{Q}_p$ of degree equal to the height of $G$, and $G$ is defined over an abelian extension of $K$.  
\end{theorem}

Schneider's theorem is a global, archimedean transcendence result because the transcendence considered in both the defining equations of varieties and in cohomology is with respect to an archimedean extension of the global field $\mbb{Q}$. By contrast, \cref{theorem.transcendence-ht-filt-one-dim} is a local, $p$-adic transcendence result because it describes the transcendence of numbers in $C$ over $p$-adic subfields. In both cases, the new  result was the condition for CM, while the more refined information about fields of definition in the CM case was already known. In the following, we will focus on similar characterizations of CM in more general settings and mostly leave aside the refined information on fields of definition. 

\subsection{Transcendence results}

Schneider's theorem was generalized by Cohen \cite{cohen:transcendence} and Shiga and Wolfart \cite{shiga-wolfart:transcendence} to all abelian varieties using the W\"ustholz analytic subgroup theorem. Conditional on strong conjectures, we further generalize to all motives over $\mbb{C}$ (i.e., to the cohomology of all smooth projective complex varieties). We say an object $M$ in a connected neutral Tannakian category has complex multiplication (CM) if its Tannakian structure group, i.e. the automorphism group of any fiber functor on the Tannakian subcategory $\langle M \rangle$ generated by $M$, is a torus. 

\begin{maintheorem}\label{main.complex}Assume the standard conjectures, the Hodge conjecture, and the Grothendieck period conjecture. Let $M \in 
\Mot(\mbb{C})$, the Tannakian category of pure motives over $\mbb{C}$ with $\QQ$-coefficients \cite{andre:motives}. Then $M$ has CM if and only if $M \in \Mot(\overline{\mbb{Q}})$ and the Hodge filtration on the Betti realization of $M$ is $\overline{\mbb{Q}}$-algebraic.
\end{maintheorem}

The main result of this paper, \cref{main.transcendence} below, is an unconditional local, $p$-adic analog of \cref{main.complex}. It includes, as a special case, a generalization of \cref{theorem.transcendence-ht-filt-one-dim} from one-dimensional formal groups to all isoclinic formal groups (see \cref{example.p-divisible-groups}), analogous to Cohen and Shiga and Wolfart's generalization of Schneider's theorem from elliptic curves to all complex abelian varieties. Once the general formalism is established, the proof is almost identical to the proof of \cref{main.complex} (a motivated reader can find the proof of the latter in \cref{ss.complex-transcendence} before continuing the introduction). To state it, we first introduce some definitions.

The analog of $\Mot(\mbb{C})$ will be the $\mbb{Q}_p$-linear Tannakian category $\AdmPair^\basic(C)$ of \emph{basic admissible pairs} over $C$. It is a full Tannakian sub-category of a category $\AdmPair(C)$ of admissible pairs over $C$, a toy-category of cohomological motives over $C$ which we give an overview of now. (For precise definitions see \cref{s.admissible-pairs}.) The objects of $\AdmPair(C)$ are isocrystals\footnote{An isocrystal for us is always a finite dimensional $\breve{\mbb{Q}}_p$-vector space equipped with a semi-linear automorphism (semi-linearity with respect to the Frobenius lift on $\breve{\mbb{Q}}_p$). One could also setup the theory using $E$-isocrystals where $\breve{\mbb{Q}}_p$ is replaced with $\breve{E}$ for $E/\mbb{Q}_p$ a finite extension, but the additional generality gained is subsumed by consideration of objects in our category with endomorphisms by $E$, and in particular by allowing more general $G$-structure.} equipped with some extra structure. To describe this extra structure, first recall that Fontaine has defined a natural complete discrete valuation ring $B^+_\dR$ with residue field $C$; as usual, we write $B_\dR=\mr{Frac}(B^+_\dR)$. Non-canonically, we have $B^+_\dR \cong C[[t]]$ and $B_\dR \cong C((t))$. The extra structure we consider is that of a $B^+_\dR$-lattice after base change to $B_\dR$. An object of $\AdmPair(C)$ is thus a pair $(W,\mc{L}_\et)$, where $W$ is an isocrystal and $\mc{L}_\et \subseteq W_{B_\dR} = W \otimes_{\breve{\QQ}_p} B_\dR$ is a $B^+_\dR$-lattice subject to an admissibility condition (see \cref{def.admpair}). It is \emph{basic} if the slope grading on the isocrystal induces a grading also on the lattice $\mc{L}_\et$, and we write $\AdmPair^\basic(C)$ for the full Tannakian subcategory of basic admissible pairs. Note that for an admissible pair $(W,\mc{L}_\et)$, the lattice $\mc{L}_\et$ induces a trace filtration $F^\bullet W_C$ (called the \emph{Hodge filtration} on $W_C$: $\Fil^i W_C$ is the image of $W_{B^+_\dR} \cap t^i \mc{L}_\et$ under the specialization map $W_{B^+_\dR}\rightarrow W_C$).

Admissible pairs arise, for example, from the cohomology of smooth proper formal schemes over $\mc{O}_C$ (using $A_{\mr{inf}}$ or prismatic cohomology \cite{bhatt-morrow-scholze:1, bhatt-scholze:prismatic}, see \cref{example.geometric-admissible-pairs}). In this case, the isocrystal comes from the crystalline cohomology of the special fiber and the lattice comes from the comparison with \'{e}tale cohomology and refines the Hodge filtration. There are equivalent presentations of this category, e.g. as the isogeny category of rigidified Breuil-Kisin-Fargues modules, but here we prefer this elementary perspective to emphasize the connection with classical definitions of period domains in $p$-adic Hodge theory. 

There is a canonical inclusion $\overline{C}_0 \subset B^+_\dR$. We say an admissible pair is $\overline{C}_0$-analytic if the Hodge filtration on $W_C$ is defined over $\overline{C}_0$ and if $\mc{L}_\et$ is obtained as the convolution of the Hodge filtration on $W_{\overline{C}_0}$ and the valuation filtration on $B_\dR$, that is, if $\mc{L}_\et=\sum_{i \in \mbb{Z}} \Fil^{-i} B_\dR \cdot \Fil^i W_{\overline{C}_0}$. The $\overline{C}_0$-analytic admissible pairs form a full Tannakian subcategory 
$\AdmPair(\overline{C}_0)$ of $\AdmPair(C)$. We similarly have a full Tannakian subcategory $\AdmPair(\overline{\breve{\mbb{Q}}}_p)$ of $\overline{\breve{\mbb{Q}}}_p$-analytic admissible pairs.  

\begin{remark}
The rationality of an admissible pair is determined by the rationality of the \'{e}tale lattice with respect to the $\breve{\mbb{Q}}_p$-vector space underlying the isocrystal. In particular, for admissible pairs coming from cohomology of smooth proper formal schemes over $\mc{O}_C$, if the formal scheme is defined over $\mc{O}_K$ for a complete subfield $K \subseteq \overline{C}_0$ then the admissible pair is $\overline{C}_0$-analytic. 
\end{remark}

There is also a natural \emph{linear} realization of $\AdmPair(C)$ given by a functor to the category of $\mbb{Q}_p$-vector spaces equipped with a $B^+_\dR$-lattice after base change to $B_\dR$. In the cohomological setting, the vector space is $\mbb{Q}_p$-\'{e}tale cohomology, the lattice is a canonical deformation of de Rham cohomology, and the trace filtration is the Hodge-Tate filtration on $C$-\'{e}tale cohomology. Moreover, in this case the data depends only on the rigid analytic generic fiber (it determines the underlying Breuil-Kisin-Fargues module, but not the rigidification), and in fact such cohomological pairs exist for any smooth proper rigid analytic variety over $C$. 

For a pair $(V,\mc{L}_\dR)$ consisting of a $\mbb{Q}_p$-vector space $V$ and a $B^+_\dR$-lattice $\mc{L}_\dR \subseteq V_{B_\dR}$, we say $\mc{L}_\dR$ is $\overline{C}_0$-analytic (resp. $\overline{\mbb{Q}}_p$-analytic) if the associated filtration on $V_{C}$ is defined over $\overline{C}_0$ (resp. $\overline{\mbb{Q}}_p$) and $\mc{L}_\dR$ is the convolution of this filtration on $V_{\overline{C}_0}$ (resp. $V_{\overline{\mbb{Q}}_p}$) and the valuation filtration on $B_\dR$. In the cohomological setting, the rationality of $\mc{L}_\dR$ against the underlying $\mbb{Q}_p$-vector space $V$ is divorced from the rationality of the variety itself --- it is thus a natural analog of the rationality of the complex Hodge filtration against singular cohomology in complex geometry.

The following unconditional local, $p$-adic analog of \cref{main.complex} is our main result:  
\begin{maintheorem}\label{main.transcendence}
Let $M \in \AdmPair^\basic(C)$, the Tannakian category of basic admissible pairs over $C$.  If $M \in \AdmPair^\basic(\overline{C}_0)$ and the de Rham lattice on the \'{e}tale realization of $M$ is $\overline{C}_0$-analytic, then $M$ has CM. Conversely, if $M$ has CM, then $M \in \AdmPair^\basic({\overline{\breve{\mbb{Q}}}_p})$ and the de Rham lattice is ${\overline{\mbb{Q}}_p}$-analytic.  
\end{maintheorem}

A key motivational observation is that we can put a condition on a $B^+_\dR$-latticed $\mbb{Q}_p$-vector space that is completely analogous to the transversality condition of a filtration with its complex conjugate in the definition of a weight $n$ pure Hodge structure. This condition allows us to cut out a category of \emph{$p$-adic Hodge structures} which is equivalent to the category of \emph{basic} admissible pairs. The relation between basic admissible pairs and $p$-adic Hodge structures is a natural \emph{structural} analog of the relation between motives over $\mbb{C}$ and complex Hodge structures, and this is why the basic hypothesis appears in \cref{main.transcendence}.  

\begin{example}\label{example.intro-elliptic-curves}
If $E/\mc{O}_C$ is an elliptic curve with reduction $E_\kappa$, $\kappa=\mc{O}_C/\mf{m}_C$, the admissible pair attached to the degree $1$ cohomology of $E$ is basic if and only if 
\begin{enumerate}
    \item $E_\kappa$ is supersingular, or
    \item $E_\kappa$ is ordinary and, up to isogeny, $E$ is the canonical lift of $E_\kappa$.
\end{enumerate}
\end{example}

\begin{example}\label{example.p-divisible-groups}The category of basic admissible pairs includes as a full subcategory the category of $p$-divisible groups up-to-isogeny over $\mc{O}_C$ equipped with a lifting of the slope decomposition from the special fiber (this is illustrated in \cref{example.intro-elliptic-curves} by passing from $E$ to $E[p^\infty]$ with its slope decomposition). This includes, e.g., all isoclinic $p$-divisible formal groups (the height $n$ one-dimensional case treated in \cite{howe:transcendence} is equivalent to isoclinic of slope $1/n$), and the field of rationality of the admissible pair is equivalent to the field of rationality of the $p$-divisible group up to isogeny. In these cases, the de Rham lattice is equivalent to the Hodge-Tate filtration, and thus \cref{main.transcendence} is indeed a generalization \cref{theorem.transcendence-ht-filt-one-dim}. \end{example}

\begin{remark}\label{remark.splitting-and-serre-result}
For $M \in \AdmPair(\overline{C}_0)$, the Hodge-Tate filtration is canonically split (e.g., by the theory of Hodge-Tate Galois representations). Thus there is a Hodge-Tate grading on the \'{e}tale realization, not just a Hodge-Tate filtration. It is then a short step from results of Sen to see that if the Hodge-Tate grading is defined over $\overline{C}_0$, then $M$ has CM --- this is made explicit, e.g., in \cite[Theorem 3]{serre:groupes-alg-ht}, which computes the transcendence degree of the field of definition of the Hodge-Tate grading as the dimension of the motivic Galois group quotiented by its center.

Our result thus implies that, in the basic case, algebraicity of the Hodge-Tate \emph{grading} is equivalent to algebraicity of the de Rham lattice. Our proof does not proceed directly through this equivalence, though we do make use of the associated Galois representation. We emphasize that, in cases where the Hodge-Tate filtration uniquely determines the de Rham lattice, it is a priori a much weaker condition to require that the Hodge-Tate filtration be defined over $\overline{C}_0$ than to require that the Hodge-Tate grading be defined over $\overline{C}_0$. In particular, this equivalence does not hold outside of the basic case: for example, when $M$ is attached to an ordinary elliptic curve, the Hodge-Tate filtration is always algebraic but the Hodge-Tate grading is only algebraic if the Serre-Tate coordinate is a root of unity (this can be deduced from \cite[Theorem 3]{serre:groupes-alg-ht} --- the condition on $q$ is equivalent to requiring that the associated admissible pair has CM, cf. also \cref{example.intro-elliptic-curves}). 

For our purposes --- especially in Part II --- it is crucial to have a result in terms of the Hodge-Tate filtration or de Rham lattice. Indeed, unlike, the Hodge-Tate grading, these exist in families, making it possible to formulate more general results about special subvarieties in the relative setting (this is closely related to the fact that, for a general admissible pair over $C$, there is a Hodge-Tate filtration but no canonical splitting). 

\end{remark}

\begin{remark}By the above considerations, our analogy between basic admissible pairs / $p$-adic Hodge structures and motives over $\mbb{C}$ / Hodge structures works well for transcendence theory. However, as exhibited by \cref{example.intro-elliptic-curves}, the admissible pairs in the cohomology of an algebraic variety are not always basic. Moreover, even in the basic case, this analogy is \emph{not} good for discussing algebraic cycles. One reason is that, to detect algebraic cycles via cohomology, one needs a global structure on the coefficients analogous to the $\mbb{Q}$-structure on singular cohomology. In our setting, this is possible only on the crystalline side where the natural candidate is the Kottwitz global isocrystal of Scholze's \cite{scholze:icm2018} conjectural cohomology for varieties over $\overline{\mbb{F}}_p$. Assuming this cohomology theory exists, the analog of the Hodge conjecture in this setting would following from a crystalline Tate conjecture for the global isocrystal combined with the variational $p$-adic Hodge conjecture of \cite{emerton:variational-p-adic}.
\end{remark}

\begin{remark} In \cref{theorem.transcendence-with-slope} we give an extension of \cref{main.transcendence} that applies outside of the basic case. In \cite{howe:transcendence} a more general statement was conjectured, but it seems likely that the version given here is actually optimal (see \cref{ss.beyond-basic}). 
\end{remark}

\subsection{Structure groups and exactness of filtrations}
The categories of admissible pairs and $p$-adic Hodge structures are interesting beyond \cref{main.transcendence}. Indeed, the moduli spaces of admissible pairs, which we study in Part II, include the local Shimura varieties of Rapoport-Viehmann \cite{rapoport-viehmann} and Scholze \cite{scholze:berkeley}, and their non-minuscule generalizations due to Scholze \cite{scholze:berkeley}. Here the basic case is classically the most interesting, e.g., in cohomological constructions of Langlands correspondences.

In constructing these moduli spaces, just as in the theory of complex period domains and Shimura varieties, it is natural to consider admissible pairs with $G$-structure for $G/\mbb{Q}_p$ a linear algebraic group. In the complex setting and in most past work in the $p$-adic setting, one takes $G$ reductive. When one is interested in the complex Hodge theory of complex projective varieties, there is good reason for this: the polarization forces the Tannakian structure group to be reductive because it gives rise to a compact mod center real form. There is an analogous construction of an inner form in the $p$-adic case, but unfortunately the necessary compactness can only hold for groups of type $A_n$ so the polarization does not help (see \cref{remark.reductivity-criterion} for a more detailed discussion). In fact, the restrictions on a $p$-adic Hodge structure that would enforce reductivity via this argument are so severe that we could not in good conscience impose them in the development of the theory! Note that this plays a role even for the classical local Shimura varieties with $G$ reductive, as there can be special subvarieties corresponding to non-reductive subgroups (see Part II). 

\cref{main.transcendence} can be reformulated as a statement about period maps for admissible pairs with $G$-structure.  Black-boxing all of the $p$-adic geometry, we now describe this formulation. This will segue us to a subtle point in the theory caused by the existence of non-reductive structure groups and will lead to our second main result. 

Let $G/\mbb{Q}_p$ be a connected linear algebraic group and let $\mc{G}$ be an admissible pair with $G$-structure. Attached to $\mc{G}$ there are two natural invariants --- the first is a $G$-isocrystal, classified by an element in the Kottwitz set $B(G)$ of twisted conjugacy classes in $G(\breve{\mbb{Q}}_p)$, and the second is an element of the double coset space $G(B^+_\dR) \backslash G(B_\dR) / G(B^+_\dR)$ measuring the position of the \'{e}tale lattice relative to the isocrystal (or rather the de Rham lattice which it spans over $B^+_\dR$). If this double coset contains $\mu(t)$ for any $\mu$ in a conjugacy class $[\mu]$ of cocharacters of $G_{\overline{\mbb{Q}}_p}$ and any choice of uniformizer $t$ for $B^+_\dR$, we say $G$ is of type $[\mu]$. In particular, if we fix a $b \in G(\breve{\mbb{Q}}_p)$ and conjugacy class $[\mu]$, there is then a natural infinite level moduli $\mc{M}_{b,[\mu]}$ of admissible pairs with $G$-structure of type $[\mu]$ equipped with a trivialization of the underlying $G$-isocrystal to $b$ and a trivialization of the \'{e}tale fiber functor. There are also moduli $\Gr_{[\mu^{\pm1}]}$ parameterizing $B^+_\dR$-lattices in relative position $[\mu^{\pm 1}]$, and Bialynicki-Birula maps from these to the flag varieties $\Fl_{[\mu^{\mp 1}]}$ parameterizing filtrations of type $[\mu^{\mp 1}]$. Together, we obtain a diagram of period maps
% https://q.uiver.app/#q=WzAsNSxbMSwwLCJcXG1je019X3tiLFtcXG11XX0iXSxbMiwyLCJcXEZsX3tbXFxtdV19Il0sWzAsMiwiXFxGbF97W1xcbXVeey0xfV19Il0sWzIsMSwiXFxHcl97W1xcbXVeey0xfV19Il0sWzAsMSwiXFxHcl97W1xcbXVdfSJdLFswLDEsIlxccGlfe1xcSFR9IiwyXSxbMCwyLCJcXHBpX3tcXEhvZGdlfSJdLFswLDMsIlxccGlfe1xcbWN7TH1fXFxkUn0iXSxbMCw0LCJcXHBpX3tcXG1je0x9X1xcZXR9IiwyXSxbMywxLCJcXEJCIl0sWzQsMiwiXFxCQiIsMl1d
\[\begin{tikzcd}
	& {\mc{M}_{b,[\mu]}} \\
	{\Gr_{[\mu]}} && {\Gr_{[\mu^{-1}]}} \\
	{\Fl_{[\mu^{-1}]}} && {\Fl_{[\mu]}}
	\arrow["{\pi_{\mc{L}_\et}}"', from=1-2, to=2-1]
	\arrow["{\pi_{\mc{L}_\dR}}", from=1-2, to=2-3]
	\arrow["{\pi_{\Hodge}}", from=1-2, to=3-1]
	\arrow["{\pi_{\HT}}"', from=1-2, to=3-3]
	\arrow["\BB"', from=2-1, to=3-1]
	\arrow["\BB", from=2-3, to=3-3]
\end{tikzcd}\]
The period $\pi_{\mc{L}_\et}$ determines the rationality of a $G$-admissible pair, while $\pi_{\mc{L}_\dR}$ is the de Rham lattice period. The map $\BB$ is the Bialynicki-Birula map, which sends a lattice to its induced filtration (the Hodge filtration for the \'{e}tale lattice, and the Hodge-Tate filtration for the de Rham lattice). If $[\mu]$ is minuscule, then $\BB$ is an isomorphism and it suffices to use only the filtration periods $\pi_\Hdg$ and $\pi_\HT$. In this case the diagram is analogous to the following uniformization diagram for a  Shimura variety $\Sh_{G, K}$ attached to a Shimura datum $(G,X)$ and level structure $K$, 
\[\begin{tikzcd}
	& X \times G(\mbb{A}_f)/K &\\
	  {\Sh_{G, K}(\mbb{C}) \cong \bigsqcup_{i \in I} \Gamma_i \backslash X^+} && {\Fl_{[\mu^{-1}]}(\mbb{C})}
	\arrow["\pi", from=1-2, to=2-3]
	\arrow["J"', from=1-2, to=2-1]
\end{tikzcd}\]
where $[\mu]$ is the class of Hodge cocharacters, $\pi$ classifies the Hodge filtration, $I$ is a finite index set, and the $\Gamma_i$ are subgroups of $G(\mbb{Q})$ determined by $K$. 

\begin{remark}\label{remark.cocharacter-filtration}At this point, one might reasonably be confused about the $\mu$'s and $\mu^{-1}$'s. To preserve our sanity, we adopted the simple convention that a cocharacter $\mu$ of $G$ defines a decreasing filtration on any representation of $V$ of $G$ such that $\Fil^p V=\bigoplus_{i \geq p} V[i]$, for $V[i]$ the weight spaces where $\mu$ acts by $z^i$. Then, $\Fl_{[\mu]}$ is the moduli of decreasing filtrations of this type. Thus, $\Fl_{[\mu]}$ agrees on the nose with the flag variety denoted $\Fl^\std_{G,\mu^{-1}}$ in \cite[p. 660]{caraiani-scholze:non-compact} but, as an argument in favor of our choice, note that $\Fl_{[\mu]}$ is also canonically identified with the flag variety denoted $\Fl_{G,\mu}$ in loc cit. Indeed, $\Fl_{G,\mu}$ is defined by attaching to $\mu$ an increasing filtration on any representation $V$ of $G$ such that $\Fil_p V=\bigoplus_{i \leq p} V[-i]$. Under the canonical identification of increasing and decreasing filtrations by negating the indices, this agrees with our definition of the decreasing filtration attached to $\mu$. The acrobatics here arise from the convention that the Hodge cocharacter attached to a complex Hodge structure acts by $z^{-p}$ on the subspace $H^{p,q}$ of type $(p,q)$. Recall though that there are several good reasons for this convention \cite[p. 252]{deligne:interpretation-modulaire}!  
\end{remark}

A point of $\mc{M}_{b,[\mu]}(C)$ is \emph{special} if the Tannakian structure group of the associated $G$-admissible pair is a torus. \cref{main.transcendence} is then nearly equivalent to

\begin{maincorollary}\label{main.G-structure-period-maps} Suppose $b$ is basic and   $x \in \mc{M}_{b,[\mu]}(C)$. If both $\pi_{\mc{L}_\dR}(x) \in \Gr_{[\mu^{-1}]}(\overline{C}_0)$ and $\pi_{\mc{L}_\et}(x) \in \Gr_{[\mu]}(\overline{C}_0)$, then $x$ is special. Conversely, if $x$ is special, then $\pi_{\mc{L}_\dR}(x) \in \Gr_{[\mu^{-1}]}(\overline{\mbb{Q}}_p)$ and $\pi_{\mc{L}_\et}(x) \in \Gr_{[\mu]}(\overline{\breve{\mbb{Q}}_p})$.
\end{maincorollary}

We note that $\BB$ is a surjection on $C$-points but not an isomorphism when $\mu$ is not minuscule; however, it induces a bijection on $\overline{C}_0$-points via the convolution construction used earlier. In Part I there will be no need to define these various spaces and maps --- we will make the $C$-points together with their Galois action explicit in \cref{ss.periods} without any $p$-adic geometry. 

For $G$ reductive, the spaces and maps appearing in this diagram were constructed in \cite{scholze:berkeley}. In this case, the double cosets $G(B^+_\dR)\backslash G(B_\dR) / G(B^+_\dR)$ are exhausted by the Cartan decomposition, so any $G$-admissible pair has a type and these diagrams for varying $[\mu]$ give a complete picture. When $G$ is not reductive, the Cartan decomposition is no longer exhaustive, but we still say $\mc{G}$ is of type $[\mu]$ if it lies in the double coset attached to $[\mu]$ and that it is \emph{good} if this holds for some $[\mu]$. 

In Part II we will generalize the construction of moduli spaces to the non-reductive case. The main difficulty is tied to the other main result of Part I:

\begin{maintheorem}\label{main.exactness}
The following are equivalent for a $G$-admissible pair $\mc{G}$:
\begin{enumerate}
    \item $\mc{G}$ is good (resp. has type $[\mu]$).
    \item The Hodge-Tate filtration for $\mc{G}$ is an exact functor from $\Rep\+ G$ to filtered $C$-vector spaces (resp. of type $[\mu]$).
    \item The Hodge filtration for $\mc{G}$ is an exact functor from $\Rep\+ G$ to filtered $C$-vector spaces (resp. of type $[\mu^{-1}]$).
\end{enumerate}
\end{maintheorem}

\begin{remark}
Lemma 3.25 of \cite{anschutz} is false\footnote{Erratum at \url{https://janschuetz.perso.math.cnrs.fr/cm_bkf_modules_erratum.pdf}}. Indeed, it is equivalent to the claim that the Hodge-Tate filtration is exact on the entire category of admissible pairs, which would contradict this theorem. That the Hodge-Tate filtration is not exact can be seen already in a simple motivating example (see \cref{example.ext-of-trivial}).
\end{remark}

\subsection{Related work, Parts II and III}
The category of $\AdmPair(\overline{C}_0)$ is equivalent to a Fontaine category of admissible $\overline{C}_0$-filtered $\varphi$-modules over $C_0$ (see \cref{remark.cryst-comparison} for related discussion). The Tannakian structure groups in this context have been extensively studied, especially in regards to classification of their possible simple factors --- see \cite{serre:groupes-alg-ht, wintenberger:groupes-alg-assoc}. With some modifications, most of \cref{main.transcendence} and its proof could be phrased entirely in this classical language. However, there are conceptual benefits in organizing things from $C$ down, and this perspective will be indispensable for the relative theory in Parts II and III.  

As noted at the start of the introduction, the category of admissible pairs over $C$ is equivalent to the category of rigidified Breuil-Kisin-Fargues modules. Ansch\"utz \cite{anschutz}, building on observations in \cite{bhatt-morrow-scholze:1}, has studied this category from a different perspective. In particular, he showed it is Tannakian and connected. Ansch\"utz also classified the CM admissible pairs, recovering a description that is essentially equivalent to an earlier result of Serre \cite[Th\'{e}or\`{e}me 6]{serre:groupes-alg-ht} on Hodge-Tate representations --- see \S \ref{ss.cm-adm-pair} for further discussion. The category of admissible pairs as presented here is also implicit to varying degrees throughout the literature on local Shimura varieties (especially in \cite{scholze:berkeley}), but the category of $p$-adic Hodge structures and the emphasis on the structural analogy between Hodge structures/motives over $\mbb{C}$ and $p$-adic Hodge structures/basic admissible pairs is a new contribution in our work.

As far as we are aware, \cref{theorem.transcendence-ht-filt-one-dim} recalled above and Serre's \cite[Th\'{e}or\`{e}me 3]{serre:groupes-alg-ht} computation of the transcendence degree of the extension generated by Hodge-Tate periods (see \cref{remark.splitting-and-serre-result}) are the only previous results on $p$-adic transcendence questions of the nature considered here. In Part II, we prove a bi-analytic Ax-Lindemann theorem for basic local Shimura varieties and their non-minuscule generalizations. This implies, in particular, a bi-analytic characterization of higher dimensional special subvarieties analogous to the characterization of special points in \cref{main.G-structure-period-maps}. It includes the zero dimensional case, thus subsumes \cref{main.transcendence} and \cref{main.G-structure-period-maps}. However, many of the ingredients developed here in Part I will also be used in Part II. This organization separates out the purely algebraic aspects of the argument appearing in Part I from the geometric aspects in Part II, which require a more substantial technical apparatus (the theory of locally spatial diamonds). We hope this separation will make the proofs and formalism more easily understood and motivated, and that the results will be more accessible to a reader whose interest is born from experience in complex transcendence. 

The relative theory introduced in Part II to prove our Ax-Lindemann theorem will only use constant coefficients since the moduli spaces in question include trivializations of the underlying $\mbb{Q}_p$-vector space or isocrystal. To formulate a more robust theory (e.g. so that the objects form a $v$-stack), it is better to allow arbitrary local systems (of $\mbb{Q}_p$-vector spaces for $p$-adic Hodge structures or isocrystals for admissible pairs). This general definition has the downside of allowing for many objects that cannot possibly arise  from geometry --- in fact, this can be seen already over a $p$-adic field and is related to our restriction to algebraically closed residue field in the definition of good reduction here (see \cref{remark.more-general-def-over-p-adic-field}). In Part III we complete our results by setting up a general formalism allowing variation of the coefficients and then cutting out a nice subcategory of variations of $p$-adic Hodge structure and formulating a potential good reduction conjecture for these. 

\subsection{Outline} 
In \cref{s.preliminaries} we recall some results and definitions for Tannakian categories, isocrystals, the twistor line, and the Fargues-Fontaine curve. It can be ignored to begin with and returned to as needed; the main novelty is in some aspects of our parallel treatment of the twistor line and the Fargues-Fontaine curve.

In \cref{s.filtrations-and-lattices}, we first recall some standard results on the relation between filtered and latticed vector spaces in a formulation that is convenient for our purposes. Afterwards, the majority of the section is dedicated to a proof of \cref{theorem.good-equivalence-bilatticed}, which implies \cref{main.exactness} above. The proof relies on a careful study of the behavior of types of lattices under extensions (\cref{thm.specialization-goes-up-when-not-good}) that will also be crucial to establishing good properties of moduli spaces for non-reductive groups in Part II. 

In \cref{s.pahs}, we define our category of $p$-adic Hodge structures and establish its basic properties. We motivate the definition using an approach to real Hodge structures via the twistor line that originates in work of Simpson; we learned of the analogy between the twistor line and the Fargues-Fontaine curve from Laurent Fargues, so that the originality here can be attributed mainly to our na\"{i}vet\'{e} in taking this analogy more literally than might initially seem a good idea! An important point for Part II is our interpretation of the Mumford-Tate group using Hodge-Tate lines in \cref{ss.hodge-tate-lines} (see also \cref{ss.hodge-lines}), analogous to the characterization of Mumford-Tate groups of complex Hodge structures using Hodge tensors. 

In \cref{s.admissible-pairs}, we define the category of admissible pairs, explaining how it falls out naturally from the category of $p$-adic Hodge structures. We show that the category of $p$-adic Hodge structures is equivalent to the subcategory of basic admissible pairs, and then extend many of the results about $p$-adic Hodge structures to admissible pairs. That this category is connected Tannakian is due to Ansch\"utz \cite{anschutz} (building on remarks in \cite{bhatt-morrow-scholze:1}) in the context of rigidifed Breuil-Kisin-Fargues modules, and this and many of the other results can also be argued in essentially the same way as for $p$-adic Hodge structures. In  \cref{ss.admissible-pairs-with-good-reduction} and \cref{ss.galois-representation} we study $\overline{C}_0$-analytic admissible pairs and their associated Galois representations --- the two key points are that $\overline{C}_0$-analytic admissible pairs are automatically good (\cref{theorem.rig-an-ap-is-good}), and that the Galois representation associated to an $\overline{C}_0$-analytic admissible pair has open image in the motivic Galois group (\cref{corollary.galois-rep-dense-open}). We describe the period maps of \cref{main.G-structure-period-maps} in \cref{ss.periods}. 

In \cref{s.transcendence} we prove \cref{main.complex} and \cref{main.transcendence}. The proofs are essentially the same, and that is more or less the point --- since we are working with a toy category of motives that is further removed from algebraic cycles, the work of the earlier sections establishes all of the nice properties analogous to the standard conjectures, the Hodge conjecture, and the Grothendieck period conjecture that are needed to give an unconditional result. We encourage the reader who is primarily interested in \cref{main.transcendence} to turn now to \cref{s.transcendence} and read the short proof of \cref{main.complex} and the surrounding discussion as this will give a firm footing for the path we take in the rest of the paper. We conclude in \cref{ss.beyond-basic} by giving an extension of \cref{main.transcendence} to the non-basic case, \cref{theorem.transcendence-with-slope}, and discussing the precise relation of the present work with \cite[Conjecture 4.1]{howe:transcendence}. 

\subsection{Notation}
Throughout this paper $p > 0$ is a prime number. A non-archimedean field is a field that that is complete with respect to a non-archimedean absolute value such that the residue field is of characteristic $p$. A $p$-adic field is a discretely valued non-archimedean field of characteristic zero whose residue field is perfect, and is strict if the residue field is algebraically closed. 

\subsection{Acknowledgements} During parts of the preparation of this work, Sean Howe was supported by MSRI as a research member at the 2019 program on Derived Algebraic Geometry, by the NSF through grants DMS-1704005 and DMS-2201112, by an AMS-Simons travel grant, and as a visitor at the 2023 Hausdorff Trimester on The Arithmetic of the Langlands Program supported by the Deutsche Forschungsgemeinschaft (DFG, German Research Foundation) under Germany's Excellence Strategy – EXC-2047/1 – 390685813. Christian Klevdal was supported by Samsung Science and Technology Foundation under Project Number SSTF-BA2001-02, and by NSF grant DMS-1840190. We would like to thank Laurent Fargues, Fran\c{c}ois Loeser, and Jared Weinstein for helpful conversations, and Alex Youcis for comments on an earlier draft. 

\section{Preliminaries}\label{s.preliminaries}

\subsection{Tannakian categories}\label{ss.tannakian}

\subsubsection{Notation}\label{sss.tannakian-categories-notation}
All categories considered are $k$-linear for some field $k$ (which should be clear from the context), and all functors are assumed to be $k$-linear. %We primarily follow notation of \cite{deligne-milne:tannakian}. 

A Tannakian category $\mc{C}$ over a field $k$ is a rigid abelian tensor category where the endomorphism ring of the unit object is $k$, and such that there exists an exact tensor functor $\omega \colon \mc{C} \to \Vect(k')$ (called a fiber functor) for $k'$ an extension of $k$. Given a fiber functor $\omega$ as above, there is an affine group scheme $\Aut^\otimes \omega$ over $k'$ whose $R$ points are the tensor automorphisms of $\omega \otimes R$. Further, $\omega$ induces an equivalence of categories $\tilde{\omega} \colon \mc{C}_{k'} \xrightarrow{\sim} \mr{Rep}\+ \Aut^\otimes(\omega)$. The Tannakian category $\mc{C}$ is neutral if there exists a fiber functor $\omega$ over $k$. A Tannakian subcategory $\mc{C}'$ of $\mc{C}$ is a strictly full subcategory that is closed under direct sum, tensor products, duals and subquotients. Given a collection of objects $M$ of $\mc{C}$, we denote by $\langle M \rangle$ the strictly full Tannakian subcategory generated by $M$, i.e.\ the smallest strictly full subcategory closed under direct sum, tensor products, duals and subquotients. 

For $G$ an affine group scheme over $k$, we denote by $\Rep\+ G$ the Tannakian category over $k$ of algebraic representations of $G$ on finite dimensional $k$-vector spaces. It is neutralized by the standard fiber functor $\omega_\std = \omega_{\std,G} \colon \Rep\+ G \to \Vect(k), (V, \rho) \mapsto V$. Any fiber functor $\omega \colon \Rep\+ G \to \Vect(k)$ determines an \'{e}tale $G$-torsor $\mr{Isom}^\otimes(\omega_\std, \omega)$ over $\Spec\+ k$, and hence a cohomology class in $H^1(k, G)$. Two fiber functors $\omega, \omega'$ are isomorphic if and only if they have the same cohomology class. If $f \colon G \to H$ is a morphism of affine group schemes over $k$, then pullback $f^\ast \colon \Rep\+ H \to \Rep\+ G$ is an exact tensor functor that commutes with the standard fiber functors, and Tannakian duality states that conversely any exact tensor functor $F \colon \Rep\+ H \to \Rep\+ G$ such that $\omega_{\std,G} \circ F = \omega_{\std, H}$ is $f^\ast$ for a uniquely determined morphism $f \colon G \to H$. 
%then the map $G \to \Aut^\otimes(\omega_\std), g \mapsto (\rho(g) \colon V \to V)_{(V, \rho) \in \Rep\+G}$ is an isomorphism. 

For $\mc{C}$ a Tannakian category over $k$ with fiber functor $\omega$ over $k'$, we write\footnote{It would be more principled (but cumbersome) to include the choice of fiber functor $\omega$ in the notation; but the choice will be either specified or clear from the context. In particular, when considering vector bundles we often implicity take it to be evaluation at a geometric point.} $G\dash\mc{C}$ for the category (groupoid) of objects in $\mc{C}$ with $G$-structure, i.e.\ the objects being exact tensor functors $Q \colon \Rep\+ G \to \mc{C}$ such that $\omega \circ Q$ is isomorphic to $\omega_\std \otimes k'$, and morphisms are natural transformations of tensor functors. If $f \colon G \to H$ is a morphisms of affine group schemes over $k$, then composition along $\Rep\+ H \xrightarrow{f^\ast} \Rep\+ G$ gives a functor $f_\ast \colon G\dash\mc{C} \to H\dash\mc{C}$. 

\begin{lemma}\label{lemma.G-structure-in-semsimple-categories}
Let $G$ be a connected linear algebraic group over a characteristic 0 field $k$, let $U$ be the unipotent radical of $G$, and let $G = MU$ be a Levi decomposition with inclusion map $s: M \hookrightarrow G$ and projection $\pi: G \rightarrow M$. Then for any semisimple Tannakian category $\mc{C}$ over $k$ with fiber functor over $k'$, and any object $F$ of $G\dash\mc{C}$, $F \cong (s \circ \pi)_\ast F$. In particular, $s_\ast$ is essentially surjective. 
\end{lemma}
\begin{proof}
Let $\mc{T}$ be the functor on $k$-algebras sending $R$ to the set of isomorphisms $F \otimes R \xrightarrow{\sim} (s \circ \pi)_\ast F \otimes R$ that induce the canonical identification $\pi_\ast F \otimes R= \pi_\ast (s \circ \pi)_\ast F \otimes R$. By definition, this is a quasi-torsor for the functor $\mc{U}$ sending $R$ to the automorphisms of $F \otimes R$ that induce the identity on $\pi_\ast F$. We claim $\mc{T}$ is an \'{e}tale torsor, and that $\mc{U}$ is a unipotent group scheme over $k$. Given this claim, we obtain the result since $H^1_\et(\Spec\+ k, \mc{U})$ is trivial (because we have assumed $k$ has characteristic zero, $\mc{U}$ has a filtration by normal subgroups with quotients $\mbb{G}_a$ so that this is reduced to the usual vanishing of $H^1_\et(\Spec\+ k, \mbb{G}_a)$). 

To establish the claim, we first note that, by replacing $\mc{C}$ with $\langle F(\Rep\+ G) \rangle$, we may in particular assume $\mc{C}$ is finitely generated, and thus that $k'$ is a finite extension of $k$.  We write $\omega$ for the given fiber functor over $k'$ on $\mc{C}$ and fix an identification of $\omega \circ F$ with the fiber functor $\omega_\std \otimes k'$ on $\Rep\+G$. Then, writing $H$ for the algebraic group over $k'$ of automorphisms of $\omega \otimes k'$, Tannakian duality gives that $\mc{C}_{k'}=\Rep\+ H$ and that $F \otimes k'$ is induced by a homomorphism $f:H \rightarrow G_{k'}$, which is a closed immersion by \cite[Proposition 2.21]{deligne-milne:tannakian}. It follows that the automorphisms of $F \otimes k'$ are identified with the centralizer of $f$ in $G_{k'}$, and thus $\mc{U}_{k'}$ is identified with the centralizer of $f$ in $U_{k'}$, which is unipotent as a closed subgroup of a unipotent group. From the description of $\mc{C}$ by descent from $k'$ as in, e.g, \cite[p.35]{deligne-milne:tannakian}, we find $\mc{U}$ can be defined via descent data from $\mc{U}_{k'}$, thus it is an algebraic group over $k$ which is unipotent because $\mc{U}_{k'}$ is. Moreover, there is a point of $\mc{T}$ over $k'$ thus $\mc{T}$ is an \'{e}tale torsor:  since $f$ is a closed immersion, $f^{-1}(U_{k'})$ is contained in the unipotent radical of $H$, which (since $\mc{C}$ is semisimple and thus $H$ is reductive) is trivial. It follows from the main theorem of \cite{Mostow.FullyReducible} that $f(H)$ is contained in a Levi subgroup $M'$ of $G_{k'}$ which is conjugate to $M$ by an element $u \in U(k')$. In particular, $\mr{Ad}(u) \circ f = s\circ \pi \circ f$, and the action of $u$ gives the isomorphism $F \otimes k' \cong (s \circ \pi)_\ast F \otimes k'$. 

\end{proof}

\subsubsection{Canonical $G$-structure}\label{ss.canonical-G-structure}
Throughout the paper, we find ourselves frequently talking about the canonical $\Aut^\otimes(\omega)$-structure attached to a neutralized Tannakian category $(\mc{C}, \omega)$ over a field $k$, which we now clarify. Setting $G = \Aut^\otimes(\omega)$, we have the following commutative (on the nose) diagram
% https://q.uiver.app/?q=WzAsMyxbMCwwLCJcXG1je0N9Il0sWzIsMCwiXFxSZXBcXCwgRyJdLFsxLDEsIlxcVmVjdF9rIl0sWzAsMiwiXFxvbWVnYSJdLFswLDEsIlxcdGlsZGV7XFxvbWVnYX0iXSxbMSwyLCJcXG9tZWdhX1xcc3RkIiwyXV0=
\[\begin{tikzcd}
	{\mc{C}} && {\Rep\+ G} \\
	& {\Vect(k)}
	\arrow["\omega", from=1-1, to=2-2]
	\arrow["{\tilde{\omega}}", from=1-1, to=1-3]
	\arrow["{\omega_\std}"', from=1-3, to=2-2]
\end{tikzcd}\]
where $\tilde{\omega}$ is an exact tensor functor and an equivalence of categories. By \cite[Proposition I.4.4.2]{rivano:tannakian}, we can choose a quasi-inverse $Q \colon \Rep\+ G \to \mc{C}$ that is a tensor functor and such that $\tilde{\omega} \circ Q$ (resp. $Q \circ \tilde{\omega}$) is isomorphic to $\Id_{\Rep\+ G}$ (resp. $\Id_{\mc{C}}$) as a tensor functor. In particular, such a $Q$ is an exact tensor functor (it is exact as an equivalence of abelian categories). If we fix such a $Q$ and isomorphism $\tilde{\omega} \circ Q \cong \Id_{\Rep\+ G}$, we obtain an isomorphism of tensor functors
\[ \omega \circ Q = (\omega_\std \circ \tilde{\omega}) \circ Q = \omega_\std \circ (\tilde{\omega} \circ Q) \cong \omega_\std \circ \Id_{\Rep\+ G}= \omega_\std. \]

The choice of a pair consisting of such a $Q$ and an isomorphism $\tilde{\omega} \circ Q = \Id_{\Rep\+ G}$ is unique up to unique isomorphism. By abuse of notation, we will sometimes refer to $Q$ as the canonical $G$-structure and sometimes refer to $Q$ with this identification as the canonical $G$-structure.

\subsubsection{Automorphisms of Tannakian categories}
Suppose $G$ is an affine algebraic group over a perfect field $k$ and let $Z(G)$ denote the center of $G$ (an affine algebraic group over $k$). If $z \in Z(G)(k)$ then for each $V \in \Rep\+ G$, $\rho(z) \in \GL(V)$ is an isomorphism of representations. Hence $z$ determines a tensor automorphism of the identity functor on $\Rep\+ G$, denoted by $z \cdot$. It will be useful  to record the basic observation that $Z(G)$ accounts for all automorphisms of Tannakian categories. 

\begin{lemma}\label{lemma.tannakian-automorphisms}
Suppose $F \colon \Rep\+ G \to \mc{C}$ is a fully faithful exact tensor functor with $\mc{C}$ a Tannakian category over $k$. Then for each tensor automorphism $\alpha \colon F \xrightarrow{\sim} F$, there is a unique $\beta \in Z(G)(k)$ such that $\alpha_V  = F(\beta\cdot)$ for all $V \in \Rep\+ G$. 
\end{lemma}
\begin{proof}
As $F$ is fully faithful, for each $V \in \Rep\+ G$, there is a unique $\beta_V \in \Hom_G(V, V)$ such that $F(\beta_V) = \alpha_V$ and $\beta_W \circ f = f \circ \beta_V$ whenever $f \in \Hom_G(V, W)$. Thus $\beta = \{ \beta_V \}$ defines an element of ${\Aut}^\otimes(\omega_{\std})(k) = G(k)$. Moreover, $\beta g = g \beta$ for all $g \in G(k)$ since $\beta_V \in \Hom_G(V,V)$ and hence $\beta \in Z(G)(k)$. 
\end{proof}

%Suppose $\mc{C}$ is a Tannakian category neutralized over $k$ by fiber functor $\omega$ and $G = \Aut^\otimes(\omega)$. A quasi-inverse to $\tilde{\omega} \colon \mc{C} \to \Rep\+G$ is an exact tensor functor $Q \colon \Rep\+G \to \mc{C}$, and any natural isomorphism $Q$ in particular an element of $G\dash\mc{C}$. Any other quasi-inverse will differ

\subsection{Isocrystals}\label{ss.isocrystals}
We introduce Kottwitz' categories of real and $p$-adic isocrystals. References are \cite{kottwitz:isocrystals} for $p$-adic isocrystals and \cite[Construction 9.3]{scholze:icm2018} and  \cite{kottwitz:bg} for real isocrystals.

\subsubsection{Real isocrystals}\label{sss.real-isocrystals}
Let $\mbb{H}=\mbb{R} + \mbb{R} i + \mbb{R} j + \mbb{R} k$ denote the Hamiltonian quaternions and let $\mbb{C} = \mbb{R} + \mbb{R}i \subseteq \mbb{H}$.  Let $W_{\mbb{R}}=\mbb{C}^\times \sqcup j \mbb{C}^\times$, the Weil group of $\mbb{R}$. 

A real isocrystal is a semilinear representation of $W_{\mbb{R}}$ on a finite dimensional complex vector space whose restriction to $\mbb{C}^\times$ is algebraic. Here the semilinearity is for the natural map $W_{\mbb{R}} \rightarrow \Gal(\mbb{C}/\mbb{R})$ sending $j$ to the non-trivial element (this can be realized as the conjugation action on $\mbb{C} \subseteq \mbb{H}$). 

We write $\Kt_{\mbb{R}}$ for the Kottwitz category of real isocrystals\footnote{In \cite[Construction 9.3]{scholze:icm2018}, $\Kt_{\mbb{R}}$ is described as the category of graded quaternionic vector spaces; this is easily seen to be equivalent to our definition by using the action of $j$ to give the semilinear automorphism of loc. cit.}. The category $\Kt_{\mbb{R}}$ is a semi-simple Tannakian category over $\mbb{R}$ (though it is not neutral). We can describe the simple objects explicitly: for each  $\lambda \in \frac{1}{2}\mbb{Z}$, $\lambda=a/b$ with $b>0$ and $\gcd(a,b)=1$, let
\[ D_\lambda:= \mbb{C}[\pi]/(\pi^{b}-(-1)^a) \]
with $\mbb{C}^\times$ acting by $z^{a}$ and $j$ acting as multiplication by $\pi$ composed with complex conjugation on the coefficients. Then each simple object in $\Kt_{\mbb{R}}$ is isomorphic to $D_\lambda$ for a unique $\lambda$ in $\frac{1}{2}\mbb{Z}$. 

\subsubsection{$p$-adic isocrystals}\label{sss.isocrystals}
Let $\overline{\mbb{F}}_p$ be an algebraic closure of $\mbb{F}_p$ and let $\breve{\mbb{Q}}_p= W(\overline{\mbb{F}}_p)[1/p]$.  We will frequently be given a complete algebraically closed non-archimedean extension $C/\mbb{Q}_p$, in which case we take $\overline{\mbb{F}}_p$ to be the algebraic closure of $\mbb{F}_p$ in $\mc{O}_C/\mf{m}_C$ and identify $\breve{\mbb{Q}}_p$ with the completion of the maximal unramified algebraic extension of $\mbb{Q}_p$ in $C$. 

A $p$-adic isocrystal is a finite dimensional $\breve{\mbb{Q}}_p$-vector space equipped with a semilinear automorphism $\varphi_W \colon W \xrightarrow{\sim} W$, i.e.\ an additive isomorphism such that $\varphi_W(\lambda w) = \varphi_{\breve{\QQ}_p}(\lambda)\varphi_W(w)$ for $\lambda \in \breve{\QQ}_p, w \in W$, where $\varphi_{\breve{\mbb{Q}}_p}$ is the automorphism of $\breve{\mbb{Q}}_p$ induced by the $p$-power Frobenius on $\overline{\mbb{F}}_p$. We write $\Kt_{\mbb{Q}_p}$ for the category of $p$-adic isocrystals. By the Dieudonn\'{e}-Manin classification, the category $\Kt_{\mbb{Q}_p}$ is a semi-simple Tannakian category over $\QQ_p$ (though it is not neutral) and we can describe the simple objects explicitly: for $\lambda \in \mbb{Q}$, $\lambda=a/b$ with $b>0$ and $\gcd(a,b)=1$, let 
\[ D_\lambda := \breve{\mbb{Q}}_p[\pi]/(\pi^b-p^{a}),\]
with Frobenius $\varphi_{D_\lambda}$ corresponding to the semilinear map $\pi \varphi_{\breve{\mbb{Q}}_p}$ (i.e. $\varphi_{\breve{\mbb{Q}_p}}$ on the coefficients of a polynomial in $\pi$ followed by multiplication by $\pi$). Each simple object in $\Kt_{\mbb{Q}_p}$ is isomorphic to one of the form $D_\lambda$ for $\lambda \in \mbb{Q}$.

\subsubsection{Isocrystals with $G$-structure}\label{sss.isocrystals-G-structure}
We recall results of \cite{kottwitz:isocrystals} in this subsection. For convenience we write $\sigma = \varphi_{\breve{\QQ}_p} \colon \breve{\QQ}_p \to \breve{\QQ}_p$ for the lift of the $p$-power Frobenius. Let $G$ be a connected linear algebraic group over $\QQ_p$. If $b \colon \Rep\+ G \to \Kt_{\QQ_p}$ is an exact tensor functor, then composition with $\omega_\Kt: \Kt_{\QQ_p} \to \Vect(\breve{\QQ}_p)$ gives a fiber functor valued in $\breve{\QQ}_p$-vector spaces. The $G_{\breve{\QQ}_p}$-torsor $\Isom^\otimes(\omega_\Kt \circ b, \omega_\std \otimes \breve{\QQ}_p)$ yields an element of $H^1(\breve{\QQ}_p, G)$ which, as $G$ is connected and linear and $\breve{\QQ}_p$ has cohomological dimension 1, is trivial by Steinberg's theorem. Thus we can (and do) fix an identification $\omega_\Kt \circ b  = \omega_{\std} \otimes \breve{\QQ}_p$. For each $V \in \Rep\+ G$, the map $\varphi_{b(V)}\sigma^{-1} \colon V_{\breve{\QQ}_p} \to V_{\breve{\QQ}_p}$ is a $\breve{\QQ}_p$-linear automorphism such that $f  \circ \varphi_{b(V)}\sigma^{-1} = \varphi_{b(W)}\sigma^{-1} \circ f$ for any map $f \colon V \to W$ in $\Rep\+ G$. In short, $\varphi_b \sigma^{-1} \in \Aut^\otimes(\omega_{\std})(\breve{\QQ}_p) = G(\breve{\QQ}_p)$. By a slight abuse of notation, we will let $b$ denote both the element $\varphi_b \sigma^{-1} \in G(\breve{\QQ}_p)$  (so $\varphi_b = b \sigma$) and the tensor functor; any $b \in G(\breve{\QQ}_p)$ determines the corresponding tensor functor 
    \[ b \colon \Rep\+ G \to \Kt_{\QQ_p}, \qquad (V, \rho) \mapsto V_b:=(V_{\breve{\QQ}_p}, \rho(b) \sigma). \]
For $b, b' \in G(\breve{\QQ}_p)$, any isomorphism $b \cong b'$ of $G$-isocrystals is given by an element $g \in G(\breve{\QQ}_p)$ such that $b' = gb\sigma(g)^{-1}$, in which case we say that $b, b'$ are $\sigma$-conjugate. Thus, the groupoid $G\dash\Kt_{\QQ_p}$ is equivalent to the groupoid with objects $b \in G(\breve{\QQ}_p)$, and $\Hom(b,b') = \{ g \in G(\breve{\QQ}_p) \colon b' = gb\sigma(g)^{-1} \}$.  If we write $B(G)$ for the quotient of $G(\breve{\QQ}_p)$ by the relation of $\sigma$-conjugacy then $B(G)$ is the set of isomorphism classes of $G$-isocrystals. As $\Kt_{\QQ_p}$ is semisimple, Lemma \ref{lemma.G-structure-in-semsimple-categories} shows that $B(G) = B(G/U)$ for $U \subseteq G$ the unipotent radical (more precisely, \cref{lemma.G-structure-in-semsimple-categories} gives injectivity of $B(G) \rightarrow B(G/U)$ while surjectivity is immediate since $G(\breve{\mbb{Q}}_p)$ surjects onto $(G/U)(\breve{\mbb{Q}}_p)$). Kottwitz \cite{kottwitz:isocrystals-II} gives a complete description of $B(G)$ for $G$ reductive. 
%If $f \colon G \to H$ is a morphism of connected linear algebraic groups over $\QQ_p$, then restriction to $G$ is an exact tensor functor $\Rep\+H \to \Rep\+G$, which thus induces a functor $f_\ast \colon G\dash\Kt_{\QQ_p} \to H-\Kt_{\QQ_p}$ with $f_\ast b(V) = b(V|_G) = f(b)(V)$ (and hence agrees with $G(\breve{\QQ}_p) \to H(\breve{\QQ}_p)$. 

For $b \in G(\breve{\QQ}_p)$, the $\breve{\QQ}_p$-linear extension of the $G$-isocrystal $b$ is an exact tensor functor from $\Rep\+ G$ to $\QQ$-graded $\breve{\QQ}_p$-vector spaces. Semisimplicity and the description of the simple objects $D_\lambda$ implies it is induced (via Tannakian duality) by a morphism $\nu_b \colon \mbb{D}_{\breve{\QQ}_p} \to G_{\breve{\QQ}_p}$ where $\mbb{D} = \varprojlim_{[n]} \mbb{G}_m$ is the pro-torus with character group $\QQ$. In our notation, $\mbb{D}_{\breve{\QQ}_p}$ acts by $z^\lambda$ on the $D_{\lambda}$ isotypic component. The $G(\breve{\QQ}_p)$-conjugacy class of $\nu_b$ is fixed by $\sigma$ and an invariant of the $\sigma$-conjugacy class of $b$. This leads to the Newton map 
    \[ \ol{\nu} \colon B(G) \to \mc{N}(G) := \left(\Hom(\mbb{D}, G)/G^{\mr{ad}}\right)(\QQ_p). \]
The latter notation means the $\QQ_p$-points of the scheme of conjugacy classes of morphisms $\mbb{D} \to G$, which can in turn be identified with $(X_\ast(T)^+ \otimes \QQ)^{\Gal(\overline{\mbb{Q}_p}/\mbb{Q}_p)}$ once a Borel pair $T \subseteq B \subseteq G_{\ol{\QQ}_p}$ is chosen.

We call a $G$-isocrystal $b$ \emph{basic} if the slope homomorphism $\nu_b$ is central, in which case $\nu_b$ is independent of the choice of $b$ in $[b]$ and is a well-defined element of $\Hom(\mbb{D}, G)$. We write $B(G)_{\basic}$ for the subset of basic isomorphism classes. Note that for $M=G/U$, the identification $B(G)=B(M)$ induces an inclusion $B(G)_{\basic} \subseteq B(M)_{\basic}$, but if the center of $M$ acts non-trivially on $U$ then the inclusion may be strict. 

\begin{example}\label{example.basic-levi-non-basic-G}
If $G$ is the group of upper triangular matrices in $\GL_2$ then, taking $M\leq G$ to be the subgroup of diagonal matrices, $b=\mr{diag}(p,1)$ is in $B(M)_{\basic}$ but is not in $B(G)_{\basic}$.
\end{example}

For $b \colon \Rep\+ G \to \Kt_{\QQ_p}$ a $G$-isocrystal, we write $G_b$ for the functor on $\mbb{Q}_p$-algebras
\[ R \mapsto \Aut^\otimes( b \otimes_{\mbb{Q}_p} R), \]
where the latter term is the group of tensor automorphisms of the $R$-linear extension of $b$. Identifying $b$ with an element of $G(\breve{\QQ}_p)$ as above identifies $G_b$ with the functor of \cite[Proposition 1.12]{rapoport-zink:period-spaces} and hence $G_b$ is represented by a connected linear algebraic group over $\QQ_p$. Moreover there is a canonical closed immersion $G_{b, \breve{\mbb{Q}}_p} \hookrightarrow G_{\breve{\mbb{Q}}_p}$ that identifies $G_{b}(\mbb{Q}_p)$ with the fixed points of $g \mapsto b \sigma(g) b^{-1}$ on $G(\breve{\QQ}_p)$. 

\newcommand{\ad}{\mr{ad}}
When $b$ is basic, the canonical closed immersion $G_{b,\breve{\mbb{Q}}_p} \rightarrow G_{\breve{\mbb{Q}}_p}$ is an isomorphism: one fun way to see this is to note that in this case the adjoint action on $\Lie\+ G_{\breve{\mbb{Q}}_p}$ gives the trivial isocrystal, so it admits a $\mbb{Q}_p$-basis such that elements in a small enough $\mbb{Z}_p$-lattice exponentiate to elements of $G_b(\mbb{Q}_p)$, thus we see that $\dim \Lie\+ G_b \geq \dim \Lie\+ G$ (equality of groups follows since $G$ is, by assumption, a connected linear algebraic group and $G_{b,\breve{\mbb{Q}}_p}$ is a closed subgroup with the same Lie algebra as $G_{\breve{\mbb{Q}}_p}$). In a slightly more useful way, following \cite[\S5.2]{kottwitz:bg}, we may $\sigma$-conjugate to assume $b$ satisfies $(b\sigma)^n=z\sigma^n$ for some positive integer $n$ and central element $z$. Then the image $b_\ad$ of $b$ in the adjoint group $G^\ad(\breve{\mbb{Q}}_p)$ is fixed by $\sigma^n$ so $b_\ad \in G(\mbb{Q}_{p^n})$ and $G_b$ is the associated inner form obtained by this conjugation.

\subsection{The twistor line}\label{ss.twistor-line}
As in \cref{sss.real-isocrystals}, let $\mbb{H}=\mbb{R} + \mbb{R} i + \mbb{R} j + \mbb{R} k$ denote the Hamiltonian quaternions, let $\mbb{C} = \mbb{R} + \mbb{R}i \subseteq \mbb{H}$, and let $W_{\mbb{R}}=\mbb{C}^\times \sqcup j \mbb{C}^\times$, the Weil group of $\mbb{R}$. 

We view $\mbb{H}=\mbb{C} +  j \mbb{C} $ as a complex vector space via right multiplication. We let $W_{\mbb{R}}$ act on $\mbb{H}$ by right multiplication and on $\mbb{C}$ by the natural map $W_{\mbb{R}} \rightarrow \Gal(\mbb{C}/\mbb{R})$. Then, the natural action of $W_{\mbb{R}}$ on $\Cont(\mbb{H}, \mbb{C})$, $(w \cdot f) (x)= w f(x w)$ preserves the space $\mbb{C}[X,Y]$ of polynomial functions on $\mbb{H}$ (where $X(a+jb)=a$ and $Y(a+jb)=b$). Concretely, we have $(z \cdot f)(X,Y) = f(Xz,Yz)$ for $z \in \mbb{C}^\times$ and
\[ (j \cdot f)(X,Y)=\overline{f(-\overline{Y}, \overline{X})}=\overline{f}(-Y, X) \]
where $\overline{f}$ denotes complex conjugation on the coefficients of the polynomial $f$. 

Let $\mbb{C}(k)$ denote the semilinear representation of $W_{\mbb{R}}$ on $\mbb{C}$ where $j$ acts by complex conjugation and $z$ acts by $z^{{2k}}$ (this is the real isocrystal $D_{2k}$ in the notation of \cref{sss.real-isocrystals}). Then, we define a scheme over $\mbb{R}$
\[ \tilde{\mbb{P}}^1 := \mr{Proj}\left(\bigoplus_{k \in \mbb{Z}} \left(\mbb{C}[X,Y] \otimes_{\mbb{C}} \mbb{C}(-k) \right)^{W_{\mbb{R}}}\right) = \mr{Proj} \left( \bigoplus_{k \in \mbb{Z}} \mbb{C}[X,Y]_{2k}^{j=1} \right). \]
The natural bundle $\mc{O}_{\tilde{\mbb{P}}}(1)$ has an $\mbb{R}$-basis of global sections 
\[ S=iXY, T=\frac{1}{2}(iX^2 - iY^2), U=\frac{1}{2}(X^2 + Y^2) \]
and these define a closed embedding $\tilde{\mbb{P}}^1 \hookrightarrow \mbb{P}^2_{\mbb{R}}$ identifying $\tilde{\mbb{P}}^1$ with the vanishing set $S^2 + T^2 + U^2=0$.
In particular, $\tilde{\mbb{P}}^1$ is the Brauer-Severi variety for $\mbb{H}$. Alternatively, we can see this using the inclusion of the graded ring into $\mbb{C}[X,Y]$ to obtain a map $\mbb{P}^1_{\mbb{C}} \rightarrow \tilde{\mbb{P}}^1$ which factors through an isomorphism $\mbb{P}^1_{\mbb{C}}=\tilde{\mbb{P}}^1 \times_{\Spec\+ \mbb{R}} \Spec\+ \mbb{C}.$

We write $\infty_{\mbb{C}}$ for the point $\Spec\+ \mbb{C} \xrightarrow{[1:0]} \mbb{P}^1_{\mbb{C}} \rightarrow \tilde{\mbb{P}}^1$. The action of \[ U(1)=\{a + bi | a^2+b^2=1 \} \subseteq \mbb{C} \subseteq \mbb{H}\]
on $\tilde{\mbb{P}}^1$ induced by left multiplication on $\mbb{H}$ is identified with the action on $\mbb{P}^1_{\mbb{C}}$ by $z \cdot [X:Y]=[z X:  z^{-1} Y]$. This action fixes $\infty_{\mbb{C}}$ and its conjugate $0_{\mbb{C}}$. The points $\infty_{\mbb{C}}$ and $0_{\mbb{C}}$  underly the same closed point in $\tilde{\mbb{P}}^1$, cut out by $S=0$, and this is the unique closed point fixed by $U(1)$. 

\begin{example}
The map $\mbb{P}^1_\mbb{C} \rightarrow \tilde{\mbb{P}}^1$ is finite \'{e}tale thus induces an isomorphism between the formal neighborhood of $[1:0]$ in $\mbb{P}^1_{\mbb{C}}$ and the formal neighborhood of the closed point underlying $\infty_{\mbb{C}}$ in $\tilde{\mbb{P}}^1$. In particular, we may choose as uniformizing parameter $t=Y/X$ on which $U(1)$ acts by $z^{-2}$ and identify the formal completion with $\mbb{C}[[t]]$. Note that the action of $U(1)$ by left multiplication on $\tilde{\mbb{P}}^1$ extends to all of $\mbb{H}^\times$, thus in particular to $W_{\mbb{R}}$; the action of $j^2$ in $W_{\mbb{R}}$ is trivial, and $j$ also fixes this closed point but acts by complex conjugation on the residue field. Note also that the subgroup $W_{\mbb{R}}$ is the full stabilizer in $\mbb{H}^\times$ of this closed point. 
\end{example}

\begin{remark}\label{remark.twistor-analytic-uniformization}
If we consider the full action by $\mbb{H}^\times$ instead of just $U(1)$, then we obtain an identification of the closed points of $\tilde{\mbb{P}}^1$ with $\mbb{H}^\times / W_{\mbb{R}}$ by taking the orbit of the closed point underlying $\infty_{\mbb{C}}$. 

We could also arrive at this construction by observing that the conjugation action of $\mbb{H}$ on itself preserves $\mbb{R}$ and the norm form, thus it preserves the orthogonal complement of $\mbb{R}$, $\mbb{R}i + \mbb{R}j+\mbb{R}k$. It then preserves the norm zero curve in $\mbb{P}^2=\mbb{P}(\mbb{R}i + \mbb{R}j+\mbb{R}k)$. If we write the coefficients as $Si + Tj + Uk$ then this curve is $S^2+T^2+U^2=0$. The action of $\mbb{H}^\times$ is transitive and the stabilizer of the closed point $S=0$ is $W_{\mbb{R}}$ so we again obtain this identification. 

In fact, this can be upgraded to an identification of analytic spaces 
\[ \tilde{\mbb{P}}^{1,\an} = \mbb{H}^\times / W_{\mbb{R}},\]
where $\mbb{H}^\times \subseteq \mbb{H}=\mbb{C}+j\mbb{C}$ gives the complex structure, $\tilde{\mbb{P}}^{1,\an}$ is formed as the quotient of the locally ringed space $\mbb{P}^{1,\an}_{\mbb{C}}$ by the semilinear Galois action as in \cite{huisman:exponential}, and the action of $W_{\mbb{R}}$ on the sheaf of functions is as described above for polynomials. 
\end{remark}

There is a classification theorem for vector bundles on $\tilde{\mbb{P}}^1$ akin to Grothendieck's classification theorem for vector bundles on $\mbb{P}^1_{\mbb{C}}$. For comparison with the $p$-adic case, it is useful to realize this via a functor from the category of real isocrystals $\Kt_{\mbb{R}}$ (see \cref{sss.real-isocrystals}) to vector bundles on $\tilde{\mbb{P}}^1$. Attached to any $W \in \Kt_{\mbb{R}}$, we obtain a vector bundle $\mc{E}(W)$ on $\tilde{\mbb{P}}^1$ as the sheaf associated to the graded module 
\[ \bigoplus_{k \in \mbb{Z}} \left(\mbb{C}[X,Y] \otimes_{\mbb{C}} W(-k) \right)^{W_{\mbb{R}}}\] 
where $W(-k):=W \otimes_{\mbb{C}} \mbb{C}(-k)$. We write $\mc{O}_{\tilde{\mbb{P}}^1}(\lambda):=\mc{E}(D_{-\lambda})$. 

\begin{remark}
The line bundles $\mc{O}_{\tilde{\mbb{P}}^1}(k)$ arising from the $\mr{Proj}$ construction are canonically identified with the bundles $\mc{E}(D_{-k})$ for $k \in \mbb{Z}$, so that the terminology will not cause confusion. By construction they are also identified with the restriction of $\mc{O}_{\mbb{P}^2}(k)$ to $\tilde{\mbb{P}}^1$ realized as the vanishing locus $S^2+ T^2 +U^2=0$. For $\lambda \in \frac{1}{2}\mbb{Z} \backslash \mbb{Z}$, $\mc{O}_{\tilde{\mbb{P}}^1}(\lambda)$ can be identified with the pushforward of $\mc{O}_{\mbb{P}^1_{\mbb{C}}}(2\lambda)$ to $\tilde{\mbb{P}}^1$, whereas if $\lambda \in \ZZ$, the pushforward of $\mc{O}_{\mbb{P}^1_{\mbb{C}}}(2\lambda)$ is identified with $\mc{O}_{\tilde{\mbb{P}}^1}(\lambda)^{\oplus 2}$.
\end{remark}

\begin{theorem}
Every vector bundle on $\tilde{\mbb{P}}^1$ is a direct sum of stable bundles. Any stable bundle has slope $\lambda \in \frac{1}{2}\mbb{Z}$, and for any $\lambda \in \frac{1}{2}\mbb{Z}$, $\mc{O}_{\tilde{\mbb{P}}^1}(\lambda)$ is the unique up to isomorphism stable bundle of slope $\lambda$. The assignment $W \mapsto \mc{E}(W)$ is a faithful and essentially surjective exact tensor functor, and it restricts to an equivalence between the category of $D_\lambda$-isotypic real isocrystals and the category of semistable vector bundles on $\tilde{\mbb{P}}^1$ of slope $-\lambda$. 
\end{theorem}

\begin{remark}
The functor $\mc{E}$ does not give an equivalence between $\Kt_{\mbb{R}}$ and vector bundles on $\tilde{\mbb{P}}^1$ because for vector bundles there are maps going up in slope (e.g. the maps from $\mc{O}_{\tilde{\mbb{P}}^1}$ to $\mc{O}_{\tilde{\mbb{P}}^1}(1)$ given by the global sections $S$, $T$, and $U$). 
\end{remark}

\begin{remark}
By the GAGA theorem of \cite{huisman:exponential}, vector bundles on $\tilde{\mbb{P}}^1$ are equivalent to those on the real analytic variety $\tilde{\mbb{P}}^{1,\an}$. If we use the uniformization of \cref{remark.twistor-analytic-uniformization}, then the functor from $\Kt_{\mbb{R}}$ can be realized by sending $W$ to 
$\mc{O}_{\mbb{H}^\times} \otimes_{\mbb{C}} W$ and then using the diagonal action of $W_{\mbb{R}}$ to descend it to $\tilde{\mbb{P}}^{1,\an}$. 
\end{remark}

\subsection{Period rings and the Fargues-Fontaine curve}\label{ss.fargues-fontaine}

Let $C$ be an algebraically closed non-archimedean extension of $\mbb{Q}_p$. We now recall the construction of some of Fontaine's period rings and the Fargues-Fontaine curve attached to $C$. All material recalled here can be found in \cite{fargues-fontaine:courbes}. After defining some rings, our presentation of the material will mirror our presentation of the twistor line in \cref{ss.twistor-line}.

First, recall that reduction mod $p$ induces a bijection 
\[ \mc{O}_{C^\flat}:=\mc{O}_C^\flat = \lim \left( \mc{O}_{C} \xleftarrow{x^p} \mc{O}_C \xleftarrow{x^p} \cdots \right) =\lim \left(\mc{O}_{C}/p \xleftarrow{x^p} \mc{O}_{C}/p  \xleftarrow{x^p} \cdots\right).\]
The rightmost description equips $\mc{O}_{C^\flat}$ with a ring structure. We write $x \mapsto x^\sharp$ for the multiplicative map to $\mc{O}_C$ given by projection to the first component in the middle description. For $|\cdot|_C$ the absolute value on $C$, the assignment $|x|:=|x^\sharp|_C$ is an absolute value on $\mc{O}_{C^\flat}$. Let $\varpi \in \mc{O}_{C^\flat}$ with $\varpi^\sharp=p$, i.e.  $\varpi = (p, p^{1/p}, \ldots)$.  Then $\mr{Frac}\ \mc{O}_{C^\flat}=\mc{O}_{C^\flat}[1/\varpi]$ and there is a natural identification 
\[ C^\flat := \lim \left(C \xleftarrow{x^p} C \xleftarrow {x^p} \cdots\right) = \mc{O}_{C^\flat}[1/\varpi].\]
The obvious extension of the absolute value makes $C^\flat$ into a non-archimedean field in characteristic $p$ with valuation ring $\mc{O}_{C^\flat}$. The surjection $\mc{O}_{C^\flat} \rightarrow \mc{O}_{C}/p$ induces, by the universal property of the $p$-typical Witt vectors functor $W$, a surjection 
\[ \theta: W(\mc{O}_{C^\flat}) \rightarrow \mc{O}_C. \]
We have $\theta([x])=x^\sharp$, where $[x]$ denotes the multiplicative lift. The kernel of $\theta$ is principal (generated, e.g., by $[\varpi] - p$). Let $A_\crys$ be the $p$-complete divided power envelope of $\ker \theta$ and let $B^+_\crys=A_\crys[1/p]$. If we fix a compatible system $(\zeta_{p^n})$ of $p$-power roots of unity in $C$, we can define $\epsilon=(1, \zeta_p, \zeta_{p^2},\ldots) \in \mc{O}_C^\flat$ and then the formal power series defining $\log([\epsilon])=\log(1+([\epsilon]-1))$ converges in $A_\crys$ to an element\footnote{Commonly referred to as the $p$-adic $2\pi i$ because it is the period describing the comparison between de Rham and \'{e}tale cohomology for $\mbb{G}_m$.} $t$. We write $B^+_\dR$ for the completion of $B^+_\crys$ along $\ker \theta$; $t$ is a generator for the kernel of the natural extension $\theta:B^+_\dR \rightarrow C$. We write $B_\crys= B^+_\crys[1/t]$, $B_\dR=B^+_\dR[1/t]$. The Frobenius on $W(\mc{O}_{C^\flat})$ extends to an endomorphism $\varphi$ on $A_\crys, B^+_\crys,$ and $B_\crys$. If $C=\overline{K}^\wedge$ for $K$ a $p$-adic field, then $\Gal(\overline{K}/K)$ acts on everything in sight.

The inclusion $\ol{\mbb{F}}_p \subseteq \mc{O}_C/\mf{m}_C$ lifts canonically to an inclusion of $\overline{\mbb{F}}_p$ into $\mc{O}_C/p$ and thus into $\mc{O}_{C^\flat}$ since $\ol{\mbb{F}}_p$ is perfect. Thus $B^+_\crys$ is canonically a $\breve{\mbb{Q}}_p = W(\ol{\mbb{F}}_p)[1/p]$-algebra compatibly with the Frobenius lifts. We write $\breve{\mbb{Q}}_p(k)$ for $\breve{\mbb{Q}}_p$ equipped with the semilinear automorphism $\varphi_{\breve{\mbb{Q}}_p(k)}=p^{-k}\varphi_{\breve{\mbb{Q}}_p}$ (the isocrystal $D_{-k}$ in the notation of \cref{sss.isocrystals}). 
 
\begin{definition}
The (algebraic) Fargues-Fontaine curve of $C^\flat$ is defined by 
\[ \FF = \FF_{C^\flat} = \mr{Proj}\left(\bigoplus_{k \in \mbb{Z}} \left(B^+_\crys \otimes_{\breve{\mbb{Q}}_p} \breve{\mbb{Q}}_p(k) \right)^{\varphi=1}\right)= \mr{Proj}\left(\bigoplus_{k \in \mbb{Z}} (B^+_\crys)^{\varphi=p^k}\right) \]
The point $\infty_C \in \FF_{C^\flat}$ is induced by the surjection $\theta: B^+_\crys \rightarrow C$, and this induces an identification of $B^+_\dR$ with the completed local ring at $\infty_C$. 
\end{definition}

\begin{remark}\label{remark.analytic-fargues-fontaine}
The analytic Fargues-Fontaine curve $\FF^\an = \FF^\an_{C^\flat}$ of $C^\flat$ is the adic space over $\Spa(\mbb{Q}_p, \mbb{Z}_p)$ formed by taking the quotient of
\[ Y:= (\Spa(W(\mc{O}_{C^\flat}), W(\mc{O}_{C^\flat})) \backslash V([\varpi]p) \]
by the (properly discontinous) action of the Frobenius $\varphi$. 
\end{remark}

There is a natural functor $D \mapsto \mc{E}(D)$ from $\Kt_{\QQ_p}$ to vector bundles on $\FF$ sending $D$ to the sheaf attached to the graded module 
\[ \bigoplus_{k \in \mbb{Z}} (B^+_\crys \otimes_{\breve{\mbb{Q}}_p} D(k))^{\varphi=1} = \bigoplus_{k \in \mbb{Z}} (B^+_\crys \otimes_{\breve{\mbb{Q}}_p} D)^{\varphi=p^k}. \]
(The $\varphi$ in the superscript is the diagonal action). For $\lambda \in \mbb{Q}$ we write $\mc{O}_\FF(\lambda):=\mc{E}(D_{-\lambda})$, where $D_{-\lambda}$ is as in \cref{sss.isocrystals}. 

\begin{theorem}
Every vector bundle on $\FF$ is a direct sum of stable bundles, and for any $\lambda \in \mbb{Q}$, $\mc{O}_\FF(\lambda)$ is the unique up to isomorphism stable bundle of slope $\lambda$.  The assignment $W \mapsto \mc{E}(W)$ is a faithful and essentially surjective exact tensor functor, and it restricts to an equivalence between the category of $D_{-\lambda}$-isotypic isocrystals and the category of semistable vector bundles on $\FF$ of slope $\lambda$. 
\end{theorem}

\begin{example} $\mc{O}_\FF(0)=\mc{E}(D_0)=\mc{O}_\FF.$ More generally, there is an obvious functor from $\mbb{Q}_p$-vector spaces to isocrystals inducing an equivalence with the category of $0$-isotypic isocrystals, and an inverse functor is given by taking $\varphi_W$-invariants. Composed with $\mc{E}$, we obtain an equivalence between $\mbb{Q}_p$-vector spaces and vector bundles on $\FF$ that are semistable of slope zero. This functor is naturally isomorphic to $V \mapsto V \otimes_{\mbb{Q}_p} \mc{O}_\FF$, and the inverse functor is given by global sections. 
\end{example} 

\begin{remark}\label{remark.fargues-fontaine-GAGA-vector-bundles}
The functor $\mc{E}$ has a simple interpretation on $\FF^\an$. Indeed, the vector bundle $W \otimes_{\breve{\mbb{Q}}_p} \mc{O}_{Y}$ on the cover $Y$ of \cref{remark.analytic-fargues-fontaine} 
admits descent data to $\FF^\an$ via the diagonal action of Frobenius. Denote by $\mc{E}^\an$ the resulting functor to vector bundles on $\FF^\an$, and let $\mc{O}_{\FF^\an}(\lambda)=\mc{E}^\an(D_{-\lambda}).$ Then for $k \in \mbb{Z}$ one finds 
\[ H^0(\FF^\an, \mc{O}(k)) = (B^+_\crys)^{\varphi=p^k},\]
and using this one obtains a map of ringed spaces $\FF^\an \rightarrow \FF$ that induces a GAGA equivalence of vector bundles identifying $\mc{E}^\an$ with $\mc{E}$.
\end{remark}

\subsubsection{$G$-bundles on the Fargues-Fontaine curve}
Let $G$ be a connected linear algebraic group over $\QQ_p$. A $G$-bundle on $\FF$ is an exact tensor functor from $\Rep\+ G$ to the category of vector bundles on $\FF$. If $b$ is a $G$-isocrystal, then the composition of $b$ with the functor $\mc{E}$ above yields a $G$-bundle $\mc{E}_b$ on $\FF$. If $G$ is reductive, then every $G$-bundle on $\FF$ is isomorphic to $\mc{E}_b$ for a $G$-isocrystal $b$ \cite{fargues:torsors}. 

\subsection{Modifications of vector bundles}\label{ss.modifications}
Let $X$ be a connected scheme that is locally the spectrum of a PID, $\infty \in |X|$ a closed point, $j \colon X \setminus \infty \hookrightarrow X$, and $\mc{E}$ a vector bundle. By assumption, the completed local ring $\hat{\mc{O}}_{X,\infty}$ is a DVR, and we denote by $\hat{\mc{E}}_\infty$ the sections of $\mc{E}$ over $\Spec\+ \hat{\mc{O}}_{X,\infty}$, a finite rank free $\hat{\mc{O}}_{X,\infty}$-module. Now suppose $\mc{L} \subset \hat{\mc{E}}_\infty \otimes \mr{Frac}(\hat{\mc{O}}_{X,\infty})$ is a $\hat{\mc{O}}_{X,\infty}$-lattice. By Beauville-Laszlo \cite{beauville-laszlo}, there is a unique vector bundle $\mc{E}_{\mc{L}}$ with $(\mc{E}_\mc{L})|_U = \mc{E}|_U$ and $\widehat{(\mc{E}_\mc{L})}_\infty = \mc{L}$. Explicitly, $\mc{E}_\mc{L} \subset j_\ast \mc{E}|_U$ has sections
    \[ \mc{E}_\mc{L}(V) = \{ s \in \mc{E}(V - \infty) \colon s|_{\mr{Frac}(\hat{\mc{O}}_{X,\infty})} \in \mc{L} \} . \]

\begin{definition}\label{def.modification-vb}
The vector bundle $\mc{E}_\mc{L}$ above is the modification of the vector bundle $\mc{E}$ by the lattice $\mc{L}$ at $\infty$. 
\end{definition}

\section{Filtrations and lattices}\label{s.filtrations-and-lattices}
Let $L$ be an algebraically closed field of characteristic zero and let $B^+$ be a complete discretely valued $L$-algebra with algebraically closed residue field $C$. We write $\theta: B^+ \rightarrow C$ for the quotient, which equips $C$ with the structure of an $L$-algebra such that $\theta$ is a map of $L$-algebras. Let $B$ denote the fraction field of $B^+$, and let $F^\bullet B$ denote the valuation filtration on $B$. For convenience, we fix a generator $t$ for the maximal ideal in $B^+$ (so, in particular, $F^i B= t^i B^+$). 

\begin{example}\label{example.filtration-lattice-situations} \hfill
\begin{enumerate}
    \item We could take $B^+=\mbb{C}[[t]]$, $B=\mbb{C}((t))$, $L=C=\mbb{C}$.
    \item If $C$ is an algebraically closed non-archimedean extension of $\mbb{Q}_p$, let $\kappa=\mc{O}_C/\mf{m}_C$ be the residue field, let $C_0=W(\kappa)[1/p]$ be the maximally absolutely unramified subextension, let $L=\overline{C}_0$, the algebraic closure of $C_0$ in $C$, and let $B^+=B^+_\dR$, $B=B_\dR$. By construction (see \cref{ss.fargues-fontaine}), $B^+_\dR$ is a $C_0$-algebra compatibly with the canonical map $\theta:B^+_\dR \rightarrow C$, and Hensel's lemma extends this to a map of $\overline{C}_0$-algebras. Note that any $p$-adic field contained in $C$ is automatically contained in $\overline{C}_0$. In this case it is convenient to take $t$ to be the ``$p$-adic $2\pi i$", whose definition depends on the choice of a compatible system of $p$-power roots of unity in $C$ thus is well-defined only up to multiplication by an element of $\mbb{Z}_p^\times$. 
    \item By the Cohen structure theorem \cite[\href{https://stacks.math.columbia.edu/tag/0C0S}{Tag 0C0S}]{stacks-project}, we can always choose a section $C \rightarrow B^+$ of $L$-algebras so that $B^+=C[[t]]$ and $B=C((t))$, but such a section is non-canonical in general and may not respect other structure --- in particular, if $K$ is a $p$-adic field and $C=\overline{K}^\wedge$, then in the setting of (2) such a section cannot be chosen to be continuous for the natural topology on $B^+_\dR$ or to respect the natural actions of $\Gal(\overline{K}/K)$.
\end{enumerate}
\end{example}

In this section, we study the relation between $B^+$-latticed $L$-vector spaces and $C$-filtered $L$-vector spaces. In the rest of the paper we will apply this study principally in case (2) above to understand the relation between the de Rham and \'{e}tale lattices and the Hodge and Hodge-Tate filtrations in $p$-adic Hodge theory. We will also consider case (1) in our motivational description of real Hodge structures. 

In \cref{ss.from-filtrations-to-lattices-and-back} we recall some basic constructions relating $C$-filtered and $B^+$-latticed $L$-vector spaces. In the sections following, we shift perspectives slightly and explore a more delicate relation between bilatticed $B$-vector spaces and filtered $C$-vector spaces that plays an important role in describing the elementary invariants of $p$-adic Hodge structures and admissible pairs. The main result is \cref{theorem.good-equivalence-bilatticed}, which will immediately imply \cref{main.exactness} after the relevant definitions are in place. 

\subsection{First relations}\label{ss.from-filtrations-to-lattices-and-back}
\newcommand{\can}{\mr{can}}

We will consider the following categories: 
\begin{definition}\label{def.filt-latt-vector-spaces}\hfill
\begin{enumerate}
    \item For $K$ any field, $\Vect^{f}(K)$ is the category of filtered $K$-vector spaces, i.e. pairs $(V, F^\bullet V)$ where $V$ is a finite dimensional $K$-vector space and $F^\bullet V$ is a separated and exhaustive decreasing filtration of $V$. A morphism   
    \[ (V, F^\bullet V) \rightarrow (V', F^\bullet V') \]
    is a morphism of $K$-vector spaces $f: V \rightarrow V'$ such that $f(F^k V) \subseteq F^k V'$ for all $k \in \mbb{Z}$. It is strict if $f(F^k V)=f(V) \cap F^k V'$ for all $k \in \mbb{Z}$.
    \item With $L, B^+, B$ as above, $\Vect^{B^+\dash\mr{latticed}}(L)$ is the category of $B^+$-latticed $L$-vector spaces consisting of pairs $(V,\mc{L})$ where $V$ is a finite dimensional $L$-vector space and $\mc{L} \subseteq V_B$ is a $B^+$-lattice, i.e.\ a $B^+$-submodule $\mc{L} \subset V_B$ such that $\mc{L} \otimes_{B^+} B \to V_B$ is an isomorphism. A morphism
    \[ (V, \mc{L}) \rightarrow (V', \mc{L}') \]
    is a morphism of $L$-vector spaces  $f: V \rightarrow V'$ such that $f(\mc{L}) \subseteq \mc{L}'$. A morphism is \emph{strict} if $f(\mc{L})=f(V)_B \cap \mc{L}'$. 
\end{enumerate}
\end{definition}

These are rigid tensor categories with the usual tensor products and duals of vector spaces / modules / filtered vector spaces. In each category, a complex is exact if it is exact as a complex of vector spaces and each morphism is strict --- in the filtered categories, this is equivalent to requiring that the complex obtained by applying the associated graded functor is exact. None of these exact categories is abelian --- indeed, in all cases there are morphisms which are not strict, while strictness is precisely the condition that the natural quotient filtration/lattice on the coimage agree with the natural submodule filtration/lattice on the image. 

\begin{remark}
In later sections we also consider variants of these categories where we fix some additional structure such as an underlying $\mbb{Q}_p$-vector space or isocrystal. 
\end{remark}

There are natural functors between latticed and filtered vector spaces:
\begin{enumerate}
    \item We denote by $\BB: \Vect^{B^+\dash\mr{latticed}}(L) \rightarrow \Vect^{f}(C)$ the Bialynicki-Birula functor 
\[ (V, \mc{L}) \mapsto (V_C, \trFil^\bullet_{\mc{L}} V_C) \]
where $\trFil^k_{\mc{L}} V_C$ is the image of $(\Fil^k B \cdot \mc{L}) \cap V_{B^+}$ in $V_C=V_{B^+} \otimes_{B^+} C$. We call $\trFil^\bullet_\mc{L} V_C$ the trace filtration of $\mc{L}$ on $V_C$.
    \item We denote by $\mc{L}_\can: \Vect^{f}(L) \rightarrow \Vect^{B^+\dash\mr{latticed}}(L)$ the canonical lattice functor 
\[ (V, F^\bullet V_{L}) \mapsto (V, \mc{L}_{F^\bullet V_{L}}) \]
where $\mc{L}_{F^\bullet V_{L}}:=\sum_i \Fil^{-i}B \cdot \Fil^i V_{L} \subseteq V_B$. 
\end{enumerate}

The following lemma is elementary: 
\begin{lemma}\label{lemma.relation-between-BB-and-canonical-filtration}
Both $\BB$ and $\mc{L}_\can$ are tensor functors, $\mc{L}_\can$ is fully faithful and exact, and $\BB \circ \mc{L}_\can $ is naturally identified with extension of scalars
\[ \Vect^{f}(L) \rightarrow \Vect^{f}(C),\; (V, F^\bullet V_{L}) \mapsto (V_C, (F^\bullet V_{L})_C) \]
\end{lemma}

\begin{example}\label{example.latticed-K-vect-BB-not-exact} Unlike $\mc{L}_\can$, the functor $\BB$ is not exact: consider the exact sequence in $\Vect^{B^+\dash\mr{latticed}}(L)$ (recall that $t$ is a generator for $\Fil^1 B$)
\[ 0 \rightarrow (L, B^+) \xrightarrow{\cdot e_1}  (L^2, B^+ e_1 + B^+ (e_2 + \frac{1}{t} e_1) )\rightarrow (L,  B^+) \rightarrow 0. \]
Applying $\BB$ produces a sequence of filtered vector spaces that is exact as a sequence of vector spaces but such that the morphisms are not strict. Indeed, taking the zeroth graded part after applying $\BB$ gives the sequence
\[ 0 \rightarrow C \rightarrow 0 \rightarrow C \rightarrow 0. \]
\end{example}

In certain situations a filtration is uniquely determined by the corresponding lattice. We will recall a version of this with $G$-structure.

\begin{definition}\label{def.filt-lattice-G-torsor}
Let $K$ be a subfield of $L$ and let $G$ be a connected linear algebraic group over $K$. 
\begin{enumerate}
    \item A \emph{filtration} on a $G$-bundle\footnote{Recall that we take the Tannakian viewpoint, i.e.\ $\omega$ is an exact $K$-linear $\otimes$-functor.} $\omega \colon \Rep\+ G \to \Vect(M)$ (for $M = K$ or $C$) is an exact tensor functor $\Fil^\bullet \omega: \Rep\+ G \rightarrow \Vect^f(M)$  and an identification $\omega \xrightarrow{\sim} \Forget \circ \Fil^\bullet \omega$ where $\Forget$ is the forgetful functor $\Vect^f(M)\rightarrow \Vect(M)$.
    \item  A \emph{lattice} on a $G$-bundle $\omega \colon \Rep\+ G \to \Vect(K)$ is an exact tensor functor $\omega_{\mc{L}}: \Rep\+ G \rightarrow \Vect^{B^+\dash\mr{latticed}}(K)$ and an identification $\omega \xrightarrow{\sim} \Forget \circ \omega_{\mc{L}}$ where $\Forget$ is the forgetful functor $\Vect^{B^+\dash\mr{latticed}}(K) \rightarrow \Vect(K)$. 
\end{enumerate}
Finally, if $\omega$ is a $G$-bundle over $K$, and $K'$ is an extension of $K$ contained in $C$ (resp. $L$) then a $K'$-filtration (resp. lattice) on $\omega$ is a filtration (resp. lattice) on $\omega \otimes K'$. 
\end{definition}
\begin{remark}
   Our primary interest is when $\omega$ is the trivial $G$-bundle, i.e.\ $\omega = \omega_\std$ for the standard fiber functor on $\Rep\+ G$. 
\end{remark}

Filtrations on the trivial $G$-bundle are parameterized by the points of the flag variety $\Fl_G$. Indeed, any filtration on $\omega_\std$ is split by  \cite[Th\'eor\`eme 2.4 of Chapitre 4]{rivano:tannakian}, thus isomorphic to the filtration attached to a cocharacter $\mu$ of $G$, 
\[ F^i_\mu(V, \rho)=\bigoplus_{j \geq i} V[j], \]
where $V[j]$ is weight $j$ component, i.e. the isotypic component for the character $z \mapsto z^j$ of $\mbb{G}_m$ under the action $\rho \circ \mu$. Two filtrations $F_\mu$ and $F_{\mu'}$ are isomorphic if and only if $\mu$ and $\mu'$ are conjugate, so that this conjugacy class $[\mu]$ is a natural invariant of a filtered $G$-bundle and $\Fl_{G,C}=\bigsqcup_{[\mu]}{\Fl_{[\mu],C}}$ where $[\mu]$ runs over all conjugacy classes of cocharacters of $G_C$ and $\Fl_{[\mu]}$, which is naturally defined over the field of definition of $[\mu]$, parameterizes filtrations of the trivial $G$-bundle of type $[\mu]$ --- the action of $G$ on the point $F_\mu \in \Fl_{[\mu]}$ induces, after base change to the field of definition of $\mu$, an isomorphism $\Fl_{[\mu]}=G/P_\mu$ where $P_\mu$ is the stabilizer of $F_\mu$. 

Lattices on the trivial $G$-bundle are parameterized by the coset 
\[ \Gr_G(C):= G(B)/G(B^+). \]
In the $p$-adic setting these will be the $C$-points of a $B^+_\dR$-affine Grassmannian --- for now, however, it suffices to treat this as notation. To see this classification, choose a trivialization $\triv_{\mc{L}}: \omega_\std \otimes B^+ \xrightarrow{\sim} \mc{L} \circ \omega_{\mc{L}}$ where $\mc{L}$ denotes the forgetful functor from $\Vect^{B^+\dash\mr{latticed}}(K)$ to $B^+$-modules (this is possible since $B^+$ is a complete DVR with algebraically closed residue field $C$). The composition of $\triv_{\mc{L}} \otimes_{B^+} B$ with the extension of scalars to $B$ of the fixed $\omega_\std \xrightarrow{\sim} \Forget \circ \omega_{\mathcal{L}}$ gives an element 
$g  \in \Aut^\otimes(\omega_{\std, B}) = G(B)$. Any other choice of trivialization of $\mc{L} \circ \omega_{\mc{L}}$ differs from $\triv$ by an element of $G(B^+)$, so the coset $gG(B^+) \in G(B)/G(B^+)$ only depends on the lattice $\omega_{\mc{L}}$. In the opposite direction, attached to a coset $g G(B^+)$, we consider the lattice on the trivial $G$-bundle
\[ \omega_{g}: (V, \rho) \mapsto (V, \rho(g) \cdot V_{B^+}) \]
and the above argument shows every lattice on the trivial $G$-bundle is isomorphic to one of this form. If $G$ is reductive, the Cartan decomposition gives
\[ \Gr_G(C) = \bigsqcup \Gr_{[\mu]}(C),\; \Gr_{[\mu]}(C):=G(B^+)t^\mu G(B^+)/G(B^+) \subseteq G(B)/G(B^+)\]
where $[\mu]$ runs over all conjugacy classes of cocharacters of $G_L$ (since $L$ is algebraically closed, these are the same as the conjugacy classes of $G_C$), and we note that the open Schubert cell $\Gr_{[\mu]}(C)$ is independent of the choice of a representative $\mu$ and of the choice of uniformizer $t$. We caution that this decomposition is only disjoint at the level of points --- once the affine Grassmannian is equipped with a geometric structure there are closure relations determined by the Bruhat order. 

If $G$ is not reductive, we still have $\bigsqcup_{[\mu]} \Gr_{[\mu]}(C) \subseteq \Gr_G(C)$, but this subset is no longer exhaustive. This will be studied further in \cref{ss.bilatticed-and-exactness}.

In any case, it is an immediate computation that if $g=g^+ t^\mu G(B^+) \in \Gr_{[\mu]}$ for $g^+ \in G(B^+_\dR)$ then $\BB \circ \omega_g = \overline{g^+} \cdot F_{\mu^{-1}}$ where $\overline{\bullet}$ denotes reduction $G(B^+)\rightarrow G(C)$. For each $[\mu]$ we thus have the following $G(L)$-equivariant commutative diagram

\[\begin{tikzcd}
	& { \Gr_{[\mu]}(C)} & \Gr_G(C) \\
	{\Fl_{[\mu^{-1}]}(L)} && {\Fl_{[\mu^{-1}]}(C)}
	\arrow[hook, from=1-2, to=1-3]
	\arrow["\BB"{description}, from=1-2, to=2-3]
	\arrow[hook, from=2-1, to=2-3]
	\arrow["{\mc{L}_\can}"{description}, from=2-1, to=1-2]
\end{tikzcd}\]

\begin{proposition}[Criteria for the filtration to determine the lattice]\label{prop.filt-det-latt-criteria} \hfill
\begin{enumerate}
    \item If the weights of $[\mu]$ in the adjoint action on $Lie G$ are $\leq 1$, then $\BB: \Gr_{[\mu]}(C) \rightarrow \Fl_{[\mu^{-1}]}(C)$ is a bijection (when $G$ is reductive, this is equivalent to requiring the weights lie in $\{-1, 0, 1\}$, i.e. that $[\mu]$ is minuscule). 
    \item Suppose $\mf{G}$ is a group acting on $B^+$ preserving the maximal ideal, that $K=C^\mf{G} \subseteq L$ for the induced action on $C$, and that $\mf{G}$ acts on $K \cdot t$ by an infinite order character. If $G$ and $[\mu]$ are defined over $K$, then $\Fl_{[\mu^{-1}]}(K)$ is identified with the $\mf{G}$-invariants in $\Gr_{[\mu]}(C)$ (for the natural action of $\mf{G}$ on $\Gr_{G,[\mu]}(C) \subseteq G(B)/G(B^+)$ induced by the action on $G(B)$). 
\end{enumerate}
\end{proposition}
\begin{proof}
We first treat (1). For general $\mu$, the stabilizer of $t^\mu G(B^+)$ for the left multiplication action of $G(B^+)$ is $t^{\mu}G(B^+)t^{-\mu} \cap G(B^+)$. If the weights of $\mu$ in the adjoint representation are all $\leq 1$, then this is exactly the pre-image under reduction to $G(C)$ of the $C$-points of the connected subgroup whose Lie algebra is $F_{\mu^{-1}}^0\Lie\+ G$ (i.e. the sum of the weight $\leq 0$ eigenspaces for $\mu$), and this connected subgroup is also the stabilizer of the filtration $F_{\mu^{-1}}^\bullet \omega_\std$. Since the map from $\Gr_{[\mu]}(C)$ to $\Fl_{[\mu^{-1}]}(C)$ is equivariant and the action of $G(B^+)$ is transitive on each, we conclude that it is a bijection.

For the second condition, it suffices to show that the canonical filtration induces an equivalence between $\Vect^f(K)$ and the category of $B^+$-latticed $K$-vector spaces $(V, \mc{L})$ such that the lattice is invariant under the semi-linear action of $\mf{G}$ on $V_B$. To see this, first note that a lattice invariant under $\mf{G}$ must come from a $K[[t]]$-lattice in $V_{K((t))}$: indeed, by multiplying by a power of $t$, we may assume $V_{B^+} \supseteq \mc{L} \supseteq t^{n}V_{B^+}$, then, arguing by induction on $n$, the $\mf{G}$-invariant lattice $\mc{L}'=\mc{L} \cap t V_{B^+}$ has a basis in $t V_{K[[t]]}$. If we write $W_C$ for the image of $\mc{L}$ in $V_C$, it is invariant under $\mf{G}$, thus of the form $W_C = W \otimes_K C$ for $W$ a subspace of $V$. Since the image of $\mc{L}$ in $V_C$ is contained in the image of $\frac{1}{t} \mc{L}'$, we may arrange our basis $e_1, \ldots, e_n$ for $\mc{L}'$ in $t V_{K[[t]]}$ so that for $d=\dim W$, the image of $\frac{1}{t} e_1, \ldots, \frac{1}{t} e_d$ spans $W$. Then $\frac{1}{t} e_1, \ldots, \frac{1}{t} e_d, e_{d+1}, \ldots, e_n$ are a basis for $\mc{L}$ in $V_{K[[t]]}$. 

With this established, the equivalence is a standard result about the Rees construction for $K((t))$, using that the image of $\mf{G}$ acting on $t$ is Zariski dense in $\mbb{G}_m$ by assumption. 
\end{proof}

\begin{example}
If $G=\GL_n$, then a lattice on the trivial $G$-bundle is equivalent to a lattice $\mc{L}$ in $B_\dR^n$ and a filtration on the trivial $G$-bundle is equivalent to a filtration of $C^n$. In the first case, to be minuscule means that there is an $i \in \mbb{Z}$ such that $\Fil^{i+1} B \cdot \mc{L} \subset {B^+}^n \subset \Fil^i B \cdot \mc{L}.$ In the second case, to be minuscule means that there an $i \in \mbb{Z}$ such that $\gr^{j} C^n=0$ for $j \neq i, i+1$. 
\end{example}

\begin{corollary}\label{corollary.filtration-trivial-implies-trivial}
A $B^+$-lattice on a vector space is trivial (i.e. $\mc{L}_\dR=V_{B^+_\dR}$) if and only if the induced filtration is trivial. 
\end{corollary}
\begin{proof}
This is a very special case of \cref{prop.filt-det-latt-criteria}-(1).
\end{proof}

\begin{example}\label{example.galois-action-lattice}
In the setting of \cref{example.filtration-lattice-situations}-(2), suppose $K \subseteq C$ is a $p$-adic field, $C=\overline{K}^\wedge$, and $\mf{G}=\Gal(\overline{K}/K)$. Then \cref{prop.filt-det-latt-criteria}-(2) applies. It also applies if $K = C = \mbb{C}$ and $\mf{G} = U(1)(\mbb{R})$ acting on $t$ by $z \mapsto z^i$ for $i \neq 0$. 
\end{example}

At this point it is natural to ask: for $G$ not reductive, if we have a lattice $\omega_{\mc{L}}$ on the trivial $G$-bundle whose classifying point does not lie in $\Gr_{[\mu]}(C)$ for any $[\mu]$, then what can we say about $\BB \circ \omega_{\mc{L}}$? The main result to be developed in the remainder of this section, \cref{theorem.good-equivalence-bilatticed}, implies that this is equivalent to inexactness of $\BB \circ \omega_{\mc{L}}$ (i.e. failure to be a filtered $G$-bundle).  

\begin{example}\label{example.latticed-Ga-not-exact}
We can interpret \cref{example.latticed-K-vect-BB-not-exact} as a lattice on the trivial $\mbb{G}_a$-bundle. It is classified by the point $\frac{1}{t} \cdot B^+ \in B/B^+=\mbb{G}_a(B)/\mbb{G}_a(B^+)$. This lies outside of any Schubert cell --- indeed, the only cocharacter of $\mbb{G}_a$ is trivial, so only the trivial coset $B^+$ itself lies in a Schubert cell. The failure of the short exact sequence in \cref{example.latticed-K-vect-BB-not-exact} to remain exact after passing to filtered vector spaces illustrates the failure of $\BB \circ \omega_{\mc{L}}$ to be exact if $\omega_{\mc{L}}$ lies outside a Schubert cell. 
\end{example}
    
\subsection{Bilatticed $B$-vector spaces and exactness of filtrations}\label{ss.bilatticed-and-exactness}

We now define a coarser category than $\Vect^{B^+\dash\mr{latticed}}(L)$ through which $\BB$ factors and whose isomorphism classes will capture the invariant $[\mu]$ discussed above. Using this category, we will formulate and prove the general principle relating the Schubert cells and the failure of exactness of $\BB$.

\begin{definition} Let $\Vect^{\bl}(B)$ be the category of bilatticed $B$-vector spaces, i.e. finite dimensional $B$-vector spaces equipped with a pair $\mc{L}_1, \mc{L}_2$ of $B^+$-lattices. A morphism
$(V, \mc{L}_1, \mc{L}_2) \rightarrow (V', \mc{L}_1', \mc{L}_2')$ is a morphism of $B$-vector spaces  $f: V \rightarrow V'$ such that $f(\mc{L}_i) \subseteq \mc{L}_i'$ for $i=1, 2$. A morphism is \emph{strict} if $f(\mc{L}_i)=f(V) \cap \mc{L}'_i$ for $i=1,2$. A complex of bilatticed $B$-vector spaces is exact if it is exact as a complex of $B$-vector spaces and each morphism is strict. 
\end{definition}

If $V$ is a $B$-vector space and $\mc{L} \subset V$ is a $B^+$-lattice, we consider the filtration on $V$ by the $B^+$-submodules $F^\bullet B \cdot \mc{L}$. Using the filtrations of this form, we obtain two $L$-linear tensor functors from $\Vect^\bl(B)$ to $\Vect^f(C)$:
\[ \BB_1: (V, \mc{L}_1, \mc{L}_2) \mapsto (\mc{L}_{1,C},\trFil_{\mc{L}_2}^\bullet \mc{L}_{1,C}) \textrm{ and }  \BB_2: (V, \mc{L}_1, \mc{L}_2) \mapsto (\mc{L}_{2,C},\trFil_{\mc{L}_1}^\bullet \mc{L}_{2,C}) \]
where $\trFil_{\mc{L}_i}^\bullet \mc{L}_{j,C}$ is the trace filtration as in \cref{ss.from-filtrations-to-lattices-and-back} formed by taking the image of $(F^\bullet B \cdot \mc{L}_i) \cap \mc{L}_j$ in  $\mc{L}_{j, C}$. 

\begin{example}\label{example.latticed-k-vect-to-bilatticed}
We define two functors from $\Vect^{B^+\dash\mr{latticed}}(L)$ to $\Vect^{\bl}(B)$:
\[ \bl_1: (V, \mc{L}) \mapsto (V_{B}, V_{B^+}, \mc{L}),\; \bl_2: (V, \mc{L}) \mapsto(V_B, \mc{L}, V_{B^+}) \]
They differ by the involution of $\Vect^{\bl}(B)$ swapping the two lattices, and for each $i=1,2$ the following diagram commutes up to a natural isomorphism:
% https://q.uiver.app/?q=WzAsMyxbMCwwLCJcXFZlY3Ree0JeKy1cXG1ye2xhdHRpY2VkfX0oSykiXSxbMSwwLCJcXFZlY3Ree1xcYmx9KEIpIl0sWzEsMSwiXFxWZWN0XmYoQykiXSxbMSwyLCJcXEJCX2kiXSxbMCwxLCJcXGJsX2kiXSxbMCwyLCJcXEJCIiwyXV0=
\[\begin{tikzcd}
	{\Vect^{B^+\dash\mr{latticed}}(L)} & {\Vect^{\bl}(B)} \\
	& {\Vect^f(C)}
	\arrow["{\BB_i}", from=1-2, to=2-2]
	\arrow["{\bl_i}", from=1-1, to=1-2]
	\arrow["\BB"', from=1-1, to=2-2]
\end{tikzcd}\]
The functors $\bl_i$, $i=1,2$ are exact, so any failure of exactness in $\BB$ is visible in the functors $\BB_i$. The isomorphism class in $\Vect^\bl$ will give the invariant describing the Schubert cell for a $B^+$-lattice on the trivial $G$-bundle. 
\end{example}

There is a simple relation between the associated gradeds for the $\BB_i$:

\begin{lemma}\label{lemma.associated-gradeds-twist} For each $j \in \mbb{Z}$, scalar multiplication induces an isomorphism of functors from $\Vect^{\bl}(B)$ to $\Vect(C)$
\[ \gr^{j}B \otimes (\gr^{-j}\circ \BB_1)  \xrightarrow{\sim} \gr^{j} \circ \BB_2. \]
\end{lemma}
\begin{proof}
We have
\begin{align*} \gr^{-i}_{\mc{L}_2} \mc{L}_{1,C} &= \left( (t^{-i} \mc{L}_2) \cap \mc{L}_1 \right) /\left( (t^{-i+1}\mc{L}_2) \cap \mc{L}_1 + (t^{-i} \mc{L}_2) \cap t \mc{L}_1 \right)\\
\gr^{i}_{\mc{L}_1} \mc{\mc{L}}_{2,C}&= \left( \mc{L}_2 \cap (t^{i} \mc{L}_1) \right) / \left( t \mc{L}_2 \cap (t^i \mc{L}_1) + \mc{L}_2 \cap (t^{i+1} \mc{L}_1)\right)  
\end{align*}
and multiplying by $t^{i}$ is an isomorphism from the the first to the second. 
\end{proof}

\begin{remark} 
In $p$-adic Hodge theory, this is the usual isomorphism from the Tate twists of the graded components for the Hodge filtration to the graded components for the Hodge-Tate filtration.
\end{remark}

For $G$ a connected linear algebraic group over $L$, in the following a filtered $G$-bundle will mean a filtered $G$-bundle over $\Spec\+ C$, i.e. an exact rigid tensor functor $\omega_f: \Rep\+ G \rightarrow \Vect^f(C)$. Recall from the previous subsection that the isomorphism classes are parameterized by the type of the filtration, a conjugacy class of cocharacters $[\mu]_{\omega_f}$ of $G$. 

A bilatticed $G$-bundle (over $B$) is an exact $L$-linear tensor functor
\[ \omega_\bl: \Rep\+ G \rightarrow \Vect^\bl(C). \]
The induced functors $\mc{L}_i \circ \omega_\bl$ from $\Rep\+ G$ to $B^+$-modules are $G$-bundles on $\Spec\+ B^+$, thus trivial (because the residue field $C$ is algebraically closed). Choosing trivializations for each, the isomorphism $\mc{L}_2 \otimes B \xrightarrow{\sim} V \xrightarrow{\sim} \mc{L}_1 \otimes B$ induces an automorphism of the trivial $G$-bundle over $\Spec\+ B$, thus an element of $G(B)$. The image in the double coset space
\[ G(B^+) \backslash G(B) / G(B^+) \]
is independent of the choice of trivializations, and this assignment identifies the double cosets with the isomorphism classes of bilatticed $G$-bundle. 

As above, if $G$ is reductive, then the Cartan decomposition identifies this double coset space with the conjugacy classes of cocharacters of $G$ (a conjugacy class is matched with the unique double coset containing $t^\mu=\mu(t)$ for any element $\mu$ in the conjugacy class). For a general connected linear algebraic $G$ we can assign in this way distinct double cosets to each conjugacy class $[\mu]$, but not all double cosets are of this form.

\begin{definition}\label{def:good-bi-latticed-torsor}
A bilatticed $G$-bundle $\omega_\bl$ is \emph{good} if it lies in the double coset $G(B^+) t^{\mu} G(B^+)$ of a conjugacy class of cocharacters $[\mu]$. In this case we write $[\mu]_{\omega_\bl}$ for this conjugacy class and call it the \emph{type} of the good bilatticed $G$-bundle. 
\end{definition}

We now explain the connection with exactness of filtrations: If $G$ is reductive, then since $\Rep\+ G$ is semisimple it is easy to verify that for any bilatticed $G$-bundle $\omega_{\bl}$, $\BB_i \circ \omega_{\bl}$ is exact, i.e. is a filtered $G$-bundle (this is because $\BB_i$ preserves split exact sequences). For $G$ non-reductive this may not be the case, as can be seen by combining \cref{example.latticed-k-vect-to-bilatticed} and \cref{example.latticed-Ga-not-exact}. In fact, we have

\begin{theorem}\label{theorem.good-equivalence-bilatticed}
The following are equivalent for a bilatticed $G$-bundle $\omega_\bl$:
\begin{enumerate}
    \item $\omega_\bl$ is good 
    \item $\BB_1 \circ \omega_\bl$ is a filtered $G$-bundle (i.e. is exact)
    \item $\BB_2 \circ \omega_\bl$ is a filtered $G$-bundle (i.e. is exact)
\end{enumerate}
In this case, if the type of $\omega_\bl$ is $[\mu]$, then the type of $\BB_1 \circ \omega_\bl$ is $[\mu^{-1}]$ and the type of $\BB_2 \circ \omega_\bl$ is $[\mu]$.  
\end{theorem}

To show this, we first observe that it is immediate from the definitions that if $\omega_\bl$ is good then the $\BB_i \circ \omega_\bl$ are exact with the claimed types --- indeed, to say that $\omega_\bl$ is good of type $[\mu]$ is the same as saying that 
it is isomorphic to the functor 
\[ \Rep\+ G \rightarrow \Vect^\bl(B), (\rho, V) \mapsto (V_B, V_{B^+}, \rho(\mu(t))\cdot V_{B^+} )\] 
and for this bilatticed $G$-bundle the associated filtrations are exactly those given by $\mu^{-1}$ and $\mu$. It is also immediate from \cref{lemma.associated-gradeds-twist} that (2) $\leftrightarrow$ (3) --- indeed, a functor $\Rep\+ G \rightarrow \Vect^f(C)$ is exact if and only if the associated graded is exact. 

Thus what remains is to show that if $\BB_1 \circ \omega_\bl$ (or $\BB_2 \circ \omega_\bl$) is exact then $\omega_\bl$ is good. This turns out to be intimately related to the behavior of types under extensions. Indeed, this is illustrated in \cref{example.latticed-Ga-not-exact}: the associated sequence of filtered vector spaces fails to be exact, and the type of the extension is $(1,-1)$ while the type of the associated graded is $(0,0)$. When a short exact sequence of bilatticed vector spaces does give rise to a short exact sequence of filtered vector spaces, then it is easy to see that the type stays the same (simply take the associated graded for the short exact sequence of filtered vector spaces). Our strategy for finishing the proof of \cref{theorem.good-equivalence-bilatticed} is thus to show that the good bilatticed $G$-bundles are exactly the ones such that, for any faithful representation $V$ of $G$, the type of $\omega_{\bl}(V)$ agrees with the type of its associated graded (for the filtration by sub-objects corresponding to the unipotent radical $U$ of $G$). 
%This is essentially the contents of \cref{theorem.specialization-goes-up-when-not-good} below, which will also be used Part II to study Schubert cells in the $B_\dR$-affine Grassmannians when $G$ is not reductive. 

\begin{example}\label{example.fvs-types-same-not-strict}
Consider the following sequence of filtered vector spaces $0 \rightarrow S \rightarrow V \rightarrow Q \rightarrow 0$, where the maps in the first row are inclusion and projection:
\[ \begin{matrix}F^k, k \leq 0: & 0 \rightarrow Ce_1 \rightarrow Ce_1 + Ce_2 \rightarrow Ce_2 \rightarrow 0 \\
F^1: & 0 \rightarrow 0 \rightarrow Ce_1 \xrightarrow{0} Ce_2 \rightarrow 0 \\
F^{k}, k \geq 2: & 0 \rightarrow 0 \rightarrow 0 \rightarrow 0 \rightarrow 0 \end{matrix} \]
This is a short exact sequence of vector spaces and the maps respect the filtrations but are not strict exact. However, both $V$ and $S \oplus Q$ have type $(1,0).$ This particular example cannot arise from an exact sequence of bilatticed vector spaces because it is easy to check that for a strict surjection of bilatticed vector spaces, the induced map on filtered vector spaces is surjective on the smallest non-zero part of the filtration (whereas the map on $F^1$ in this example is the zero map!). 
\end{example}

\subsection{Dominance and extensions}

Note that, in \cref{example.latticed-Ga-not-exact} as discussed above, the type of the extension does not change in an arbitrary way: the type $(1,-1)$ of the extension is not equal to the type $(0,0)$ of the associated graded, but it does lie above it in the standard dominance partial ordering. We recall this partial order now, then prove that the type of an extension always rises over the type of the graded; we will need to bootstrap from this computation to identify when equality holds. 

Recall that a conjugacy class of cocharacters $[\mu]$ of $\GL_n$ is equivalent to a multiset of $n$ integers, where $i$ appears with multiplicity equal to the multiplicity of the character $z \rightarrow z^i$ of any representative of $[\mu]$ acting on the standard representation of $\GL_n$. We will write such a multiset as $[\mu]=(\mu_1, \ldots, \mu_n)$ 
where $\mu_1 \geq \mu_2 \geq \cdots \geq \mu_n$. Then, if $[\mu]=(\mu_1, \ldots, \mu_n)$ and $ [\nu]=(\nu_1, \ldots, \nu_n)$ we say $[\mu] \leq [\nu]$ if for each $k$, $\sum_{i=1}^k \mu_i \leq \sum_{i=1}^k \nu_i$. This is the dominance partial order.

For $V$ a bilatticed vector space of type $[\mu]=(\mu_1, \ldots, \mu_n)$, $\mu_1 \geq \ldots \geq \mu_n$, we would like to give a characterization of this number $\sum_{i=1}^k \mu_i$ in terms of $V$. To that end, first note that it can be thought of as the largest integer appearing in the type of the $k$th exterior power $\Lambda^k V$. On the other hand, if we define the \emph{order}, $\mr{ord}(W)$, of a bilatticed vector space $W=(B,\mc{L}_1,\mc{L}_2)$ to be smallest $i$ such that $t^i \mc{L}_1 \subseteq \mc{L}_2$, then it is trivial to see that if $W$ has type $[\nu]=(\nu_1,\ldots,\nu_n)$  then $\mr{ord}(W)=\nu_1$. Thus $\sum_{i=1}^k \mu_i = \mr{ord}(\Lambda^k V)$, and this gives us a useful way to get a hold of these numbers.  Using this interpretation, we can see that the type always rises in extensions: the proof is almost identical to that of \cite[Lemma 1.2.3]{katz:isocrystals}.
\newcommand{\ord}{\mr{ord}}
\begin{proposition}\label{prop.latt-graded-properties}
$V \in \Vect^\bl(B)$ be equipped with a  filtration by strict sub-objects $F^i V$, and let $\gr^\bullet(V)=\bigoplus_i F^i V/F^{i+1} V$, an object of $\Vect^{\bl}(B)$. If $\gr^\bullet V$ is of type $[\mu]$ and $V$ is of type $[\nu]$ then $[\mu] \leq [\nu]$.
\end{proposition}
\begin{proof}
By induction on the length of the filtration by strict sub-objects, it suffices to treat the case where the filtration has two steps, i.e. where we have a strict exact sequence in $\Vect^\bl(B)$
\[ 0 \rightarrow S \rightarrow V \rightarrow Q \rightarrow 0. \]
We want to show the type of $V$ lies above the type of $S \oplus Q$. 

Thus we need to show $\ord(\Lambda^k V) \geq \ord (\Lambda^k (S \oplus Q))$ for each $k$. To that end, first note $\Lambda^k (S \oplus Q)) = \oplus_j \Lambda^{j} S \otimes \Lambda^{k-j} Q$, thus $\mr{ord}(\Lambda^k (S \oplus Q)) )= \mr{max}_j \ord(\Lambda^{j} S \otimes \Lambda^{k-j} Q)$. On the other hand, $\Lambda^k V$ has a filtration by strict subojects such that the associated graded terms are $\Lambda^{j} S \otimes \Lambda^{k-j} Q$, so $\mr{ord}(\Lambda^k V) \geq \mr{max}_j  \ord(\Lambda^{j} S \otimes \Lambda^{k-j} Q).$ 
\end{proof}

\begin{remark}
We can also define the \emph{order} of a filtered vector space $(V, F^\bullet V)$ to be the largest $k$ such that $F^k V\neq 0$; a simple computation shows $\BB_2$ preserves order. 
\end{remark}

\subsection{The action of unipotent elements on affine Schubert cells}\label{ss.unipotent-elements-acting}\hfill\\

\noindent {\bf Notation}: \emph{To make the proofs more explicit and to free up the letter $B$ for Borel, for the remainder of \cref{s.filtrations-and-lattices} we fix an isomorphism $B^+=C[[t]]$, and write $C[[t]]$ in place of $B^+$ and  $C((t))$ in place of $B$ --- see \cref{example.filtration-lattice-situations}-(3).}\\

We also fix some notation for this subsection only: let $G=\GL_n$, let $T \subseteq \GL_n$ be the diagonal maximal torus, and let $\mc{B}$ denote the set of Borel subgroups containing $T$. Let $B^+$ be the upper-triangular Borel with unipotent subgroup $U^+$, and let $B^{-}$ be lower triangular Borel with unipotent subgroup $U^{-}$. We write $X_*(T)$ for the group of cocharacters of $T$. For each root $\alpha$ there is a co-root $\alpha^\vee \in X_*(T)$; if $\alpha$ is positive (by which we mean with respect to $B^+$) we call $\alpha^\vee$ a positive co-root. We call a cocharacter $\mu \in X_*(T)$ dominant if $\alpha(\mu)\geq 0$ for each positive root $\alpha$, and let $X_\ast(T)^+$ be the set of such dominant cocharacters. (Concretely, these are the cocharacters $\mu(z) = \mr{diag}(z^{a_1}, \ldots, z^{a_n})$ for integers $a_1 \geq a_2 \geq \cdots \geq a_n$.)
For cocharacters $\mu, \mu' \in X_\ast(T)$, we write $\mu \leq \mu'$ if $\mu'-\mu$ is a non-negative sum of positive co-roots. On the dominant co-characters this is equivalent to the dominance relation discussed in the previous subsection plus the condition that $\det \mu=\det \mu'$.

We write $K=\GL_n(C[[t]])$. Every conjugacy class $[\mu]$ of cocharacters of $\GL_n$ has a unique dominant representative $\mu$, so the Cartan decomposition reads
\[ \GL_n(C((t))) = \bigsqcup_{\mu \in X_\ast(T)^+} K t^\mu K .\]
For $B \in \mc{B}$ with unipotent subgroup $U$, the Iwasawa decomposition is 
\[ \GL_n(C((t))) = \bigsqcup_{\mu \in X_*(T)} U(C((t))) t^\mu K. \]
Essentially, we need to compare these two decompositions. The following lemma can be deduced from results of Bruhat and Tits \cite{bruhat-tits:groups-reductifs-sur-un-corps-local} in a more general setting. We give a direct proof then explain in \cref{remark.bt-proof} how it can be deduced from \cite{bruhat-tits:groups-reductifs-sur-un-corps-local}.

\begin{lemma}\label{lemma.first-unipotent-shifting}
Let $B \in \mc{B}$ with unipotent subgroup $U$. Let $\mu \in X_*(T)^+$ be a dominant co-character, and let $u \in U(C((t)))$.  Let $\nu \in X_*(T)^+$ be the dominant cocharacter such that $u \cdot t^\mu K \subseteq K t^\nu K$. Then, 
\begin{enumerate}
    \item $\mu \leq \nu$. 
    \item if $B=B^{-}$ then $\mu=\nu$ if and only if $u \cdot t^\mu K= t^\mu K$ or equivalently 
    \[ t^{-\mu} u t^{\mu} \in U(C[[t]])=U(C((t))) \cap K. \]
    \item if $B=B^+$ then  $\mu=\nu$ if and only if $u \in U(C[[t]])=K \cap U(C((t)))$. 
\end{enumerate}
\end{lemma}
\begin{proof}
Part (1) is an immediate consequence of \cref{prop.latt-graded-properties}. Indeed, if we consider the bilatticed vector space $(C((t))^n, C[[t]]^n, u t^\mu \cdot C[[t]]^n)$, then it has a filtration by strict subojects coming from the standard filtration stabilized by $B$, and the associated graded is isomorphic $(C((t))^n, C[[t]]^n, t^\mu \cdot C[[t]]^n)$.

For the rest, we first note that (2) and (3) are equivalent: suppose (2) holds, and suppose $u \in U^+$ is such that $u \cdot t^\mu \in K t^\mu K$. Then, by taking transpose, $t^{\mu} u^t=(t^{\mu} u^t t^{-\mu}) t^\mu \in K t^\mu K$, and since $(t^{\mu} u^t t^{-\mu}) \in U^{-}$, (2) gives that $u^t=t^{-\mu}(t^{\mu} u^t t^{-\mu})t^{\mu} \in U^-(C[[t]])$, and thus $u \in U^+(C[[t]])$. In the other direction, suppose (3) holds, and let $u \in U^{-}$ be such that $u t^{\mu} \in K t^\mu K$. Then, taking transpose,  $t^{\mu} u^t=(t^{\mu} u^t t^{-\mu}) t^\mu \in K t^\mu K$ and by (3), $t^{\mu} u^t t^{-\mu} \in U^+(C[[t]])$. Taking transpose again, we find $t^{-\mu} u^t t^{\mu} \in U^-(C[[t]])$. 

It thus suffices to prove $(3)$. The if direction is clear, so, it remains only to show that if $u \not \in U^+(C[[t]])$ then $\nu > \mu$. To that end, let $i$ be the index of the first row of $u$ whose entries are not in $C[[t]]$, and let $j>i$ be the first index such that $u_{i,j} \not\in U^+(C[[t]])$. Then, for $V=(C((t))^n, C[[t]]^n, u t^\mu \cdot C[[t]]^n)$, and $F_i V$ the standard increasing filtration of $V$ attached to $B^+$, one finds that the type of $F_jV/F_{i-1}V$ is strictly larger than the type $(\mu_i, \ldots, \mu_j)$ of its associated graded. Indeed, we use the opposite order 
    \[ \ol{\ord}(V) = \max\{k \colon t^{-k}\mc{L}_2 \subseteq \mc{L}_1 \} \]
which picks out the smallest integer in the type of $V$. The vector $u_{i,j}t^{\mu_j}e_i + \cdots + t^{\mu_j} e_j$ is in the second lattice of $F_jV/F_{i-1}V$, so if $t^{-k}\mc{L}_2(F_jV/F_{i-1}V) \subseteq \mc{L}_1(F_jV/F_{i-1}V)$ then $k \leq \mu_j + v(u_{i,j}) < \mu_j$ (where $v$ denotes the additive $t$-adic valuation). Hence $\ol{\ord}(F_jV/F_{i-1}V) \leq \mu_j + v(u_{i,j}) < \mu_j$, 
\iffalse
{\color{blue} More importantly, the following example violates the inequality $\ord(F_j V/F_{i-1}V) \geq \mu_{i} + -v(u_{i,j})$. 

Take $u = \begin{pmatrix} 1 & 1 & t^{-1} \\ 0 & 1 & t^{-1} \\ 0 & 0 & 1 \end{pmatrix}$, and $\mu_1 = 2, \mu_2 = 1, \mu_3 = 0$, so 
    \[ \mc{L}_2 = \langle t^2e_1, te_1 + te_2, t^{-1}e_1 + t^{-1}e_2 + e_3 \rangle. \]
Then $\mu_1 + -v(u_{1,3}) = 3$, but $\ord(V) \leq 2$, since $t^2e_1, t^2e_2 = t(te_1 + te_2) = t^2e_1$ and 
    \[ t^2 e_3 = t^2(t^{-1}e_1 + t^{-1}e_2 + e_3) - (te_1 + te_2). \]
Of course, in this example, 
    \[ \mc{L}_2 = \langle t^{-1}(e_1 + e_2 + te_3), t^2(e_1 + e_2 + e_3), t^2e_1 \rangle \]
so the type of $V$ is $(2,2, -1) > (2,1,0)$ as required. }
\fi
 which implies the type of $F_j V/ F_{i-1} V$ is strictly greater than the type $(\mu_i, \ldots, \mu_j)$ of its associated graded by \cref{prop.latt-graded-properties}. Applying \cref{prop.latt-graded-properties} to the filtration $0 \subseteq F_{i-1} V \subseteq F_j V \subseteq V$ and the non-strict inequalities that the type of $F_{i-1} V$ is greater than or equal to $(\mu_1, \ldots, \mu_{i-1})$ and the type of $V/F_j V$ is greater than or equal to $(\mu_{j+1}, \ldots, \mu_n)$, we conclude. 
\end{proof}

\begin{remark} \label{remark.bt-proof}
Alternatively, \cref{lemma.first-unipotent-shifting}-(1) follows from \cite[Corollary 4.3.17]{bruhat-tits:groups-reductifs-sur-un-corps-local} and \cref{lemma.first-unipotent-shifting}-(2) follows from \cite[Proposition 4.4.4-(ii)]{bruhat-tits:groups-reductifs-sur-un-corps-local}; one can then obtain \cref{lemma.first-unipotent-shifting}-(3) by using the equivalence between (2) and (3) explained in the proof above. To make it possible to follow these citations, we explain how our notation compares to the notation in \cite[Chapter 4]{bruhat-tits:groups-reductifs-sur-un-corps-local}:  For $D$ the dominant chamber for the roots of a Borel $B$ containing $T$, the group  $\hat{\mc{B}}^0_D$ of loc. cit. is the group generated by the set of diagonal matrices with entries in $C((t))$ such that all of the entries have the same valuation and the $C((t))$-points of the unipotent subgroup of the \emph{opposite} Borel to $B$. In particular, the fixed chamber $\bf{D}$ in loc.\ cit.\ corresponds in our setup to the dominant chamber for the roots of $B^+$, so that $\hat{\mc{B}}^0$, which by definition in loc. cit. is $\hat{\mc{B}}^0_{\bf{D}}$, contains our $U^{-}$. 
\end{remark}

We need refinements of parts (2)-(3) that apply to arbitrary $B \in \mc{B}$. We fix as above a dominant cocharacter $\mu \in X_\ast(T)^+$ and $B \in \mc{B}$ with unipotent radical $U$. The conjugation action of $\mu$ induces a decomposition $U=U_{>0} \cdot U_0 \cdot U_{<0}$ where $\Lie\+ U_{>0}$ (resp. $\Lie\+ U_0$, resp. $\Lie\+ U_{<0}$) consists of all roots $\alpha$ in $\Lie\+ U$ such that $\alpha(\mu)>0$ (resp. $=0$, resp. $<0$). Note that since $\mu$ is dominant $U_{>0} \leq U \cap U^+$ and $U_{<0} \leq U \cap U^{-}$, while $U_0$, the subgroup of elements in $U$ centralized by $\mu$, can contain a mix of both positive and negative root subgroups. By \cite[Theorem 3.3.11]{conrad-gabber-prasad:pseudoreductive} this decomposition exists already for any closed subgroup $U' \subseteq U$ preserved by conjugation by $t^\mu$, and is compatible with the decomposition of $U$, i.e.
\begin{equation}\label{eq.U-intersections}U'=U'_{>0} U'_0 U'_{<0} \textrm{ for } U'_{>0} = U' \cap U_{>0}, U'_0=U' \cap U_0, \textrm{ and } U'_{<0}=U' \cap U_{<0}.\end{equation}

\begin{lemma}\label{lemma.refined-unipotent-action}
With notation as above, let $u \in U'(C((t)))$ and let $u=u_{>0} u_0 u_{<0}$ denote its product decomposition. Then $u \cdot t^\mu \in K t^\mu K$ if and only if all of the following hold:
\begin{enumerate}
    \item $u_{>0} \in U'_{>0} \cap K$
    \item $u_0=t^{-\mu} u_0 t^{\mu} \in U'_0 \cap K$
    \item $t^{-\mu}u_{<0}t^{\mu} \in U'_{<0} \cap K$
\end{enumerate}
Furthermore, in this case $u \cdot t^\mu K = u_{>0} \cdot t^\mu K = u_{>0} u_0 \cdot t^\mu K$.
\end{lemma}
\begin{proof}
It follows from the identities
\[ u t^\mu = u_{>0} u_0 u_{<0} t^\mu  = (u_{>0} u_0) t^\mu (t^{-\mu} u_{<0} t^\mu) = u_{>0} t^{\mu} ( u_0 t^{-\mu}u_{<0} t^\mu) \]
that if (1)-(3) are satisfied then $u \cdot t^\mu \in K t^\mu K$ and the last claim also holds. It remains to show that, if $u \cdot t^\mu \in K t^\mu K$ then (1)-(3) hold. 

Using \cref{eq.U-intersections}, we may assume $U'=U$. We further refine $U_{0}=U_{0,+}\cdot U_{0,-}$ where $U_{0,+}=U_{0} \cap U^+$ and $U_{0,-}=U_{0} \cap U^-$. We thus write $u=u_{>0} u_{0,+} u_{0,-} u_{<0}$. We note moreover that $U_{0,-}$ normalizes $U_{<0}$ and commutes with $t^\mu$, so we can rewrite this as
\[ u = u_{>0} u_{0,+} u_{0,-} u_{<0} = u_{>0} u_{0,+} u'_{<0} u_{0,-} \textrm{ where } u'_{<0}=u_{0,-} u_{<0} u_{0,-}^{-1} \]

We now choose a chain of adjacent Borel subgroups $B^+=B_0, B_1, B_2, \ldots B_l=B$ such that, for each $1\leq i \leq l$, there is a unique root $\alpha_i$ such that $\alpha_i$ is in $\Lie\+ U^+$ and $-\alpha_i$ is in $\Lie\+ U$ and to move from $B_{i-1}$ to $B_{i}$ we swap $\alpha_i$ for $-\alpha_i$. Note that, for each $i$, $-\alpha_i$ is a weight of $\Lie\+ U_{<0}$ or $\Lie\+ U_{0,-}$, and because $U_{0,-}$ normalizes $U_{<0}$, we can and do choose this chain so that $-\alpha_1, \ldots, -\alpha_s$ lie in $U_{<0}$ and $-\alpha_{s+1}, \ldots, -\alpha_l$ lie in $U_{0,-}$. Thus we may write 
\[ u= u_{>0} u_{0,+} u'_{<0} u_{0,-} = u_{>0} u_{0,+} u'_{-\alpha_1}\ldots u'_{-\alpha_s}u_{-\alpha_{s+1}}\ldots u_{-\alpha_l} \]
so that $u'_{<0}=u'_{-\alpha_1}\ldots u'_{-\alpha_s}$ and $u_{0,-}=u_{-\alpha_{s+1}}\ldots u_{-\alpha_l}$ for unique elements $u_{-\alpha_i}', i = 1,\ldots,s,$ and $u_{-\alpha_i}, i = s+1, \ldots, l$ of the root subgroups $U_{-\alpha_i}(C((t)))$.  Suppose now that $u_{0,-}$ is not in $K$. Then at least one factor $u_{-\alpha_{i}}$ must not be in $K$; let $s+1 \leq i \leq l$ be the index of the last root factor that is not in $K$. Note that since $t^\mu$ acts trivially by conjugation on $U_0$, we may pass the factors starting at $i$ through $t^\mu$ to obtain
\[ u \cdot t^\mu K = u_{>0} u_{0,+} u'_{<0} u_{-\alpha_{s+1}}\ldots u_{-\alpha_l} t^\mu K =  u_{>0} u_{0,+} u'_{<0} u_{-\alpha_{s+1}}\ldots u_{-\alpha_{i-1}} t^\mu u_{-\alpha_i} K. \]
We may fix an identification of $U_{-\alpha_i}(C((t)))$ with $C((t))$ such that $u_{-\alpha_i} =t^{-m}$ for $m>0$. From the following computation with $2 \times 2$ matrices, applied to the principal $\mr{SL}_2$ containing $U_{\alpha_i}$ and $U_{-\alpha_i}$ (the latter as the lower triangular unipotents and the former as the upper triangular unipotents), 
\[ \begin{bmatrix} 1 & 0\\ t^{-m} &1 \end{bmatrix} = \begin{bmatrix}1 & t^{m} \\ 0 & 1\end{bmatrix} \begin{bmatrix}t^m & 0 \\ 0 & t^{-m} \end{bmatrix} \begin{bmatrix} 0 & -1 \\ 1 & t^{m} \end{bmatrix}. \]
we find 
\[ u \cdot t^\mu K = u_{>0} u_{0,+} u'_{< 0} u_{-\alpha_{s+1}}\ldots u_{-\alpha_{i-1}} v_{\alpha_i} t^{\mu + m \alpha_i^\vee} K \]
where $v_{\alpha_i}$ is in the root subgroup for $\alpha_i$. The product $s$ on the left of $t^{\mu + m \alpha_i^\vee}$ lies in the unipotent radical of $B_{i-1}$. If we conjugate by the Weyl element $w$ sending $\mu + m \alpha_i^\vee$ to a dominant weight $\mu'$, this does not change the Schubert cell, so we get $w s w^{-1} t^{\mu'} \in K t^\mu K$. By construction, $\mu'$ is strictly larger than $\mu$ in the lexicographic order.  Since the element $w s w^{-1}$ lies in the unipotent radical of $w B_{i-1} w^{-1}$, \cref{lemma.first-unipotent-shifting}-(1) implies that $w s w^{-1} t^{\mu'}$ lies in $K t^\nu K$ for $\nu \geq \mu'$, but as lexicographic order refines dominance order, we see $\nu$ is strictly greater than $\mu$ in the lexicographic order, a contradiction to $\nu=\mu.$ 

Thus we have $u_{0,-} \in K$, and we may pass it through $t^\mu$ to obtain
\[  u \cdot t^\mu K = u_{>0} u_{0,+} u'_{<0} \cdot t^\mu K. \]
Arguing similarly to the above with $u'_{<0}=u'_{-\alpha_1}\ldots u'_{-\alpha_s}$, we find that $t^{-\mu} u'_{<0} t^{\mu} \in K$ so we may pass it through $t^\mu$ to obtain 
\[ u \cdot t^\mu K = u_{>0} u_{0,+} \cdot t^\mu K.\]

Since $u_{>0} u_{0,+} \in B^+(C((t)))$, we conclude that $u_{>0} u_{0,+}$ is in $K$ by \cref{lemma.first-unipotent-shifting}-(3). Thus, so is each factor. We have already established $u_{0,-} \in K$ and $t^{-\mu} u'_{<0} t^\mu \in K$, thus since $u_{0,-}$ commutes with $t^{\pm \mu}$, we find also that
\[ t^{-\mu} u_{<0} t^\mu = t^{-\mu} u_{0,-}^{-1} u'_{<0} u_{0,-} t^{\mu} = u_{0,-}^{-1} (t^{-\mu} u'_{<0} t^{\mu}) u_{0,-} \in K.\] 
\end{proof}

\subsection{Conclusion}

We now give the promised specialization result refining and generalizing \cref{prop.latt-graded-properties}. 

\begin{theorem}\label{thm.specialization-goes-up-when-not-good}
Let $G$ be a connected linear algebraic group over $C$ with Levi decomposition $G=M \ltimes U$.  Let $g \in G(C((t)))$ with $g=mu$ and let $[\mu]$ be the conjugacy class of cocharacters of $M$ identifying the Schubert cell containing $m$ (i.e. $m \in M(C[[t]]) t^{\mu} M(C[[t]])$). For any representation $\rho: G \rightarrow \GL_n$, if $[\nu]$ is the conjugacy class of cocharacters indexing the Schubert cell containing $\rho(g)$, then $[\nu]\geq [\rho\circ \mu]$ (the conjugacy class of cocharacters indexing the Schubert cell containing $\rho(m)$). If $\rho$ is faithful, then $[\nu]=[\rho \circ \mu]$ if and only if $g \in G(C[[t]]) t^{\mu} G(C[[t]]).$
\end{theorem} 
\begin{proof}

It changes nothing to multiply on the left and right by an element of $M(C[[t]]))$, so we may assume $m=t^{\mu}$. Then, replacing $G$ with $\mbb{G}_m \ltimes U$, we may assume $G$ is solvable. Then its image under $\rho$ lies in a Borel subgroup, and conjugating $\rho$ suitably by an element of $\GL_n(C)$ we may assume that $\rho \circ \mu$ is dominant and this is a standard Borel $B \in \mc{B}$. Then $\rho(u)$ factors through the unipotent radical of $B$, and the first claim follows from \cref{lemma.first-unipotent-shifting}. If the representation is faithful then $U$ is a subgroup of the unipotent radical of $B$ preserved under conjugation by $t^\mu$, so if $[\mu]=[\nu]$ then \cref{lemma.refined-unipotent-action} gives
\[ u t^\mu = u_{>0}u_0 u_{<0} t^{\mu} = u_{>0} u_0 t^{\mu} t^{-\mu}u_{<0} t^{\mu} \in U(C[[t]]) t^{\mu} U(C[[t]]). \]
\end{proof}

\begin{proof}[Conclusion of proof of \cref{theorem.good-equivalence-bilatticed}]
Recall that it remains only to show $(3) \implies (1)$ in \cref{theorem.good-equivalence-bilatticed}.  To that end, let $\omega_\bl$ be a bilatticed $G$-bundle, and suppose $\omega_f:=\BB_2 \circ \omega_\bl$ is a filtered $G$-bundle. Pick a representative $g \in G(C((t)))$ for the double coset corresponding to $\omega_\bl$, and write $g=mu$ for a Levi decompostion $G=MU$. Let $[\mu]$ be the conjugacy class of cocharacters of $M$ describing the Schubert cell of $m$. We want to show $g \in G(C[[t]]) t^\mu G(C[[t]])$.

To that end, let $\rho: G \rightarrow \GL(V)$ be a faithful representation. By \cref{thm.specialization-goes-up-when-not-good}, it suffices to show $[\rho\circ \mu]$ is the type of $\rho_\ast \omega_{\bl}$, i.e. that the bilatticed vector space associated to $\rho(g)$ is of type $[\rho \circ \mu]$. 

Let $\mf{u}:=\Lie\+ U$. Then $\rho(G)$ is contained in the parabolic subgroup stabilizing the filtration of $V$ by $\mf{u}^i V$. The action on the associated graded is through $M$, thus the associated graded for this filtration is of type $[\rho \circ \mu]$. Since the associated filtration by subobjects in the category of filtered vector spaces after applying $\BB_2$ is strict by assumption, we conclude the filtration on $V$ is also of type $[\rho \circ \mu]$. Since this agrees with the type of $\rho_\ast \omega_{\bl}$, we are done. 
\end{proof}

\section{Real and $p$-adic Hodge structures}\label{s.pahs}

In this section we will define the category of $p$-adic Hodge structures and establish some of its basic properties. By way of motivation, we begin in \cref{ss.real-hodge} by recalling a geometric perspective on the definition of a real Hodge structure via the theory of vector bundles on the twistor line (due to Simpson \cite{simpson:mixed-twistor}) that leads to a definition of a category of extended real Hodge structures containing real Hodge structures as a full subcategory. This perspective will be mirrored in our definition of $p$-adic Hodge structures in \cref{ss.p-adic-hodge-definition}, replacing the twistor line with the Fargues-Fontaine curve. With these definitions in place, in \cref{ss.structural-properties} and \cref{ss.pahs-invariants} we give a symmetric treatment of the basic structural properties of both extended real and $p$-adic Hodge structures --- in particular, we discuss Mumford-Tate groups and some fundamental invariants. In \cref{ss.hodge-tate-lines} we explain how to compute the Mumford-Tate group of a $p$-adic Hodge structure using Hodge-Tate lines, analogous to the Hodge tensors in classical Hodge theory. It is important here to give a criterion that applies to non-reductive groups; we explain more about the lack of reductivity in this theory and the relation to polarizability in classical Hodge theory in \cref{remark.reductivity-criterion}. 

\subsection{Definition of (extended) real Hodge structures}\label{ss.real-hodge} Traditionally, a real Hodge structure of weight $k$ is defined to be a pair $(V, \Fil^\bullet V_{\mbb{C}})$ where $V$ is a finite dimensional real vector space and $\Fil^\bullet$ is a decreasing filtration on $V_{\mbb{C}}$ satisfying the following transversality condition with its complex conjugate
\[ \Fil^{p} V_{\mbb{C}} \oplus  \overline{\Fil^{k-p+1} V_{\mbb{C}}} = V_{\mbb{C}}. \]

Our goal here is to reinterpret this transversality condition using the theory of vector bundles on the twistor line $\tilde{\mbb{P}}^1$, which can be constructed as the solution set of $U^2+V^2+W^2=0$ in $\mbb{P}^2_{\mbb{R}}$. Other constructions of $\tilde{\mbb{P}}^1$ and the theory of vector bundles on $\tilde{\mbb{P}}^1$ are discussed in detail in \cref{ss.twistor-line}. The key points here are:
\begin{enumerate}
    \item $\tilde{\mbb{P}}^1/\mbb{R}$ is a smooth projective curve with a canonical point $\infty_{\mbb{C}} \in \tilde{\mbb{P}}^1(\mbb{C})$.
    \item There is an action of the circle group $U(1)= a + bi$, $a^2+b^2=1$ on $\tilde{\mbb{P}}^1$ fixing a unique closed point, the one underlying $\infty_{\mbb{C}}$, and we may identify the completed local ring at $\infty_{\mbb{C}}$ with $\mbb{C}[[t]]$ where the coordinate $t$ is chosen so that $U(1)$ acts on $t$ by $z \cdot t = z^{-2}t$. 
    \item There is a slope formalism for vector bundles on $\tilde{\mbb{P}}^1$. Any vector bundle is a direct sum of stable bundles, any stable bundle has slope $\lambda \in \frac{1}{2} \mbb{Z}$, and for any such $\lambda$ there is a unique up-to-isomorphism stable bundle $\mc{O}_{\tilde{\mbb{P}}^1}(\lambda)$ of slope $\lambda$ (these are matched with the real isocrystals of \cref{sss.real-isocrystals}).  
\end{enumerate}

Now, suppose given a real vector space $V$ and a filtration $F^\bullet V_{\mbb{C}}$ on its complexification. By the canonical lattice construction of \cref{ss.from-filtrations-to-lattices-and-back} and \cref{prop.filt-det-latt-criteria}-(2), there is a unique promotion of $F^\bullet V_{\mbb{C}}$ to a $U(1)$-equivariant lattice in $V_{\mbb{C}((t))}$, $\mc{L}:=\sum_{i} t^{-i}\mbb{C}[[t]] \cdot \Fil^i V_{\mbb{C}}$. Let $(V \otimes_{\mbb{R}} \mc{O}_{\tilde{\mbb{P}}^1})_{\mc{L}}$ be the modification of $(V \otimes_{\mbb{R}} \mc{O}_{\tilde{\mbb{P}}^1})$ by the lattice $\mc{L}$ at $\infty$ (see Definition \ref{def.modification-vb}). 
%If we form the trivial vector bundle $V \otimes_{\mbb{R}} \mc{O}_{\tilde{\mbb{P}}^1}$, then we may use this lattice to construct a \emph{modified} vector bundle $(V \otimes_{\mbb{R}} \mc{O}_{\tilde{\mbb{P}}^1})_{\mc{L}}$ on $\tilde{\mbb{P}}^1$ --- indeed $V_{\mbb{C}((t))}$ is identified with the space of Laurent expansions at $\infty_{\mbb{C}}$ of meromorphic sections of $(V \otimes_{\mbb{R}} \mc{O}_{\tilde{\mbb{P}}^1})$, and the new bundle is formed by declaring it to have the same meromorphic sections as $V \otimes_{\mbb{R}} \mc{O}_{\tilde{\mbb{P}}^1}$ and declaring the holomorphic sections to be those that are holomorphic as sections of $V \otimes_{\mbb{R}} \mc{O}_{\tilde{\mbb{P}}^1}$ away from $\infty$ such that the Laurent expansion at $\infty$ lies in $\mc{L}$. 
By a result of Simpson \cite{simpson:mixed-twistor},

\begin{lemma}\label{lemma.Hodge-iff-semistable} $(V, F^\bullet V_{\mbb{C}})$ is a weight $k$ Hodge structure if and only if the modified vector bundle $(V \otimes_{\mbb{R}} \mc{O}_{\tilde{\mbb{P}}^1})_{\mc{L}}$ is semistable of slope $k/2$. 
\end{lemma} 
\begin{proof}
Here we just recall why the modification is semistable if it is a weight $k$ Hodge structure. First, observe that, by the constructions in \cref{ss.twistor-line}, it suffices to check that the pullback of the modification to $\mbb{P}^1$ is a direct sum of $\mc{O}(k)$. One can commute the pullback with the modification, so this is equivalent to modifying $V_\mbb{C} \otimes \mc{O}_{\mbb{P}^1_{\mbb{C}}}$ at $\infty$ by the Hodge filtration and at $-\infty$ by its complex conjugate. If we write $V_{\mbb{C}}=\bigoplus h^{p,q}$, then the filtrations both break up as a direct sum over this decomposition, where the Hodge filtration on $h^{p,q}$ is concentrated in degree $p$ and its conjugate in degree $q$. If we fix a basis for $h^{p,q}$, this further decomposes into modifications of one-dimensional trivial bundles along filtrations with these concentrations. The resulting modification is a line bundle of degree $p+q=k$, thus $\mc{O}(k)$, and we conclude. 
\end{proof}

Thus, the canonical filtration realizes the category of $\mbb{C}$-filtered real vector spaces as a full rigid tensor sub-category of $\mbb{C}[[t]]$-latticed real vector spaces, and the condition to be a weight $k$ Hodge structure can be formulated entirely in terms of the lattice. This motivates

\begin{definition}
The category of \emph{extended real Hodge structures} of weight $k$ is the full subcategory of $\mbb{C}[[t]]$-latticed real vector spaces (\cref{def.filt-latt-vector-spaces}) consisting of $(V, \mc{L})$ such that $(V \otimes_{\mbb{R}} \mc{O}_{\tilde{\mbb{P}}^1} )_{\mc{L}}$ is semistable of slope $k/2$. The category of extended real Hodge structures is the full subcategory of $\mbb{Z}$-graded $\mbb{C}[[t]]$-latticed real vector spaces consisting of those whose degree $k$ component is a weight $k$ extended real Hodge structure. 
\end{definition}

By \cref{lemma.Hodge-iff-semistable} and \cref{prop.filt-det-latt-criteria}-(2), the canonical lattice embeds the category of real Hodge structures as a full rigid tensor subcategory of extended real Hodge structures --- the ones such that the lattice is preserved by $U(1)$. Moreover, the functor $\BB$ of \cref{ss.from-filtrations-to-lattices-and-back} attaches to any extended real Hodge structure $(V,\mc{L})$ a Hodge filtration on $V_{\mbb{C}}$, recovering the Hodge filtration on real Hodge structures. In general the lattice contains more information.

\subsection{Definition and examples of $p$-adic Hodge structures}\label{ss.p-adic-hodge-definition}
Let $C/\mbb{Q}_p$ be an algebraically closed non-archimedean extension. In \cite{fargues-fontaine:courbes}, Fargues and Fontaine constructed from $C$ a $1$-dimensional scheme over $\Spec\+ \mbb{Q}_p$, $\FF=\FF_{C^\flat, \mbb{Q}_p}$, the Fargues-Fontaine curve, that behaves in every way like a smooth proper curve over $\mbb{Q}_p$ except that it is of infinite type. The geometry of $\FF$ encodes the relation between Fontaine's period rings. The construction is recalled in more detail in \cref{ss.fargues-fontaine}; the key points for this section are:
\begin{enumerate}
    \item There is a canonical closed point $\infty_C$ and identification of the residue field $\kappa(\infty_C)$ with $C$.  The completed local ring of $\infty_C$ is the complete discrete valuation ring $B^+_\dR$ and the complement of $\infty_C$ is the spectrum of the principal ideal domain $B_\crys^{\varphi=1}$. 
    \item If $C=\overline{K}^\wedge$ for a $p$-adic field $K$, then there is an action of $\Gal(\overline{K}/K)$ on $\FF=\FF_{\mbb{Q}_p, C^\flat}$ and $\infty_{C}$ is the unique closed point fixed by $\Gal(\overline{K}/K)$; the induced action on the residue field $C$ is the standard one, and $t \in B^+_\dR$, the ``$p$-adic $2\pi i$'' which generates the maximal ideal in $B^+_\dR$, transforms by the $p$-adic cyclotomic character of $\Gal(\overline{K}/K)$. 
    \item There is a slope formalism for vector bundles on $\FF$, and any semistable vector bundle has slope $\lambda \in \mbb{Q}$ and is a direct sum of copies of the unique up-to-isomorphism stable bundle $\mc{O}_\FF(\lambda)$ of slope $\lambda$ (these are matched with the $p$-adic isocrystals of \cref{sss.isocrystals}). 
\end{enumerate}

Our definition of a $p$-adic Hodge structure will be analogous to the definition of extended real Hodge structures but using $\FF$ instead of $\tilde{\mbb{P}}^1$.

\begin{definition}\label{def.p-adic-hs}
The category of \emph{$p$-adic Hodge structures} over $C$ of weight $\lambda \in \QQ$ is the full subcategory of $B_\dR^+$-latticed $\QQ_p$-vector spaces (\cref{def.filt-latt-vector-spaces}) consisting of $(V, \mc{L})$ such that $(V \otimes_{\QQ_p} \mc{O}_{\FF} )_{\mc{L}}$ --- the modification of $V \otimes_{\QQ_p} \mc{O}_{\FF}$ by the lattice $\mc{L}$ at $\infty_C$ as in \cref{def.modification-vb} --- is semistable of slope $-\lambda/2$. The category of $p$-adic Hodge structures over $C$ is the full subcategory of $\mbb{Q}$-graded $B_\dR^+$-latticed $\QQ_p$-vector spaces consisting of those whose degree $\lambda$ component is a weight $\lambda$ $p$-adic Hodge structure over $C$. 
\end{definition}

\begin{remark} The convention on weights is chosen so that the Tate structure $\mbb{Q}_p(k)$ of \cref{example.tate-p-adic-hs} has weight $-2k$.
\end{remark}

\begin{remark}
If $K \subseteq C$ is a $p$-adic field and $F^\bullet V_K$ is a filtration of $V_K$, then we can promote it to a lattice $\mc{L}$ on $V_{B_\dR}$ using the canonical lattice construction. Using this, one could define an analog of real Hodge structures instead of extended real Hodge structures, however, outside of the minuscule case where the filtration and lattice are equivalent, this construction does not arise in any natural way that we are aware of --- for example, \cref{main.transcendence} implies that, for a $p$-adic Hodge structure coming from geometry over a $p$-adic field, the lattice will never come from such a filtration unless the $p$-adic Hodge structure has complex multiplication. 
\end{remark}

The category $\HS(C)$ of $p$-adic Hodge structures is a full subcategory of the category of $\mbb{Q}$-graded $B^+_\dR$-latticed $\mbb{Q}_p$-vector spaces, and is stable under the natural tensor product and dual functors. The forgetful functor
    \[ \omega_\et \colon \HS(C) \to \Vect(\QQ_p) \]
is a faithful exact tensor functor. The functor $\BB$ from \cref{ss.from-filtrations-to-lattices-and-back} induces a Hodge-Tate filtration $\Fil^i_\HT$ on $\omega_\et \otimes C$, with 
\[ \Fil^i_\HT(V, \mc{L}) =\textrm{Image in $V_C$ of } (V_{B^+_\dR} \cap (\Fil^i B_\dR) \cdot \mc{L}) \]
It is a tensor functor, but we caution that it is not exact, even when restricted to $p$-adic Hodge structures  --- see \cref{example.ext-of-trivial} below. 

\begin{example}\label{example.latticed-vect-space-geometry}
If $X/C$ is a smooth proper rigid analytic variety then, by \cite[Theorem 13.1]{bhatt-morrow-scholze:1}, for each degree $i$, we obtain a $B^+_\dR$-latticed $\mbb{Q}_p$-vector space
\[ (H^i_\et(X, \mbb{Q}_p), \mc{L}^i_\dR) \]
where $\mc{L}^i_\dR$ is a canonical deformation of $H^i_\dR(X)$ to $B^+_\dR$ embedded as a $B^+_\dR$-lattice inside of $H^i_\et(X,\mbb{Q}_p) \otimes B_\dR$. The induced filtration on $H^i_\et(X, \mbb{Q}_p) \otimes C$ is the Hodge-Tate filtration. If $\mf{X}/\mc{O}_C$ is a smooth proper formal model, then
\[ (H^i_\et(X, \mbb{Q}_p) \otimes \mc{O}_\FF)_{\mc{L}^i_\dR} \]
is semistable if and only if the isocrystal determined by the $i$th crystalline cohomology of the special fiber is $i/2$-isotypic. In particular, if it is a $p$-adic Hodge structure it is of weight $i$, but in many cases it is not a $p$-adic Hodge structure, for example if $X$ is an elliptic curve with ordinary reduction.   
\end{example}

\begin{example}\label{example.tate-p-adic-hs}
For $k \in \mbb{Z}$, the Tate $p$-adic Hodge structure of weight $-2k$ is
\[ \mbb{Q}_p(k):= (\mbb{Q}_p \cdot t^k, B^+_\dR) \cong (\mbb{Q}_p, \Fil^{-k} B^+_\dR). \]
The Hodge-Tate filtration on $\mbb{Q}_p(k) \otimes C$ is concentrated in degree $k$ (i.e. the only non-zero graded is in degree $k$). For $k \geq 0$, $\mbb{Q}_p(k)$ is isomorphic to the dual of the  $p$-adic Hodge structure $(H^{2k}_\et(\mbb{P}^k, \mbb{Q}_p), \mc{L}^{2k}_\dR)$ of the previous example. We have natural identifications $\mbb{Q}_p(a)\otimes \mbb{Q}_p(b)=\mbb{Q}_p(a+b)$ and  $\mbb{Q}_p(k)^\vee=\mbb{Q}_p(-k)$. 

It is clear that any 1-dimensional $B^+_\dR$-latticed $\mbb{Q}_p$-vector space is isomorphic to $\mbb{Q}_p(k)$ for some $k$, so these give all of the one-dimensional $p$-adic Hodge structures. 
\end{example}

\subsection{Structural properties of extended real and $p$-adic Hodge structures}\label{ss.structural-properties}
We work in one of the following two situations:
\begin{enumerate}
    \item[\emph{Real}:] $K=\mbb{R}$, $C=\mbb{C}$, $X=\tilde{\mbb{P}}^1$, $B^+=\mbb{C}[[t]]$, $B=\mbb{C}((t))$, $\infty=\infty_{\mbb{C}}$.
    \item[\emph{$p$-adic}:] $K=\mbb{Q}_p$, $C/\mbb{Q}_p$ is an algebraically closed non-archimedean extension, ${X=\FF_{C^\flat}}$, $B^+=B^+_\dR$, $B=B_\dR$, $\infty=\infty_C$. 
\end{enumerate}
In the first case, by a Hodge structure we mean an extended real Hodge structure. In the second case, by a Hodge structure we mean a $p$-adic Hodge structure. We write $\HS(C)$ for the category of Hodge structures. In both cases, $\BB$ induces a functor from $B^+$-latticed $K$-vector spaces to $C$-filtered $K$-vector spaces. On extended real Hodge structures, this is the Hodge filtration. On $p$-adic Hodge structures, this is the Hodge-Tate filtration (we emphasize this distinction because the Hodge filtration will arise at a different point in the $p$-adic theory --- see \cref{s.admissible-pairs}). We caution that in both cases $\BB$ is not exact on $\HS(C)$ --- see \cref{example.ext-of-trivial}. 

\begin{lemma}\label{lemma.exact-pahs-to-vb}
The functor $(V, \mc{L}) \mapsto (V \otimes \mc{O}_X)_{\mc{L}}$ from $B^+$-latticed $K$-vector spaces to vector bundles on $X$ is exact. 
\end{lemma}
\begin{proof}
Let $f:(V,\calL) \rightarrow (V', \calL')$ be a morphism. We write
\[ T=f(V), \; \calM_1 = f(\mc{L}) \subset T_{B}, \textrm{ and } \calM_2=f(V)_{B} \cap \mc{L}' \subset T_{B}\]
so that, by assumption, $\mc{M}_1 \subseteq \mc{M}_2$. The inclusion of lattices induces an exact sequence of coherent sheaves on $X$
\[ 0 \rightarrow (T \otimes \mc{O}_X)_{\mc{M}_1} \rightarrow (T \otimes \mc{O}_X)_{\mc{M}_2} \rightarrow {\infty}_* (\mc{M}_1/\mc{M}_2) \rightarrow 0. \]
and we deduce that a morphism is strict if and only if the image of the induced map of vector bundles is saturated. The exactness is then immediate.  
\end{proof}

\begin{lemma}\label{lemma.p-adic-Hodge-abelian}
The category of Hodge structures is abelian.
\end{lemma}
\begin{proof}
It suffices to show that the category of Hodge structures of a fixed weight $\lambda$ is abelian, and the only non-trivial thing to show is that the morphisms are all strict, i.e. that the coimage is equal to the image. 

We can see this equality through slope considerations: proceeding as in the previous lemma, let $f:(V,\calL) \rightarrow (V', \calL')$ be a morphism of Hodge structures of weight $\lambda$ and write
\[ T=f(V), \; \calM_1 = f(\mc{L}) \subset T_{B}, \textrm{ and } \calM_2=f(V)_{B} \cap \mc{L}' \subset T_{B}.\]
We have $\calM_1 \subset \calM_2$ by assumption, and we want to deduce that $\calM_1 = \calM_2$. 

We have an exact sequence of coherent sheaves on $X$
\[ 0 \rightarrow (T \otimes \mc{O}_X)_{\mc{M}_1} \rightarrow (T \otimes \mc{O}_X)_{\mc{M}_2} \rightarrow {\infty}_* (\mc{M}_2/\mc{M}_1) \rightarrow 0. \]
Thus, the slope $\lambda_2$ of $(T \otimes \mc{O}_X)_{\mc{M}_2}$ is at least as large as the slope $\lambda_1$ of $(T \otimes \mc{O}_X)_{\mc{M}_1}$, with equality if and only if $\mc{M}_1=\mc{M}_2$. Now, the map $f$ realizes $(T \otimes \mc{O}_X)_{\mc{M}_1}$ as a quotient of $(V \otimes \mc{O}_X)_\calL$, which is semistable of slope $\lambda/2$, thus $\lambda_1 \geq \lambda/2$. On the other hand, $(T \otimes \mc{O}_X)_{\mc{M}_2}$ is a subbundle of $(V' \otimes \mc{O}_X)_{\mc{L'}}$, which is also semistable of slope $\lambda/2$, so its slope is $\lambda_2 \leq \lambda/2$. We conclude as we have shown
\[ \lambda/2 \leq \lambda_1 \leq \lambda_2 \leq \lambda/2. \]
\end{proof}

\begin{lemma}\label{lemma.pahs-two-three-sequence}
Suppose that
\[ 0\rightarrow (V_1, \mc{L}_1) \rightarrow (V_2, \mc{L}_2) \rightarrow (V_3, \mc{L}_3) \rightarrow 0\]
is an exact sequence of $B^+$-latticed $K$-vector spaces. If any two of the terms are Hodge structures of the same weight $\lambda$, then so is the third.
\end{lemma}
\begin{proof}
By \cref{lemma.exact-pahs-to-vb}, we obtain an exact sequence of vector bundles
\[ 0 \rightarrow (V_1 \otimes \mc{O}_X)_{\mc{L}_1} \rightarrow (V_2 \otimes \mc{O}_X)_{\mc{L}_2} \rightarrow (V_3 \otimes \mc{O}_X)_{\mc{L}_3}  \rightarrow 0 \]
The result follows as semistable bundles of a fixed slope satisfy the two out of three property in a short exact sequence. 
\end{proof}

\begin{example}\label{example.ext-of-trivial}
The $B^+$-latticed $K$-vector space $(K^2, B^+ e_1 + B^+( \frac{1}{t}e_1 + e_2))$ is a Hodge structure. Indeed, this follows from \cref{lemma.pahs-two-three-sequence} since it is an extension of the trivial Hodge structure by itself in the category of  $B^+$-latticed $K$-vector spaces. Note that the exact sequence defining this extension does not remain exact after applying the filtration functor (cf. \cref{example.latticed-K-vect-BB-not-exact}). 
\end{example}

\begin{theorem}\label{theorem.pahs-connected-tannakian} The category of $\HS(C)$ of Hodge structures is a connected neutral Tannakian category over $K$ with fiber functor $\omega_\et:(V,\mc{L}) \mapsto V$.
\end{theorem}
\begin{proof}
In \cref{lemma.p-adic-Hodge-abelian} we have shown the category is abelian, and it is immediate from the definition of the tensor product and dual that it is neutral Tannakian over $K$ (with the forgetful fiber functor to $\Vect(K)$). It remains to show it is connected: by \cite[Corollary 2.22]{deligne-milne:tannakian}, it suffices to show that if $V$ is an object such that the strictly full subcategory $\langle V \rangle_{\oplus}$ whose objects are those isomorphic to subquotients of $V^{\oplus k}$ for some $k$ is stable under tensor product then $V$ is trivial. Suppose $V$ is such an object and write $a$ and $b$ for the minimum and maximum of the set $\{ i\+ |\+ \mr{Gr}^{i}(V_C) \neq 0 \}$. Then, for any object $V'$ in  $\langle V \rangle_{\oplus}$, $\Gr^i(V'_C) \neq 0$ implies $i \in [a,b].$

Now, if $a=b=0$, then $V$ is trivial by \cref{corollary.filtration-trivial-implies-trivial} and we are done. Otherwise, either $a<0$ or $b>0$ --- the arguments in the two cases are parallel, so we treat just the case $b>0$. The filtration is a tensor functor, so $F^{2b}(V^{\otimes 2}) = (F^bV)^{\otimes 2}$ and $F^{2b + 1}(V^{\otimes 2}) = 0$. We conclude that $\Gr^{2b}(V^{\otimes 2}) \neq 0$ and thus, by the considerations of the previous paragraph, $V^{\otimes 2} \not\in \langle V \rangle_\oplus$ since $2b \not\in [a,b].$ 
\end{proof}

\subsection{Invariants of Hodge structures}\label{ss.pahs-invariants}
We continue with the notation of \cref{ss.structural-properties}. Given a Hodge structure $V$ over $C$, the \emph{Mumford-Tate} group $\MT(V) = \Aut^\otimes(\omega_\et|_{ \langle V \rangle})$ is the Tannakian structure group of the Tannakian subcategory $\langle V \rangle$ generated by $V$. It is a closed subgroup of $\GL(V)$, and by Theorem \ref{theorem.pahs-connected-tannakian} is connected.

\begin{example}\label{example.hodge-structures}\hfill
\begin{enumerate}
\item For $V$ the Hodge structure of \cref{example.ext-of-trivial}, $\MT(V)=\mbb{G}_a$. 
\item For $\mbb{Q}_p(k)$ the Tate $p$-adic Hodge structure of \cref{example.tate-p-adic-hs}, $\MT(\mbb{Q}_p(k))=\mbb{G}_m$ if $k\neq 0$ and is trivial for $k=0$. 
\end{enumerate}
\end{example}

We extend this definition to the category $G\dash\HS(C)$ of  Hodge structures with $G$-structure (see \cref{sss.tannakian-categories-notation}) for $G/K$ a connected linear algebraic group. An object of $G\dash\HS(C)$ is an exact tensor functor such that $\omega_\et \circ \mc{G}$ is isomorphic to $\omega_\std$, and we define $\MT(\mc{G})$ to be the automorphism group of $\omega_\et$ restricted to the Tannakian subcategory of $\HS(C)$ generated by the essential image of $\mc{G}$. As in \cref{ss.canonical-G-structure}, $\mc{G}$ has a canonical refinement to a Hodge structure with $\MT(\mc{G})$-structure, and a choice of identification $\triv_\et \colon \omega_\et \circ \mc{G} = \omega_\std$ identifies $\MT(\mc{G})$ with a closed subgroup of $G$. 

Additionally, there is a natural exact tensor functor $\omega_\Isoc$ from $\HS(C)$ to the category of isocrystals $\Kt_K$ of \cref{sss.isocrystals} defined as follows. Recall the simple objects $D_\lambda$ of $\Kt_K$ defined in \cref{ss.isocrystals}, and in the real case let $D_\lambda=0$ for $\lambda \not \in \frac{1}{2}\mbb{Z}$. Denote $\mc{D}_\lambda:=\End(D_{-\lambda})=\End(\mc{O}_X(\lambda))$, a division algebra. Then,
\[ \omega_\Isoc(\bigoplus_{w \in \mbb{Q}} (V_{w}, \mc{L}_{w})) = \bigoplus_{w\in \mbb{Q}}  \Hom(\mc{O}(w/2), (V_{w} \otimes \mc{O}_X)_{\mc{L}_{w}}) \otimes_{\mc{D}_{w/2}} D_{-w/2}. \]
Thus, to any $\mc{G} \in G\dash\HS(C)$, we can associate the $G$-isocrystal $\omega_\Isoc \circ \mc{G}$. Isomorphism classes of $G$-isocrystals are classified by the Kottwitz set $B(G)$ (see \cref{sss.isocrystals}) and we write $[b]_{\mc{G}}$ for the point classifying $\omega_\Isoc \circ \mc{G}$. By definition of $\HS(C)$, the slope homomorphism is central in $\MT(\mc{G})$ (up to multiplication by $-2$ it is equal to the weight homomorphism determining the $\mbb{Q}$-grading). In particular, if $\mc{G}=\MT(\mc{G})$, $[b]_{\mc{G}}$ is basic. We will often assume this latter condition or equivalently that the weight morphism is central since it can always be arranged after a reduction of structure group. 

There is also a natural exact tensor functor $\omega_\bl$ from $B^+$-latticed vector spaces to  $\Vect^\bl(B)$ (see \cref{s.filtrations-and-lattices}) defined by 
\[ \omega_{\bl}((V, \mc{L})) \mapsto (V_{B}, \mc{L}, V \otimes B^+). \]
We say $\mc{G} \in G\dash\HS(C)$ is \emph{good} if $\omega_\bl \circ \mc{G}$ is a good bilatticed $G$-bundle, in which case we write $[\mu]_\mc{G}:=[\mu]_{\omega_\bl \circ \mc{G}}$ for the classifying conjugacy class of $\overline{K}$-cocharacters.  By \cref{theorem.good-equivalence-bilatticed}, $\mc{G}$ is good if and only if $\BB \circ \mc{G}: \Rep\+ G \rightarrow \Vect^f(C)$ is a filtered $G$-bundle (i.e. is exact), and then $[\mu]_{\mc{G}}$ is also the type of this filtered $G$-bundle. 

\begin{remark} Note that one also obtains a filtered $G$-bundle of type $[\mu^{-1}]_{\mc{G}}$ through $\BB_1$, which is a filtration on $\omega_{\Isoc} \otimes C$ --- in the $p$-adic case, we will treat this perspective in \cref{s.admissible-pairs}. 
\end{remark}

\begin{example}\hfill
\begin{enumerate}
    \item If $G$ is reductive, any $\mc{G} \in G\dash\HS(C)$ is good.
\item \cref{example.ext-of-trivial} and \cref{example.latticed-K-vect-BB-not-exact} show that the $\mbb{G}_a$-Hodge structure of \cref{example.hodge-structures} is not good.
\end{enumerate}
\end{example}

When $G=\MT(\mc{G})$ or, more generally, $[b]_{\mc{G}}$ is basic, the invariant $[\mu]_{\mc{G}}$ determines $[b]_\mc{G}$: we explain this only in the $p$-adic case. Recall that Kottwitz (\cite[\S6]{kottwitz:isocrystals-II}) has defined for any connected reductive group $G/\mbb{Q}_p$ and conjugacy class $[\mu]$ of cocharacters of $G_{\overline{\mbb{Q}}_p}$ a subset $B(G, [\mu]) \subset B(G)$, and that any such subset contains a unique \emph{basic} element (i.e. an element such that the slope homomorphism is central).  We extend the definition of $B(G, [\mu])$ to any connected linear algebraic group as follows: write $U$ for the unipotent radical of $G$. As explained in \cref{sss.isocrystals-G-structure}, since isocrystals are a semisimple category, $B(G)=B(G/U)$. The projection $G\rightarrow G/U$ also identifies the conjugacy class of cocharacters of $G_{\overline{\mbb{Q}}_p}$ with those of $(G/U)_{\overline{\mbb{Q}}_p}$. Thus we may declare $B(G,[\mu])=B(G/U,[\mu])$. For $G$ non-reductive we continue to call an element of $B(G)$ basic if the slope morphism is central; thus, although $B(G/U,[\mu])$ has a unique basic element, this element may not be basic in $G$; thus $B(G,[\mu])$ either has one or zero basic elements. 

The following should also be true in the real case, but we only prove it in the $p$-adic case (due to the reference to \cite{caraiani-scholze:cohomology-compact-shimura}):
\begin{theorem}\label{theorem.hs-isocrystal-unique-basic-element}
In the $p$-adic case, if $\mc{G} \in G\dash\HS(C)$ is good and the weight morphism is central (e.g. if $G=\MT(\mc{G})$) then $[b]_{\mc{G}}$ is the unique basic element in $B(G,[\mu^{-1}]_{\mc{G}})$. 
\end{theorem}
\begin{proof}
Recall from above that the weight morphism being central is equivalent to $[b]_{\mc{G}}$ being basic. Then, for $G$ reductive, this is \cite[Prop. 3.5.3]{caraiani-scholze:cohomology-compact-shimura}: note that, because of our choice of ordering of the lattices when defining the associated bilatticed vector space, the relevant double-coset for applying \cite[Prop. 3.5.3]{caraiani-scholze:cohomology-compact-shimura} is $G(B^+_\dR)\mu^{-1}(t)G(B^+_\dR)$. In the definition of Schubert cells in \cite[pp. 683-684]{caraiani-scholze:cohomology-compact-shimura}, there is also an inverse appearing on the cocharacter, so this is exactly the double coset associated to the Schubert cell appearing in \cite[Prop. 3.5.3]{caraiani-scholze:cohomology-compact-shimura}. 

Now, because $\mc{G}$ is good, the type of $\omega_\bl \circ \mc{G}$ is the same as the type of $\omega_\bl \circ \mc{G}^\mr{ss}$ where $\mc{G}^\mr{ss}$ denotes the composed functor $\Rep\+G/U \rightarrow \Rep\+G \xrightarrow{\mc{G}} \HS(C)$ --- indeed, the map on isomorphism classes is induced by the quotient map on double cosets in \cref{s.filtrations-and-lattices}.  
\end{proof}

\subsection{Hodge-Tate lines}\label{ss.hodge-tate-lines}
In this section we assume we are in the $p$-adic case, and refer to the real case only for analogy. It would be possible to proceed symmetrically, but we wish to use the term Hodge-Tate lines and reserve the term Hodge lines for a related but distinct concept in \cref{s.admissible-pairs}. 

A powerful tool in the study of Mumford-Tate groups in classical Hodge theory is through a characterization using Hodge tensors. We now give a similar characterization in the $p$-adic setting. In classical Hodge theory one typically works in the polarizable case, but here it is necessary to adjust slightly to allow for non-reductive structure groups: indeed, because our Mumford-Tate groups are not necessarily reductive, if $\mc{G} \in G\dash\HS(C)$ for $G$ reductive, $\MT(\mc{G})$ may not be observable in $G$, i.e. it may not be realizable as the stabilizer of a vector in a representation. Thus we must consider all even integer weights instead of only weight zero tensors. 

Suppose $V$ is a $p$-adic Hodge structure. For $k \in \mbb{Z}$, the space of weight $2k$ Hodge-Tate lines in $V$ is
\[ \HT^{2k}(V):=\left(\Hom_{\HS(C)}(\mbb{Q}_p(-k), V_{2k}) \backslash \{0\}\right) / \mbb{Q}_p^\times. \]
By evaluation, $\HT^{2k}(V)$ is identified with a projective subspace of $\mbb{P}(V_{2k})$. 

\begin{theorem}\label{theorem.hodge-tate-lines-mumford-tate-group}
Suppose $G$ is a connected linear algebraic group and $\mc{G} \in G\dash\HS(C)$ is equipped with a trivialization $\omega_\et \circ \mc{G}=\omega_\std$. Then $\MT(\mc{G}) \leq G$ is the subgroup of $G$ preserving every line $\ell \in \HT^{2k}(\mc{G}(V))$ for every $V \in \Rep\+G$ and $k \in \mbb{Z}$. 
\end{theorem}
\begin{proof}
It is immediate that $\MT(\mc{G})$ stabilizes these lines, since they underlie one-dimensional sub-$p$-adic Hodge structures in $\mc{G}(V)$. On the other hand, $\MT(\mc{G})$ is a closed subgroup of $G$, so there is some line in some representation of $G$ such that $\MT(\mc{G})$ is equal to the stabilizer of that line. In particular, this line corresponds to a $1$-dimensional sub $p$-adic Hodge structure, and by \cref{example.tate-p-adic-hs}, it must be a Tate $p$-adic Hodge structure. 
\end{proof}

\newcommand{\Ad}{\mr{Ad}}
\begin{remark}\label{remark.reductivity-criterion}
In this remark we discuss further why we do not assume a condition that implies reductivity of Mumford-Tate groups. Recall that in classical Hodge theory, the Mumford-Tate groups of $\mbb{Q}$-Hodge structures arising from algebraic geometry are always reductive. This follows from the existence of a polarization: indeed, if $(V, \Fil^\bullet V_{\mbb{C}})$ is a $\mbb{Q}$-Hodge structure, then a polarization on $V$ can be defined as a bilinear form $\langle,\rangle$ on $V$ (alternating or symmetric depending on the parity of the weight) such that the map $h: \mbb{S} \rightarrow \GL(V_{\mbb{R}})$ defining the Hodge structure (for $\mbb{S}=\mr{Res}_{\mbb{C}/\mbb{R}} \mbb{G}_m$ the Deligne torus) factors through the similitude group $G$ of $\langle, \rangle$ and such that, for $c$ denoting complex conjugation on $G^\ad(\mbb{C})$, $\Ad h(i) \circ c$ is a Cartan involution for $G^\ad(\mbb{C})$ (note $h(i)$ is the Weil operator). Thus the involution $\Ad h(i) \circ c$ defines a compact real form of $\MT(V)/Z(G) \cap \MT(V)$. We conclude that the unipotent radical of this real form and thus also of the Mumford-Tate group itself is trivial, and thus the Mumford-Tate group is reductive. 

This can be reinterpreted from the perspective adopted at the beginning of this section: first, giving the inner product is equivalent to refining $V$ to a $G$-Hodge structure $\mc{G}$ where $G \subseteq \GL_V$ is the similitude group. Now, the involution $\Ad h(i) \circ c$ defines an inner form that can be checked to be equivalent to the automorphism group of the induced real $G$-isocrystal. 

This interpretation can be transported directly to the $p$-adic setting: Suppose $G/\mbb{Q}_p$ is a connected reductive group and $\mc{G} \in G\dash\HS(C)$. The automorphism group $J(\mbb{Q}_p)$ of the $G$-isocrystal $\omega_\Isoc \circ \mc{G}$ is the $\mbb{Q}_p$-points of an inner form $J$ of $G$ (the inner form $G_b$ for $b$ a representative of $[b]_{\mc{G}}$). If $J/Z_J (\mbb{Q}_p)$ is compact, then the $\mbb{Q}_p$-points of the corresponding inner form of $\MT(\mc{G})/\MT(\mc{G})\cap Z(G)$ are also compact (as a closed subgroup of a compact group), thus this inner form has no unipotent radical so neither does $\MT(\mc{G}).$ Note any choice of element $b \in G(\breve{\mbb{Q}}_p)$ representing $[b]_{\mc{G}}$ plays the role of the Weil operator $h(i)$. 

If $G=\GL_n$, then $J(\mbb{Q}_p)$ is compact mod center exactly when $b$ corresponds to an isoclinic isocrystal of slope $a/n$ for $(a,n)=1$ --- indeed, the endomorphisms of this isocrystal are the central simple algebra of Brauer class $a/n$, which is a division algebra exactly when $a$ is coprime to $n$. Recalling that $[b]_{\mc{G}}$ is uniquely determined by $[\mu]_{\mc{G}}$, this occurs exactly when, writing $[\mu]_{\mc{G}}=(a_1, a_2, \ldots, a_n)$, $\sum a_i$ is coprime to $n$. Indeed, in this case, the relation between $[\mu]$ and $b$ is just that the Newton polygon $b$ is the straight line starting at $(0,0)$ and ending at the endpoint of the Hodge polygon.  

Unfortunately, as explained in \cite[Appendix by Rapoport]{scholze:lubin-tate}, up to the obvious manipulations, this captures all possibilities. In particular, there are no compact inner forms that arise in this way from similitude groups (outside of the exceptional overlap with groups of type $A_1$), so having a bilinear pairing, e.g. from a polarization, does not help at all! 
\end{remark}

\section{Admissible pairs}\label{s.admissible-pairs}
In this section we work in the $p$-adic context (see \ref{ss.structural-properties}), so $C/\mbb{Q}_p$ is an algebraically closed non-archimedean extension. Suppose $(V,\mc{L}_\dR)$ is a $p$-adic Hodge structure and let $W=\omega_\Isoc(V)$. By construction, there is a canonical identification $W_{B^+_\dR}=\mc{L}_\dR$ and thus $V_{B_\dR}=W_{B_\dR}$. Letting $\mc{L}_\et := V_{B^+_\dR}$, we obtain a $B^+_\dR$-latticed isocrystal $(W, \mc{L}_\et)$, and we can recover $V$ as
\[ V = H^0(\FF_{C^\flat}, \mc{E}(W)_{\mc{L}_\et}). \]

This motivates the study of a semi-linear category of \emph{admissible pairs} over $C$, to be defined precisely below, and the above construction will identify $\HS(C)$ with a natural subcategory of admissible pairs. In the introduction, we explained that admissible pairs were a natural toy category of cohomological motives because, in the cohomological setting, their rationality reflects the rationality of defining equations. However, the fact that we can attach an admissible pair to any $p$-adic Hodge structure means this local, $p$-adic setting is much better behaved than its global, archimedean analog --- to any $p$-adic Hodge structure we can attach a motivic object, so we are not left to wonder which ones come from motives! 

The rest of this section develops the essential properties of admissible pairs. Many of the proofs of structural results are the same or very similar to the proofs for extended real or $p$-adic Hodge structures given in \cref{s.pahs}, so we move more quickly through the basic structural material. 

\begin{remark}
It would be possible to continue as in \cref{s.pahs} by developing a symmetric theory for both $p$-adic and extended real Hodge structures, but this quickly loses touch with the aspects that are classically interesting in real Hodge theory (one would even arrive ultimately at a trivial ``transcendence" result in this case --- see \cref{remark.real-hs-transcendence}). However, although the symmetric treatment served us well in \cref{s.pahs} as a tool for better understanding the arguments in the $p$-adic case, in this section it will be clearer to consider only the $p$-adic case and focus instead on the connections with classical notions in $p$-adic Hodge theory.
\end{remark}

\subsection{Definitions, examples, and first properties}

\begin{definition}\label{def.admpair} \hfill
\begin{enumerate}
    \item A \emph{$B^+_\dR$-latticed isocrystal} is a pair $(W, \mc{L})$ where $W$ is an isocrystal and $\mc{L} \subseteq W_{B_\dR}$ is a $B^+_\dR$-lattice.  A morphism $(W, \mc{L}_\et) \rightarrow (W', \mc{L}_\et')$ is a morphism of isocrystals $f: W \rightarrow W'$ such that $f(\mc{L}_\et) \subseteq \mc{L}'_\et$. It is strict if $f(\mc{L}_\et)=f(W)_{B_\dR} \cap \mc{L}'_\et$. A complex of $B^+_\dR$-latticed isocrystals is exact if it is exact as a complex of isocrystals and each morphism is strict. 
    \item An \emph{admissible pair} is a $B^+_\dR$-latticed isocrystal $(W, \mc{L}_\et)$ such that $\mc{E}(W)_{\mc{L}_\et}$ --- the modification of the vector bundle $\mc{E}(W)$ on $\FF_{C^\flat}$ associated to $W$ by the lattice $\mc{L}_\et$ at $\infty$ as in \cref{def.modification-vb} --- is semistable of slope zero. 
    \item An admissible pair $(W, \mc{L})$ is \emph{basic} if the isotypic decomposition $W=\bigoplus_{\lambda \in \mbb{Q}} W_\lambda$ induces a decomposition $\mc{L}=\bigoplus \mc{L}_\lambda$, $\mc{L}_\lambda := W_\lambda \otimes B_\dR \cap \mc{L}$.
\end{enumerate}
We write $\AdmPair(C)$ for the category of admissible pairs and $\bAdmPair(C)$ for the full subcategory of basic admissible pairs.
\end{definition}

The category $\AdmPair(C)$ has realizations to isocrystals and $\QQ_p$-vector spaces, which in the geometric case (\cref{example.geometric-admissible-pairs}) correspond to rational crystalline and $p$-adic \'etale cohomology. We begin by defining these realization functors and the additional structures on them, then show that $\AdmPair(C)$ is a neutral Tannakian category over $\QQ_p$. 

The isocrystalline realization
\begin{equation}
    \omega_\Isoc \colon \AdmPair(C) \to \Isoc, \qquad (W, \mc{L}) \mapsto W 
\end{equation}
is an exact tensor functor. By abuse of notation, we also denote the composition with the forgetful functor $\Isoc \to \Vect(\breve{\QQ}_p)$ by $\omega_\Isoc$. The \'etale realization is 
\begin{equation}
    \omega_\et \colon \AdmPair(C) \to \Vect(\QQ_p), \qquad (W, \mc{L}) \mapsto H^0(\FF_{C^\flat}, \mc{E}(W)_\mc{L})
\end{equation}
where $\mc{E}(W)_\mc{L}$ is the modified vector bundle on the Fargues-Fontaine curve $\FF_{C^\flat}$. This is also an exact tensor functor, moreover we will see in \cref{theorem.adm-pair-propeties} that $\AdmPair(C)$ is a neutral Tannakian category over $\QQ_p$ with fiber functor $\omega_\et$. 

If $(W, \mc{L})$ is an admissible pair, the $B_\dR$-linear extension of the restriction morphism 
\[ \omega_\et(W) = H^0(\FF_{C^\flat}, \mc{E}(W)_\mc{L}) \to H^0(\Spec\+ B_\dR, \mc{E}(W)_{\mc{L}}|_{\Spec\+ B_\dR}) = W \otimes B_\dR \]
is an isomorphism $\omega_\et(W) \otimes B_\dR \cong \omega_\Isoc(W) \otimes B_\dR$. These isomorphism are functorial and give a canonical de Rham comparison isomorphism
\begin{equation}\label{eqn:c_dR}
c_\dR \colon \omega_\et \otimes B_\dR \xrightarrow{\sim} \omega_\Isoc \otimes B_\dR
\end{equation}
between $B_\dR$-valued fiber functors on $\AdmPair(C)$.

The isocrystalline and \'etale realization functors on $\AdmPair(C)$ can be enriched to land in $B_\dR^+$-latticed vector spaces by the \'etale and de Rham lattice functors, respectively given by 
    \[ \omega_{\mc{L}_\et}(W, \mc{L}) = \mc{L}, \qquad \omega_{\mc{L}_\dR}(W, \mc{L}) = c_\dR^{-1}(W \otimes B_\dR^+). \]
\begin{remark}
We call $\mc{L}$ the \'etale lattice because it is the image of $\omega_\et(W, \mc{L}) \otimes B_\dR^+$ under $c_\dR$. The terminology `de Rham lattice' matches with the terminology in \cref{s.pahs} via the functor $\omega_\HS$ of \cref{theorem.adm-pair-propeties}; in the geometric setting (\cref{example.geometric-admissible-pairs}) the de Rham lattice is a canonical deformation of de Rham cohomology. 
\end{remark}

Using the Bialyinicki-Birula functor from $B_\dR^+$-latticed vector spaces to $C$-filtered vector spaces, the \'etale lattice gives the Hodge filtration $\Fil_\Hdg^\bullet$ on $W_C$ and the de Rham lattice gives the Hodge-Tate filtration $\Fil_\HT^\bullet$ on $\omega_\et(W, \mc{L})_C$. Specifically, for $(W, \mc{L}) \in \AdmPair(C)$, the Hodge filtration is 
    \[ \Fil_\Hdg^i = \text{Image in $W_C$ of } (W_{B_\dR^+} \cap (F^iB_\dR)\cdot \mc{L}) \]
and the Hodge-Tate filtration is
    \[ \Fil_\HT^i = \text{Image in $\omega_\et(W, \mc{L})_C$ of } (\omega_\et(W, \mc{L})_{B_\dR^+} \cap (F^iB_\dR)\cdot W_{B^+_\dR}). \]
While the Hodge and the Hodge-Tate filtration give tensor functors valued in $\Vect^f(C)$, we caution that they are not exact (by essentially the same computation as  \cref{example.ext-of-trivial}). 

\begin{lemma}\label{lemma.hodge-filtration-trivial-implies-trivial}
For an admissible pair $(W, \mc{L})$, the following are equivalent. 
\begin{enumerate}
    \item $(W, \mc{L})$ is trivial, i.e.\ isomorphic to a direct sum of $(\breve{\QQ}_p, B_\dR^+)$.
    \item The Hodge filtration on $W_C$ is trivial.
    \item The Hodge-Tate filtration on $\omega_\et(W, \mc{L})_C$ is trivial.
\end{enumerate}
\end{lemma}
\begin{proof}
The type of the Hodge-Tate filtration is the inverse of the type of the Hodge filtration, as follows from \cref{lemma.associated-gradeds-twist}, so (2) holds if and only if (3) holds. It is immediate that (1) implies (2). To see (2) implies (1), note that if the Hodge filtration is trivial then the lattice is trivial by \cref{corollary.filtration-trivial-implies-trivial}. Then $\mc{E}(W)_{\mc{L}}=\mc{E}(W)$, so admissibility implies the isocrystal $W$ is also trivial. 
\end{proof}

Before continuing with the general structure of $\AdmPair(C)$, we give some examples of admissible pairs. 

\begin{example}\label{example.tate-admissible-pairs}
Recall from \cref{ss.isocrystals} that, for $k$ in $\mbb{Z}$, $D_{-k}$ denotes the isocrystal $\breve{\mbb{Q}}_p$ with Frobenius acting by $p^{-k}$. We define the Tate admissible pair
\[ \breve{\mbb{Q}}_p(k):= (D_{-k}, \Fil^{k}B^+_\dR \cdot D_{-k}). \]
For $\omega_\HS$ as in \cref{theorem.adm-pair-propeties} below, we have a canonical identification $\omega_\HS(\breve{\mbb{Q}}_p(k))=\mbb{Q}_p(k)$, for $\mbb{Q}_p(k)$ the Tate $p$-adic Hodge structure of  \cref{example.tate-p-adic-hs}. The Hodge filtration on $\breve{\mbb{Q}}_p(k) \otimes C$ is concentrated in degree $-k$. 

Since every one-dimensional isocrystal is isoclinic, every one dimensional admissible pair is basic. It follows from \cref{example.tate-p-adic-hs} and \cref{theorem.adm-pair-propeties} that, up to isomorphism, the Tate admissible pairs give all one-dimensional admissible pairs. 
\end{example}

\begin{example}\label{example.geometric-admissible-pairs}
If $\mf{X}/\mc{O}_C$ is a smooth proper formal scheme with generic fiber $X$, a rigid analytic variety over $C$, the results of \cite{bhatt-morrow-scholze:1} provide cohomological admissible pairs $(H^i_\crys(\mf{X}_{\kappa}/W(\kappa))[1/p], \mc{L}_\et)$ where $\mc{L}_\et=H_\et^i(X, \mbb{Q}_p)\otimes B^+_\dR$. The associated $B^+_\dR$-latticed $\QQ_p$-vector space obtained by applying $(\omega_\et, \omega_{\mc{L}_\dR})$ is that of \cref{example.latticed-vect-space-geometry}.
\end{example}

The obvious tensor and dual on $B^+_\dR$-latticed isocrystals preserve admissible pairs, and we have:

\newcommand{\BKF}{\mr{BKF}}
\newcommand{\rig}{\mr{rig}}
\newcommand{\MG}{\mr{MG}}
\begin{theorem}\label{theorem.adm-pair-propeties}
$\AdmPair(C)$ is a connected neutral Tannakian category with fiber functor $\omega_\et$. It is equivalent to the isogeny category of rigidified Breuil-Kisin-Fargues modules $\BKF^\circ_{\rig}$ of \cite{anschutz}. The functor $\omega_\et^\mc{L} = (\omega_\et, \omega_{\mc{L}_\dR})$ induces an equivalence 
\[ \omega_\HS: \bAdmPair(C) \xrightarrow{\sim} \HS(C).\] 
\end{theorem}

Note that the fiber functor $\omega_\et$ on $\HS(C)$ (with the Hodge-Tate filtration) defined in the previous section is canonically identified with $\omega_\et$ (with the Hodge-Tate filtration) on basic admissible pairs via $\omega_\HS$, so that the overlap of notation will cause no confusion. 

\begin{proof}
By \cite[Theorem 3.19]{anschutz}, the category $\BKF^\circ_{\rig}$ is equivalent to the category of quadruples $(\mc{F}, \mc{F}',\beta,\alpha)$ where $\mc{F}$ and $\mc{F}'$ are vector bundles on $\FF$ with $\mc{F}$ trivial (equivalently, semistable of slope zero), $\alpha: \mc{F}|_{\FF\backslash \infty_C} \xrightarrow{\sim} \mc{F}'|_{\FF\backslash \infty_C}$, and $\beta:\bigoplus \gr^\lambda \mc{F}' \xrightarrow{\sim} \mc{F'}$ (where the graded pieces are for the slope filtration). There is a natural functor from $\AdmPair(C)$ to this category 
\[ (W, \mc{L}_\et) \mapsto (\mc{E}(W)_{\mc{L}}, \mc{E}(W), \alpha_\can, \beta_\can)\]
where $\alpha_\can$ is the isomorphism given by the modification construction and $\beta_\can$ is the canonical isomorphism of $\mc{E}(W)$ with its slope graded. It is an exercise in the definitions to verify this is an equivalence. 

The remaining properties claimed for $\AdmPair(C)$ then follow from the corresponding  properties established in \cite{anschutz}. However, we can also justify them directly as in \cref{s.pahs}. Indeed, the proof that the category is abelian is almost identical to the proof of \cref{lemma.p-adic-Hodge-abelian}, using that the modifications in each graded piece are semistable of the same slope. That it is neutral Tannakian is then clear, and the connectedness follows as in the proof of \cref{theorem.pahs-connected-tannakian} using \cref{lemma.hodge-filtration-trivial-implies-trivial}. 

It remains to see that $\omega_\et^{\mc{L}}$ induces an equivalence between basic admissible pairs and $p$-adic Hodge structures. The functor $\omega_\HS$ is given by
\[ (W, \mc{L})=\bigoplus_{\lambda \in \mbb{Q}} (W_\lambda, \mc{L}_\lambda) \mapsto \bigoplus_{w \in \mbb{Q}} \omega_\et^{\mc{L}}(W_{-w/2}, \mc{L}_{-w/2}) \]
and we have described the inverse at the beginning of this section. 
\end{proof}

\subsection{$G$-structure and motivic Galois groups of admissible pairs}
For $G/\mbb{Q}_p$ a connected linear algebraic group, recall (from \cref{ss.tannakian}) the category $G\dash\AdmPair(C)$ of admissible pairs with $G$-structure (or $G$-admissible pairs for short) has objects exact tensor functors $\mc{G}: \Rep\+ G \rightarrow \AdmPair(C)$ such that $\omega_\et \circ \mc{G}$ is isomorphic to $\omega_\std$. A $G$-admissible pair is basic if it factors through $\bAdmPair(C)$.

\begin{definition}
If $W$ is an admissible pair, then the \emph{motivic Galois group} of $W$, $\MG(W)$, is the automorphism group of $\omega_\et|_{\langle W \rangle}$. 
\end{definition}

Any admissible pair $W$ has a canonical $\MG(W)$-structure. More generally, for any $\mc{G} \in G\dash\AdmPair(C)$, we define $\MG(\mc{G})$ to be the automorphism group of $\omega_\et$ restricted to the Tannakian subcategory generated by the essential image of $\mc{G}$. Then, as in \cref{ss.canonical-G-structure}, $\mc{G}$ has a canonical refinement to a $\MG(\mc{G})$-admissible pair, and a choice of isomorphism $\omega_\et \circ \mc{G}=\omega_\std$ identifies $\MG(\mc{G})$ with a closed subgroup of $G$. If $\mc{G}$ is basic, then $\MG(\mc{G})$ is canonically identified with $\MT(\omega_\HS \circ \mc{G})$, the Mumford-Tate group of the associated $G$-$p$-adic Hodge structure, via \cref{theorem.adm-pair-propeties}. For any $\mc{G} \in G\dash\AdmPair(C)$, the isomorphism class of $\omega_\Isoc \circ \mc{G}$ is classified by an element of $B(G)$, denoted $[b]_{\mc{G}}$.

\begin{example} A $G$-admissible pair $\mc{G}$ factors through $\AdmPair(C)^{\basic}$ if and only the slope morphism is central in the motivic Galois group; if $G=\MG(\mc{G})$, this is equivalent to $[b]_{\mc{G}}$ being basic. A good example to keep in mind is an admissible pair arising as an extension of $\breve{\mbb{Q}}_p$ by $\breve{\mbb{Q}}_p(1)$. The trivial extension is basic, with motivic Galois group $\mbb{G}_m$, and the others are non-basic and have motivic Galois group $\mbb{G}_m \ltimes \mbb{G}_a$. If we view these extensions as arising from elliptic curves in a Serre-Tate disk, then only the lifts isogenous to the canonical lift (i.e. with Serre-Tate coordinate a root of unity) give rise to basic admissible pairs (see \cref{example.intro-elliptic-curves}). 
\end{example}

We write $\omega_\bl$ for the functor to bilatticed $B_\dR$-vector spaces 
\[ (W, \mc{L}_\et) \mapsto (W_{B_\dR}, \omega_{\mc{L}_\dR}(W, \mc{L}) \cong W_{B^+_\dR}, \mc{L}_\et) \]
It is an exact tensor functor, and extends the functor $\omega_\bl$ on $p$-adic Hodge structures under the equivalence of \cref{theorem.adm-pair-propeties}. 
Note that the Hodge filtration on $W_C$ is given by $\BB_1 \circ \omega_{\bl}$, while the Hodge-Tate filtration on $\omega_{\et}(W)\otimes_{\mbb{Q}_p}C = \mc{L}_\et \otimes_{B^+_\dR} C $ is given by $\BB_2 \circ \omega_\bl$. 

\begin{definition}
If $\mc{G} \in G\dash\AdmPair(C)$, we say $\mc{G}$ is good if $\omega_{\bl}\circ \mc{G}$ is a good bilatticed $G$-bundle as in Definition \ref{def:good-bi-latticed-torsor}, and write $[\mu]_{\mc{G}}$ for the associated type.  
\end{definition}

By \cref{theorem.good-equivalence-bilatticed}, $\mc{G}$ is good if and only if the Hodge filtration (resp. Hodge-Tate filtration) composed with $\mc{G}$, $\Rep\+ G \rightarrow \Vect^f(C)$ is a filtered $G$-bundle (i.e.\ is exact), 
and then $[\mu^{-1}]_{\mc{G}}$ (resp. $[\mu]_{\mc{G}}$) is the type of the Hodge (resp. Hodge-Tate) filtered $G$-bundle. If $\mc{G}$ is not basic then the invariant $[b]_{\mc{G}}$ is no longer uniquely determined by $[\mu]_{\mc{G}}$, but the proof of \cref{theorem.hs-isocrystal-unique-basic-element} extends immediately to show
\begin{theorem}\label{theorem.good-ap-kottwitz-set}
If $G$ is a connected linear algebraic group and $\mc{G} \in G\dash\AdmPair(C)$ is good, then $[b]_{\mc{G}}$ is contained in $B(G,[\mu^{-1}]_{\mc{G}})$. 
\end{theorem}

\subsection{Hodge lines}\label{ss.hodge-lines}

Suppose $W$ is an admissible pair. For $k \in \mbb{Z}$, the space of weight $2k$ Hodge lines in $W$ is
\[ \Hdg^{2k}(W):=\left(\Hom_{\AdmPair(C)}(\breve{\mbb{Q}}_p(-k), W) \backslash \{0\}\right) / \mbb{Q}_p^\times. \]
By evaluation, $\Hdg^{2k}(W)$ is a projective subspace of $\mbb{P}(W^{\varphi_W={p^k}})$, where we note 
\[ W^{\varphi_W=p^k}=\Hom_\Isoc(D_{k}, W) \] 
is a finite dimensional $\mbb{Q}_p$-vector space (recall $D_{k}$ from \cref{ss.isocrystals} denotes the 1-dimensional isocrystal $\breve{\mbb{Q}}_p$ with Frobenius acting by $p^{k}$). 

We also may identify $\Hdg^{2k}(W)$ with a projective subspace of $\mbb{P}(\omega_\et(W))$ by evaluation after application of $\omega_\et$. If $W$ is basic with associated $p$-adic Hodge structure $V$, $\Hdg^{2k}(W)=\HT^{2k}(V)$ compatibly with this identification, and essentially the same proof as \cref{theorem.hodge-tate-lines-mumford-tate-group} yields an extension to all admissible pairs: 

\begin{theorem}
Suppose $G/\mbb{Q}_p$ is a connected linear algebraic group and $\mc{G} \in G\dash\AdmPair(C)$ is equipped with a trivialization $\omega_\et \circ \mc{G}=\omega_\std$. Then $\MG(\mc{G}) \leq G$ is the subgroup of $G$ preserving every line $\ell \in \Hdg^{2k}(\mc{G}(V))$ for each $V \in \Rep\+G, k \in \mbb{Z}$. 
\end{theorem}

\begin{remark}
Let $W \in \AdmPair(C)$ and let $V=\omega_\et(W)$. When we identify $\Hdg^{2k}(W)$ with a projective subspace of $\mbb{P}(V)$, it lies inside the projectivization of 
\[ \Hom_{\textrm{$B^+_\dR$-latticed $\mbb{Q}_p$-vector spaces}}(\mbb{Q}_p(-k), (V, W \otimes B^+_\dR)). \]
This is an equality when $W$ is $k$-isotypic, but in general this set of homomorphisms can be larger, so that the Hodge lines cannot typically be determined only using the data of the $B^+_\dR$-latticed $\mbb{Q}_p$-vector space attached to a general admissible pair. 
\end{remark}

\subsection{$\overline{C}_0$-analyticity of an admissible pair and its de Rham lattice}\label{ss.admissible-pairs-with-good-reduction}
\newcommand{\MF}{\mr{MF}}

In this subsection we give precise definitions of the rationality properties of the lattices associated to an admissible pair, as needed for the main theorem. The rationality of the \'etale lattice determines the field of definition of an admissible pair, and we consider the subcategory consisting of those objects that can be defined over some strict $p$-adic subfield of $C$ (recall a non-archimedean field is strict $p$-adic if its value group is discrete and its residue field is algebraically closed). Note that any strict $p$-adic subfield is contained in $\overline{C}_0$ and contains $\breve{\QQ}_p$. 

\begin{definition} \label{def.rigid-analytic-adm-pair}
Let $(W, \mc{L})$ be an admissible pair over $C$ and let $V = \omega_\et(W, \mc{L})$.
\begin{enumerate}
    \item The de Rham lattice $\omega_{\mc{L}_\dR}(W, \mc{L})$ of $(W, \mc{L})$ is \emph{$\overline{C}_0$-analytic} if there is a filtration $\Fil^\bullet$ on $V_{\ol{C}_0}$ such that $\omega_{\mc{L}_\dR}(W, \mc{L}) = \mc{L}_\can(\Fil^\bullet V_{\ol{C}_0}) = \sum_{i \in \ZZ} \Fil^{-i}V_{\ol{C}_0} \otimes \Fil^i B_\dR$. For $K \subseteq C$ a strict $p$-adic subfield, we say that the de Rham lattice is defined over $K$ (or is $K$-analytic) if $\Fil^\bullet$ is defined over $K$.
    \item $(W, \mc{L})$ is \emph{$\overline{C}_0$-analytic} if there is a filtration $\Fil^\bullet$ on $W_{\ol{C}_0}$ such that $\mc{L} = \mc{L}_\can(\Fil^\bullet W_{\ol{C}_0}) = \sum_{i \in \ZZ} \Fil^{-i}W_{\ol{C}_0} \otimes_{\ol{C}_0} \Fil^i B_\dR$. For $K \subseteq C$ a strict $p$-adic subfield, a $\overline{C}_0$-analytic admissible pair $(W,\mc{L})$ has \emph{good reduction over $K$} if this filtration is defined over $K$. 
\end{enumerate}

We write $\AdmPair(\overline{C}_0)$ for the full subcategory of $\AdmPair(C)$ consisting of $\overline{C}_0$-analytic admissible pairs, and an admissible pair with $G$-structure $\mc{G}$ is $\overline{C}_0$-analytic  if it factors through $\AdmPair(\overline{C}_0)$. For $K \subseteq C$ a strict $p$-adic subfield, we write $\AdmPair^\goodred(K)$ for the full subcategory of $\AdmPair(C)$ consisting of admissible pairs with good reduction over $K$, and an admissible pair with $G$-structure $\mc{G}$ has good reduction over $K$ if it factors through $\AdmPair^\goodred(K)$. 
\end{definition}

\begin{remark}\label{remark.more-general-def-over-p-adic-field}
In Part III we will give a definition of admissible pairs over an arbitrary locally spatial diamond. For the diamond $\Spd K$, $K$ a $p$-adic field, this will amount to an admissible pair $(W, \mc{L})$ over $C=\overline{K}^\wedge$ equipped with a semi-linear action of $\Gal(\overline{K}/K)$ on $W$ that preserves $\mc{L}$ for the induced semi-linear action on $W_{B_\dR}$. An admissible pair with good reduction over $K$ will then be one where the semi-linear action is unramified; if the residue field of $K$ is algebraically closed, then the semi-linear action is trivial and we recover the notion above. We make the definition above only in the case of algebraically closed residue field to avoid introducing this semi-linear action, which is irrelevant for our transcendence results. 
\end{remark}

Note that any $\overline{C}_0$-analytic admissible pair has good reduction over a sufficiently large strict $p$-adic subfield $K \subseteq C$. In particular, if $\mc{G}$ is a $\overline{C}_0$-analytic admissible pair with $G$-structure, then since $\Rep\+ G$ has a tensor generator, one can always find a strict $p$-adic subfield $K \subseteq C$ such that $\mc{G}$ has good reduction over $K$.

On the category of admissible pairs with good reduction, there is a canonical splitting of the Hodge-Tate filtration. Below we write $\mathfrak{G}_K=\Gal(\overline{K}/K)$. 
\begin{lemma}\label{lemma.ht-splitting}
Let $(W, \mc{L}) \in \AdmPair(\overline{C}_0)$, and let $V=\omega_\et(W,\mc{L})$. 
\begin{enumerate}
    \item If $\mc{L} = \mc{L}_\can(\Fil^\bullet W_{\ol{C}_0})$ for a filtration $\Fil^\bullet$ on $W_{\ol{C}_0}$, then 
        \[ \Fil^\bullet W_{\ol{C}_0} = \varinjlim_{K \subset \ol{C}_0} \left(\Fil_\Hdg^\bullet W_C\right)^{\mf{G}_K}, \]
    where $K$ ranges over discretely valued subfields of $\ol{C}_0$. In particular, $\Fil^\bullet W_{\ol{C}_0}$ is uniquely determined by $(W, \mc{L})$ and recovered via the Hodge filtration. 
    \item For each $i$, the natural map $F^{-i}W_{\ol{C}_0} \otimes \Fil^i B_\dR \subseteq \mc{L} \cong V \otimes B_\dR^+$ yields a commutative diagram
%\[ \Fil^{-i} W_{\ol{C}_0} \otimes \Fil^{i}B_\dR \rightarrow \mc{L} \xrightarrow[\sim]{c_\dR^{-1}}V \otimes B_\dR^+ \]
%https://q.uiver.app/#q=WzAsNCxbMCwwLCJcXEZpbF57LWl9V197XFxvbHtDfV8wfSBcXG90aW1lcyBcXEZpbF5pIEJfXFxkUiJdLFsxLDAsIlYgXFxvdGltZXMgQl9cXGRSXisiXSxbMSwxLCJWIFxcb3RpbWVzIEMiXSxbMCwxLCJcXGdyXnstaX0gV197XFxvbHtDfV8wfSBcXG90aW1lcyBDKGkpIl0sWzEsMl0sWzAsM10sWzAsMSwiIiwxLHsic3R5bGUiOnsidGFpbCI6eyJuYW1lIjoiaG9vayIsInNpZGUiOiJ0b3AifX19XSxbMywyLCIiLDEseyJzdHlsZSI6eyJ0YWlsIjp7Im5hbWUiOiJob29rIiwic2lkZSI6InRvcCJ9fX1dXQ==
\[\begin{tikzcd}[sep=small]
	{\Fil^{-i}W_{\ol{C}_0} \otimes \Fil^i B_\dR} & {V \otimes B_\dR^+} \\
	{\gr^{-i} W_{\ol{C}_0} \otimes C(i)} & {V \otimes C}
	\arrow[from=1-2, to=2-2]
	\arrow[from=1-1, to=2-1]
	\arrow[hook, from=1-1, to=1-2]
	\arrow[hook, from=2-1, to=2-2]
\end{tikzcd}\]
which induces an isomorphism
\[ \gr^{-i} W_{\ol{C}_0} \otimes_{\overline{C}_0} C(i) \xrightarrow{\sim} \gr^{i} V_C. \]
\item The previous isomorphisms yield a functorial splitting of the Hodge-Tate filtration on $\AdmPair(\overline{C}_0)$
\[ (W, \mc{L}) \mapsto \bigoplus_{i \in \ZZ} \gr^{-i}W_{\ol{C}_0} \otimes C(i) \xrightarrow{\sim} V_C  \]

\end{enumerate}
\end{lemma}
\begin{proof}
Functoriality in (3) follows from (1) since the Hodge filtration is functorial. It is immediate from the definitions $\mc{L}_\can(\Fil^\bullet) = \sum_{i \in \ZZ} \Fil^{-i} W_{\ol{C}_0} \otimes \Fil^i B_\dR$ and $\Fil_\Hdg^\bullet = (t^\bullet\mc{L} \cap W_{B_\dR^+})/(t^\bullet\mc{L} \cap tW_{B_\dR^+})$ that $\Fil_\Hdg^\bullet = \Fil^\bullet \otimes C$ and (1) follows by taking the colimit over Galois invariants. (2) then follows from (1) and the following diagram with exact rows and columns:
% https://q.uiver.app/#q=WzAsMjEsWzEsMSwidF57aSsxfVdfe0JfXFxkUl4rfSBcXGNhcCB0XFxtY3tMfSJdLFsyLDEsInRee2krMX1XX3tCX1xcZFJeK30gXFxjYXAgXFxtY3tMfSJdLFsxLDIsInRee2l9V197Ql9cXGRSXit9IFxcY2FwIHRcXG1je0x9Il0sWzIsMiwidF57aX1XX3tCX1xcZFJeK30gXFxjYXAgXFxtY3tMfSJdLFszLDEsIlxcRmlsX1xcSFRee2krMX0gVl9DIl0sWzMsMiwiXFxGaWxeaV9cXEhUIFZfQyJdLFsxLDMsIlxcRmlsXnstaSArIDF9X1xcSGRnIFdfQyBcXG90aW1lcyBDKGkpIl0sWzIsMywiXFxGaWxeey1pfV9cXEhkZyBXX0MgXFxvdGltZXMgQyhpKSJdLFszLDMsIlxcZ3JeaVZfQyJdLFsxLDAsIjAiXSxbMiwwLCIwIl0sWzMsMCwiMCJdLFs0LDEsIjAiXSxbNCwyLCIwIl0sWzQsMywiMCJdLFszLDQsIjAiXSxbMiw0LCIwIl0sWzEsNCwiMCJdLFswLDMsIjAiXSxbMCwyLCIwIl0sWzAsMSwiMCJdLFswLDJdLFsyLDZdLFsxLDNdLFszLDddLFs2LDddLFsyLDNdLFswLDFdLFsxLDRdLFszLDVdLFs3LDhdLFsyMCwwXSxbOSwwXSxbMTAsMV0sWzExLDRdLFsxOSwyXSxbMTgsNl0sWzYsMTddLFs3LDE2XSxbNCw1XSxbNSw4XSxbOCwxNV0sWzgsMTRdLFs1LDEzXSxbNCwxMl1d
\[\begin{tikzcd}[sep=small]
	& 0 & 0 & 0 \\
	0 & {t^{i+1}W_{B_\dR^+} \cap t\mc{L}} & {t^{i+1}W_{B_\dR^+} \cap \mc{L}} & {\Fil_\HT^{i+1} V_C} & 0 \\
	0 & {t^{i}W_{B_\dR^+} \cap t\mc{L}} & {t^{i}W_{B_\dR^+} \cap \mc{L}} & {\Fil^i_\HT V_C} & 0 \\
	0 & {\Fil^{-i + 1}_\Hdg W_C \otimes C(i)} & {\Fil^{-i}_\Hdg W_C \otimes C(i)} & {\gr^iV_C} & 0 \\
	& 0 & 0 & 0
	\arrow[from=2-2, to=3-2]
	\arrow[from=3-2, to=4-2]
	\arrow[from=2-3, to=3-3]
	\arrow[from=3-3, to=4-3]
	\arrow[from=4-2, to=4-3]
	\arrow[from=3-2, to=3-3]
	\arrow[from=2-2, to=2-3]
	\arrow[from=2-3, to=2-4]
	\arrow[from=3-3, to=3-4]
	\arrow[from=4-3, to=4-4]
	\arrow[from=2-1, to=2-2]
	\arrow[from=1-2, to=2-2]
	\arrow[from=1-3, to=2-3]
	\arrow[from=1-4, to=2-4]
	\arrow[from=3-1, to=3-2]
	\arrow[from=4-1, to=4-2]
	\arrow[from=4-2, to=5-2]
	\arrow[from=4-3, to=5-3]
	\arrow[from=2-4, to=3-4]
	\arrow[from=3-4, to=4-4]
	\arrow[from=4-4, to=5-4]
	\arrow[from=4-4, to=4-5]
	\arrow[from=3-4, to=3-5]
	\arrow[from=2-4, to=2-5]
\end{tikzcd}\]
\end{proof}

As a consequence, we find $\overline{C}_0$-analytic admissible pairs avoid some of the perversities one encounters regarding types for general admissible pairs: 

\begin{theorem}\label{theorem.rig-an-ap-is-good}
Any $\overline{C}_0$-analytic admissible pair with $G$-structure is good. 
\end{theorem}
\begin{proof}
The splitting in \cref{lemma.ht-splitting} is functorial, so the Hodge-Tate filtation on $\omega_\et \circ \mc{G}$ is exact and thus the result follows from \cref{main.exactness}.
\end{proof}

\begin{example}\label{example.no-non-trivial-rigid-analytic-ext-of-trivial-by-itself}
Building on \cref{example.ext-of-trivial}, one easily finds that 
\[ \mr{Ext}^1_{\AdmPair(C)}(\triv, \triv)=B_\dR/B^+_\dR \]
and that the only good extension (viewed as a $\mbb{G}_a$-admissible pair) is the trivial one corresponding to the zero coset. By \cref{theorem.rig-an-ap-is-good} we conclude that 
\[ \mr{Ext}^1_{\AdmPair(\overline{C}_0)}(\triv, \triv)=0. \]
This is an incarnation of the fact that a self-extension of the trivial representation of the Galois group of a $p$-adic field is crystalline if and only if it is is unramified (see \cref{ss.galois-representation} below). 
\end{example}

\begin{example}[Existence of $G$-admissible pairs with fixed invariants]\label{example.existence-of-G-admissible-pairs} Combining \cref{theorem.rig-an-ap-is-good} and \cref{theorem.good-ap-kottwitz-set}, we find a $\overline{C}_0$-analytic $G$-admissible pair is good with invariants $[\mu]$ and $b$ such that $b \in B(G,[\mu^{-1}])$. We claim furthermore that given any $b \in B(G,[\mu^{-1}])$, there exists a $\overline{C}_0$-analytic $G$-admissible pair with invariants $b$ and $[\mu]$ (so far it is not obvious that there is any $G$-admissible pair with these fixed invariants!). Indeed, choosing a Levi decomposition $G=MU$, we can always push-out from $M$ to $G$, so it suffices to assume $G$ is reductive. In the reductive case, this existence follows from \cite[Proposition 3.1]{rapoport-viehmann}. 
\end{example}

\subsection{Galois representations for admissible pairs with good reduction}\label{ss.galois-representation}
In this subsection we take $K$ to be a strict $p$-adic field and assume $C=\overline{K}^\wedge$. Then $\omega_\et|_{\AdmPair^{\goodred}(K)}$ promotes naturally to a functor $\omega_{\mf{G}_K}: \AdmPair^{\goodred}(K) \rightarrow \Rep_{\mbb{Q}_p} \mf{G}_K$. Indeed, for $(W, \mc{L})$ in $ \AdmPair^{\goodred}(K)$ the Hodge filtration of $W_C$ is defined over $K$ and $\mc{L} = \sum_{i \in \ZZ} \Fil_\Hdg^{-i}W_K \otimes_K \Fil^i B_\dR$ is preserved by the $\mf{G}_K$ action on $W_{B_\dR}= W\otimes B_\dR$ (acting on $B_\dR$). Consequently there is an induced action $\rho$ of $\mf{G}_K$ on $\omega_{\et}(W,\mc{L})=H^0(\FF, \mc{E}(W)_{\mc{L}})$, which is encoded by the following commutative diagram for each $\sigma \in \mf{G}_K$:

\begin{equation}\label{eqn:galois-rep-admissible-phi-module}
% https://q.uiver.app/#q=WzAsNCxbMSwwLCJXIFxcb3RpbWVzIEJfXFxkUiA9IFxcb21lZ2FfXFxJc29jKFcsIFxcbWN7TH0pIFxcb3RpbWVzIEJfXFxkUiJdLFswLDAsIlxcb21lZ2FfXFxldChXLCBcXG1je0x9KSBcXG90aW1lcyBCX1xcZFIiXSxbMCwxLCJcXG9tZWdhX1xcZXQoVywgXFxtY3tMfSkgXFxvdGltZXMgQl9cXGRSIl0sWzEsMSwiVyBcXG90aW1lcyBCX1xcZFIgPSBcXG9tZWdhX1xcSXNvYyhXLCBcXG1je0x9KSBcXG90aW1lcyBCX1xcZFIiXSxbMCwzLCJcXElkIFxcb3RpbWVzIFxcc2lnbWEiXSxbMSwyLCJcXHJobyhcXHNpZ21hKSBcXG90aW1lcyBcXHNpZ21hIiwyXSxbMiwzLCJjX1xcZFIiXSxbMSwwLCJjX1xcZFIiXV0=
\begin{tikzcd}
	{\omega_\et(W, \mc{L}) \otimes B_\dR} & {W \otimes B_\dR = \omega_\Isoc(W, \mc{L}) \otimes B_\dR} \\
	{\omega_\et(W, \mc{L}) \otimes B_\dR} & {W \otimes B_\dR = \omega_\Isoc(W, \mc{L}) \otimes B_\dR}
	\arrow["{\mr{Id} \otimes \sigma}", from=1-2, to=2-2]
	\arrow["{\rho(\sigma) \otimes \sigma}"', from=1-1, to=2-1]
	\arrow["{c_\dR}", from=2-1, to=2-2]
	\arrow["{c_\dR}", from=1-1, to=1-2]
\end{tikzcd}
\end{equation}

The following is essentially the crystalline comparison as reinterpreted by Fargues and Fontaine in this special case (see \cref{remark.cryst-comparison} below). 
\begin{theorem}\label{theorem.full-faithful-essential-image}
The functor $\omega_{\mf{G}_K}$ is fully faithful; its essential image is stable under sub-objects and is contained in the subcategory of crystalline representations of $\mf{G}_K$. 
\end{theorem}
\begin{proof}
Let $(W,\mc{L}) \in \AdmPair^\goodred(K)$. Let $V=\omega_\et(W,\mc{L})$ and let $\rho$ be the Galois representation on $V$. We first explain how to reconstruct $(W,\mc{L})$ from $(V,\rho)$.

Because $V \otimes \mc{O}_\FF|_{\FF \backslash \infty_C}=\mc{E}(W)|_{\FF \backslash \infty_C}$, we have by construction an isomorphism $V_{B_\crys} = W_{B_\crys}$. Taking $\mf{G}_K$-invariants thus recovers $W_{K_0}$ (showing that $V$ is crystalline), compatibly with the natural Frobenius action on both sides (with $K_0=B_\crys^{\mf{G}_K}$). Since $K_0$ has an algebraically closed residue field, the Dieudonn\'{e}-Manin classification for Frobenius modules over $K_0$ shows $W=\sum_{\lambda=\frac{a}{b}} \breve{\mbb{Q}}_p \cdot (W_{K_0})^{\varphi^b = p^a}$, so we recover $W$. Once we have $W$, we recover $\mc{L}$  as $\mc{L}=V_{B^+_\dR} \subseteq V_{B_\dR}=W_{B_\dR}$.

From this it is clear that $\omega_{\mf{G}_K}$ is fully faithful --- indeed, any endomorphism of $W$ or $V$ is determined by its action on $W_{B_\dR} = V_{B_\dR}$. Moreover, if $S \subseteq V$ is a subspace preserved by the Galois action, then the exact sequence
\[ 0 \rightarrow (S_{B_\crys})^{\mf{G}_K} \rightarrow (V_{B_\crys})^{\mf{G}_K} = W_{K_0} \rightarrow ((V/S)_{B_\crys})^{\mf{G}_K}  \]
plus the equality $\dim_{\mbb{Q}_p} V = \dim_{\breve{\mbb{Q}}_p} W$ and the inequalities $\dim_{K_0}(S_{B_\crys})^{\mf{G}_K} \leq \dim_{\mbb{Q}_p} S$ and  $\dim_{K_0} ((V/S)_{B_\crys})^{\mf{G}_K} \leq \dim_{\mbb{Q}_p} V/S$ show that $(S_{ B_\crys})^{\mf{G}_K}$ is a sub $\varphi$-module of $W_{K_0}$ with the same dimension as $S$. Passing to the associated isocrystal $W' \subset W$ as above, we find $W'_{B_\dR} = S_{B_\dR}$ and thus $(W',\mc{L}\cap W'_{B_\dR})=(W', S_{B^+_\dR})$ is a sub-object corresponding to $S$. 
\end{proof}

\begin{remark}\label{remark.cryst-comparison}
    Of course, the essential image of this functor is the category of crystalline representations of $\mf{G}_K$, and the construction in the proof explains also the relation to Fontaine's category of filtered $\varphi$-modules over $K_0$ when we take into account that the $\mf{G}_K$-invariant lattices on $W_{B_\dR}$ are precisely the ones in the image of $\mc{L}_\mr{can}$. We do not discuss this further here as it is not necessary for our present purposes; in Part III we will explain this equivalence more generally in the context of \cref{remark.more-general-def-over-p-adic-field} where we allow also $K$ with non-algebraically closed residue field. 
\end{remark}

If $\mc{G}$ is a $G$-admissible pair with good reduction over $K$, then we obtain a continuous representation $\rho: \mf{G}_K \rightarrow G(\mbb{Q}_p)$ after choosing a trivialization $\omega_\et \circ \mc{G} \cong \omega_\std$. \cref{theorem.full-faithful-essential-image} implies that the induced map from the Tannakian structure group of $\omega_{\mf{G}_K} \circ \mc{G}$ to $\MG(\mc{G})$ is an isomorphism. The former is the Zariski closure of the compact subgroup $\rho(\mf{G}_K)$ in $G(\mbb{Q}_p)$, thus the image of $\rho$ is Zariski-dense in $\MG(\mc{G})$. In fact, we can do slightly better: 

\begin{corollary}\label{corollary.galois-rep-dense-open}
Notation as above, $\rho(\mf{G}_K)$ is an open subgroup of $\MG(\mc{G})(\mbb{Q}_p)$. 
\end{corollary}
\begin{proof}
\cref{lemma.ht-splitting} implies $\rho$ is a Hodge-Tate representation. By \cite[Th\'{e}or\`{e}me 1]{serre:groupes-alg-ht}, its image is open in its Zariski closure, which, by the above discussion, is $\MG(\mc{G})$. 
\end{proof}

We record for later use the following remark regarding a Galois theoretic consequence of the rationality of the de Rham lattice. 

\begin{lemma}\label{remark:de-Rham-lattice-Galois-equivariant}
If $\mc{G}$ is a $G$-admissible pair with good reduction over $K$, \emph{and} the de Rham lattice $\omega_{\mc{L}_\dR} \circ \mc{G}$ on $\omega_\et \circ \mc{G}$ is also defined over $K$, then the de Rham lattice $\omega_{\mc{L}_\dR} \circ \mc{G} \subset (\omega_\et \otimes B_\dR) \circ \mc{G}$ is preserved by the $B_\dR$-linear extension $\rho \otimes \mr{Id}_{B_\dR}$ of the action of $\mf{G}_K$ on $\omega_\et$. %natural action of $\mf{G}_K$ on lattices on $\omega_\et \circ \mc{G}$ through the action on $B_\dR$ (concretely, $\sigma \in \mf{G}_K$ acts by $1 \otimes \sigma$ on $\omega_\et \otimes B_\dR$). 
\end{lemma}
\begin{proof}
By definition there is an exact $K$-filtration $\Fil$ on $\omega_\et \circ \mc{G}$ such that $\omega_{\mc{L}_\dR} \circ \mc{G} = \mc{L}_\can \circ \Fil$. Consequently $\mc{L}_\can \circ \Fil = \sum_{i \in \ZZ} \Fil^{-i} (\omega_\et \circ \mc{G}) \otimes_K \Fil^i B_\dR$ is visibly preserved by $\mr{Id}_{\omega_\et \circ \mc{G}} \otimes \sigma^{-1}$ for $\sigma \in \mf{G}_K$; while \eqref{eqn:galois-rep-admissible-phi-module} shows it is preserved by $\rho(\sigma) \otimes \sigma$. Thus it is preserved by $\rho(\sigma) \otimes \mr{Id}_{B_\dR}$.

\end{proof}

\subsection{Periods}\label{ss.periods}
We now describe some period constructions. We note that this kind of analysis goes back at least to Fargues \cite{fargues:two-towers} in the Lubin-Tate case. We write $\omega_\HT, \omega_\Hdg$ for the Hodge-Tate and Hodge-Tate filtrations on $\omega_\et, \omega_\Isoc$ respectively, and view them as tensor functors $\AdmPair(C) \to \Vect^f(C)$ (which are not exact, but are exact when restricted to $\AdmPair(\ol{C}_0)$). 
\iffalse
\begin{enumerate}
    \item $\omega_\HT: \AdmPair(C) \rightarrow \Vect^f(C)$ for the Hodge-Tate filtration on $\omega_\et \otimes C$ (this is a tensor functor that is not exact, but has exact restriction to the subcategory of $\overline{C}_0$-analytic admissible pairs),
    \item $\omega_\Hdg: \AdmPair(C) \rightarrow \Vect^f(C)$ for the Hodge filtration on $\omega_\Isoc \otimes C$ (this a tensor functor that is not exact but has exact restriction to the subcategory of $\overline{C}_0$-analytic admissible pairs)
    \item $\omega_{\mc{L}_\dR}: \AdmPair(C) \rightarrow \Vect^{B_\dR^+\dash\mr{latticed}}(\mbb{Q}_p)$ for the de Rham lattice ${c_\dR^{-1}(\omega_\Isoc \otimes B_\dR^+)}$ on $\omega_\et$, and  
    \item $\omega_{\mc{L}_\et}: \AdmPair(C) \rightarrow \Vect^{B_\dR^+\dash\mr{latticed}}(\breve{\mbb{Q}}_p)$ for the \'{e}tale lattice ${c_\dR(\omega_\et \otimes B_\dR^+)}$ on $\omega_\Isoc$\footnote{For an admissible pair $(W, \mc{L})$, the \'etale lattice $\omega_{\mc{L}_\et}(W, \mc{L})$ is canonically identified with $\mc{L}$.}.
\end{enumerate}
\fi

Let $G$ be a connected linear algebraic group over $\mbb{Q}_p$ and fix an element $b \in G(\breve{\mbb{Q}}_p)$. Let $\mc{G}_b$ be the $G$-isocrystal $(V,\rho) \mapsto (V \otimes_{\mbb{Q}_p} \breve{\mbb{Q}}_p, \rho(b) \varphi_{\breve{\mbb{Q}}_p})$. For $\mc{G}$ a $G$-admissible pair with underlying isocrystal of type $[b] \in B(G)$, fix trivializations
\[ \triv_\Isoc: \omega_\Isoc \circ \mc{G} \cong \mc{G}_b \textrm{ and } \triv_\et: \omega_\et \circ \mc{G} \cong \omega_\std.\]
Note that $\triv_\Isoc$ also induces a isomorphism $\mr{Forget} \circ \omega_\Isoc \circ \mc{G} \cong \omega_\std \otimes \breve{\QQ}_p$; we also call this $\triv_\Isoc$ and allow $\omega_\Isoc \circ \mc{G}$ to be valued in either isocrystals or $\breve{\QQ}_p$-vector spaces depending on the context. Then 
\[ \omega_\std \otimes B_\dR =_{\triv_\et} \omega_\et \circ \mc{G} \otimes B_\dR \xrightarrow{c_\dR} \omega_\Isoc \circ \mc{G} \otimes B_\dR =_{\triv_\Isoc} \omega_\std \otimes B_\dR \]
is given by $c_\dR(\mc{G}, \triv_\et, \triv_\Isoc) = \triv_\Isoc \circ c_\dR \circ \triv_\et^{-1} \in G(B_\dR)$. This period matrix, combined with knowledge of $[b]$, uniquely determines the triple $(\mc{G}, \triv_\et, \triv_\Isoc)$ up to isomorphism of $\mc{G}$ matching the trivializations. In other words, if we write $\mc{M}_{b}(C)$ for the set of such isomorphism classes, then we obtain an injection 
$c_\dR: \mc{M}_{b}(C) \rightarrow G(B_\dR)$. We also have (see \cref{s.filtrations-and-lattices} for the notation):
\begin{enumerate}
    \item The \emph{de Rham lattice period} $\pi_{\mc{L}_\dR}(\mc{G}, \triv_\et) \in \Gr_G(C)$ classifying the $B^+_\dR$-lattice $ \omega_{\mc{L}_\dR} \circ \mc{G}$ on the trivial $G$-bundle $\omega_\et \circ \mc{G} =_{\triv_\et} \omega_\std$.
    \item The \emph{\'{e}tale lattice period} $\pi_{\mc{L}_\et}(\mc{G}, \triv_\Isoc) \in \Gr_G(C)$ classifying the $B^+_\dR$-lattice $\omega_{\mc{L}_\et} \circ \mc{G}$ on the trivial $G$-bundle $\omega_\Isoc \circ \mc{G} =_{\triv_\Isoc} \omega_\std \otimes \breve{\mbb{Q}}_p$.
\end{enumerate}

These are related by the following diagram:
\[\begin{tikzcd}[column sep=small]
	& { \substack{\mc{M}_{b}(C) \subseteq G(B_\dR) \\ \\ c_\dR(\mc{G}, \triv_\Isoc, \triv_\et) = g} } \\
	\\
	{\substack{Gr_{G}(C) = G(B_\dR)/G(B^+_\dR) \\ \\ \pi_{\mc{L}_\et}(\mc{G},\triv_\Isoc)=gG(B^+_\dR)}} && {\substack{\Gr_{G}(C) = G(B_\dR)/G(B^+_\dR) \\ \\\pi_{\mc{L}_\dR}(\mc{G}, \triv_\et)=g^{-1}G(B^+_\dR) }}
	\arrow["{\pi_{\mc{L}_\dR}}"{description}, from=1-2, to=3-3]
	\arrow["{\pi_{\mc{L}_\et}}"{description}, from=1-2, to=3-1]
\end{tikzcd}\]

\begin{remark}\label{rmk:lattice-period-maps-G_K-equivariant}
If $C = \ol{K}^\wedge$ for a $p$-adic field $K$, then the period mappings $\pi_{\mc{L}_\et}$ and $\pi_{\mc{L}_\dR}$ are $\mf{G}_K = \mr{Gal}(\ol{K}/K)$-equivariant for the natural actions. Indeed, if $g \in G(B_\dR)$, the lattice $gG(B_\dR^+)$ on the trivial $G$-bundle sends a representation $V \in \Rep\+ G$ to the lattice $gV_{B_\dR^+} \subset V_{B_\dR}$. Then 
    \[ \sigma \cdot g V_{B_\dR^+} = \sigma(g) \sigma^{-1}(V_{B_\dR^+}) = \sigma(g) V_{B_\dR^+} \]
since $\sigma$ preserves $B_\dR^+$.
\end{remark}

For $[\mu]$ such that $b \in B(G, [\mu])$, write $\mc{M}_{b,[\mu]}(C) \subseteq \mc{M}_{b}$ for the subset consisting of $\mc{G}$ that are good of type $[\mu]$. This is equivalent to requiring that $c_\dR(\mc{G}, \triv_\Isoc, \triv_\et)$ lies in the cell $G(B^+_\dR)t^\mu G(B^+_\dR)$. In this case we also have:
\begin{enumerate}
    \item The Hodge-Tate filtration period $\pi_{\HT}(\mc{G}, \triv_\et) \in \Fl_{[\mu]}(C)$ classifying the filtration $\omega_\HT$ on the trivial $G$-bundle $\omega_\et \otimes C =_{\triv_\et} \omega_\std \otimes C$.
    \item The Hodge filtration period $\pi_{\Hdg}(\mc{G}, \triv_\Isoc) \in \Fl_{[\mu^{-1}]}(C)$ classifying the filtration $\omega_\Hdg$ on the trivial $G$-bundle $\omega_\Isoc \otimes C =_{\triv_\Isoc} \omega_\std \otimes C$. 
\end{enumerate} 

Writing $[t^\mu] := t^\mu G(B_\dR^+)/G(B_\dR^+) \in \Gr_G(C)$, the diagram is then refined to

\[\begin{tikzcd}[column sep=small]
	& { \substack{\mc{M}_{b,[\mu]}(C) \subseteq G(B^+_\dR)t^\mu G(B^+_\dR) \\ \\ c_\dR(\mc{G}, \triv_\Isoc, \triv_\et) = g_1 t^\mu g_2} } \\
	\\
	{\substack{Gr_{[\mu]}(C) = G(B^+_\dR)[t^\mu]\\ \\ \pi_{\mc{L}_\et}(\mc{G},\triv_\Isoc)=g_1 [t^\mu]}} && {\substack{\Gr_{[\mu^{-1}]}(C) = G(B^+_\dR)[t^{-\mu}]  \\ \\ \pi_{\mc{L}_\dR}(\mc{G}, \triv_\et)=g_2^{-1}[t^{-\mu}]}} \\
	\\
	{\substack{\Fl_{[\mu^{-1}]}(C)=G(C)/P_{\mu^{-1}}(C) \\ \\ \pi_{\Hdg}(\mc{G},\triv_\Isoc)=\overline{g}_1P_{\mu^{-1}}} } && {\substack{\Fl_{[\mu]}(C)=G(C)/P_{\mu}(C) \\ \\ \pi_{\HT}(\mc{G},\triv_\Isoc)=\overline{g}_2^{-1}P_{\mu}} }
	\arrow["{\pi_{\mc{L}_\dR}}"{description}, from=1-2, to=3-3]
	\arrow["{\pi_{\mc{L}_\et}}"{description}, from=1-2, to=3-1]
	\arrow["\BB"{description}, from=3-1, to=5-1]
	\arrow["{\pi_\Hdg}"{description}, curve={height=-12pt}, from=1-2, to=5-1]
	\arrow["{\pi_\HT}"{description}, curve={height=18pt}, from=1-2, to=5-3]
	\arrow["\BB"{description}, from=3-3, to=5-3]
\end{tikzcd}\]

Note $\mc{G} \in G\dash\AdmPair(C)$ is $\overline{C}_0$-analytic if and only if, for one (equivalently, any) choice of $\triv_{\Isoc}$,  $\pi_{\mc{L}_\et}([\mc{G},\triv_{\Isoc}])$ is in the image of the canonical lattice map 
\[ \mc{L}_\can: \Fl_G(\overline{C}_0) \rightarrow \Gr_G(C). \]

In particular, to detect the $\overline{C}_0$-analytic points we may assume $C=\overline{C}_0^\wedge$ so that by \cref{prop.filt-det-latt-criteria}-(2) these are exactly the points of $\Gr_G(C)$ stabilized by a finite index subgroup $\mf{I}$ of $\Gal(\overline{C}_0/C_0)$. If $K=\overline{C}_0^{\mf{I}}$, the points stabilized by $\mf{I}$ are exactly those $G$-admissible pairs factoring through $\AdmPair^{\goodred}(K)$. For each $V \in \Rep\+ G$, we have the crystalline Galois representation $\mf{I} \to \GL(\omega_\et \circ \mc{G}(V))$ from \S \ref{ss.admissible-pairs-with-good-reduction}. These are compatible for varying $V \in \Rep\+ G$, and arise via Tannakian theory from a crystalline $G$-valued representation
    \[ \rho = \rho(\mc{G}, \triv_\et): \mf{I} \rightarrow G(\mbb{Q}_p), \qquad \rho(\sigma) = \triv_\et^{-1} \circ \rho_\et(\sigma) \circ \triv_\et. \]
The next lemma shows how this can be recovered from the period of $\mc{G}$ (relative to the trivializations $\triv_\et, \triv_\Isoc$). 

\begin{lemma}\label{lemma:crystalline-galois-rep-etale-trivialized}
For $g = c_\dR(\mc{G}, \triv_\Isoc, \triv_\et) \in G(B_\dR)$ as above, the associated Galois representation $\rho: \mf{I} \rightarrow G(\mbb{Q}_p)$ is recovered by the action on $g$ by 
\[ \sigma(g)=g \rho(\sigma), \]

\end{lemma}
\begin{proof}
To see this, we trace definitions noting that the action of $\sigma$ on $c_\dR$ is 
\[ \sigma \cdot c_\dR = (\mr{Id}_{\omega_\Isoc} \otimes \sigma) \circ c_\dR \circ (\mr{Id}_{\omega_\et} \otimes \sigma^{-1}). \]
We substitute the defining property \eqref{eqn:galois-rep-admissible-phi-module} for $\rho_\et$,
\[ \mr{Id}_{\omega_\Isoc} \otimes \sigma = c_\dR \circ (\rho_\et(\sigma)\otimes \sigma)  \circ  c_{\dR}^{-1}, \]
to get 
\[ \sigma \cdot c_\dR = c_\dR \circ \rho_\et(\sigma) \otimes \mr{Id}_{B_\dR}. \]
Putting all of this together, we see that 
\begin{align*}
    \sigma(g) &= (\mr{Id}_{\omega_\std} \otimes \sigma) \circ g \circ (\mr{Id}_{\omega_\std} \otimes \sigma^{-1}), \\
    &= (\mr{Id}_{\omega_\std} \otimes \sigma) \circ (\triv_\Isoc^{-1} \circ c_\dR \circ \triv_\et) \circ (\mr{Id}_{\omega_\std} \otimes \sigma^{-1}), \\
    &= \triv_\Isoc^{-1} \circ (\mr{Id}_{\omega_\Isoc} \otimes \sigma) \circ c_\dR \circ (\mr{Id}_{\omega_\et} \otimes \sigma^{-1}) \circ \triv_\et, \\
    &= \triv_\Isoc^{-1} \circ (\sigma \cdot c_\dR) \circ \triv_\et, \\
    &= \triv_\Isoc^{-1} \circ (c_\dR \circ \rho_\et(\sigma) \otimes \mr{Id}_{B_\dR}) \circ \triv_\et, \\
    &= (\triv_{\Isoc}^{-1} \circ c_\dR \circ \triv_\et) \circ (\triv_\et^{-1}\circ \rho_\et(\sigma) \otimes \mr{Id}_{B_\dR} \circ \triv_\et), \\
    &= g\rho(\sigma).
\end{align*}

\end{proof}

We obtain the following precise analog of the Grothendieck period conjecture:

\begin{corollary}
If $\mc{G}$ is a $\overline{C}_0$-analytic $G$-admissible pair and $G=\MG(\mc{G})$, then $c_\dR \circ \mc{G}$ is a generic point of the torsor of isomorphisms between $\omega_\et \circ \mc{G}$ and $\omega_\Isoc \circ \mc{G}$. 
\end{corollary}
\begin{proof}
We need to show that the $\mf{I}$-orbits of $c_\dR \circ \mc{G}$ are Zariski dense in $\mr{Isom}^\otimes(\omega_\et \circ \mc{G}, \omega_\Isoc \circ \mc{G})$, but this follows since the image of $\rho$ is Zariski dense in $G$ by \cref{corollary.galois-rep-dense-open}.
\end{proof}

\subsection{Admissible pairs with complex multiplication}\label{ss.cm-adm-pair}

An admissible pair $M$ has complex multiplication (CM) if its motivic Galois group $\MG(M)$ is a torus. We now give the easy direction of \cref{main.transcendence}/\cref{main.G-structure-period-maps}:
\begin{proposition}\label{prop.cm-has-rig-an-periods}
Suppose $\mc{G} \in G\dash\AdmPair(C)$ has complex multiplication. Then, for any choice of $\triv_\et$ and $\triv_{\Isoc}$, $\pi_{\mc{L}_\et}(\mc{G}, \triv_{\Isoc})$ is $\overline{\breve{\mbb{Q}}}_p$-analytic and $\pi_{\mc{L}_\dR}(\mc{G}, \triv_\et)$ is $\overline{\mbb{Q}}_p$-analytic. In particular, if $M \in \AdmPair(C)$ has complex multiplication then it is $\overline{\breve{\mbb{Q}}}_p$-analytic  and its de Rham lattice is $\overline{\mbb{Q}}_p$-analytic.
\end{proposition}
\begin{proof}
By functoriality of periods, we may assume $G=T$ is a torus. Then every cocharacter is minuscule so each cell of the affine Grassmannian is equal to the flag variety and consists of a single point, thus is defined over a finite extension of the base field (which is $\mbb{Q}_p$ for the de Rham lattice and $\breve{\mbb{Q}}_p$ for the \'{e}tale lattice). 
\end{proof}

For completeness, we recall now the classification of CM admissible pairs, a result due to Ansch\"utz \cite{anschutz} in the language of Breuil-Kisin-Fargues modules and essentially equivalent to a result of Serre \cite[Th\'{e}or\`{e}me 5 and Th\'{e}or\`{e}me 6]{serre:groupes-alg-ht} on abelian Galois representations. This classification will not be used in any of our results.

We write $\AdmPair^{\CM}(C)$ for the Tannakian subcategory of CM admissible pairs. By \cref{prop.cm-has-rig-an-periods}, $\AdmPair^\CM(C) \subseteq \AdmPair(\overline{\breve{\mbb{Q}}}_p)$. In fact, the motivic Galois group $S$ of $\AdmPair^\CM(C)$ is the maximal abelian quotient of the motivic Galois group of $\AdmPair(\overline{\breve{\mbb{Q}}}_p)$ --- there can be no additive component because the only $\overline{C}_0$-analytic extension of the trivial admissible pair by itself is trivial by \cref{example.no-non-trivial-rigid-analytic-ext-of-trivial-by-itself}. Note also that, for $T$ a torus, any $T$-admissible pair is basic, so $\AdmPair^\CM(C)$ is equivalent via \cref{theorem.adm-pair-propeties} to the category $\HS^\CM(C)$ of $p$-adic Hodge structures with complex multiplication. 

Consider the pro-torus over $\mbb{Q}_p$, $S'=\lim_{\mbb{Q}_p \subseteq E \subseteq C, [E:\mbb{Q}_p]<\infty} \mr{Res}^{E}_{\mbb{Q}_p} \mbb{G}_m$, where the transition maps are the norm maps. Then the character group $X^*(S')$ is the space of locally constant functions on $\Gal(\overline{\mbb{Q}}_p/\mbb{Q}_p)$, and there is a cocharacter $\mu:\mbb{G}_{m,\overline{\mbb{Q}}_p} \rightarrow S'_{\overline{\mbb{Q}}_p}$ corresponding to 
\[ \mu^*: X^*(S') \rightarrow X^*(\mbb{G}_m)=\mbb{Z}, f \mapsto f(\Id).\]
Then $\mu^{-1}(t)$ classifies an $S'$-admissible pair whose associated $S'$-$p$-adic Hodge structure is classified by $\mu(t)$ (the weight homomorphism corresponds to the map on character groups given by integrating $f/2$). Explicitly, if we consider the representation of $S'$ determine by $S' \twoheadrightarrow  \mr{Res}^{E}_{\mbb{Q}_p} \mbb{G}_m$ and the standard representation of the restriction of scalars given by $E$ acting on itself by multiplication, then the associated admissible pair (resp. $p$-adic Hodge strucure) is given by the covariant Dieudonn\'{e} module of a Lubin-Tate formal group for $\mc{O}_E$ (resp. its Tate module). Arguing as in \cite{anschutz} or \cite{serre:groupes-alg-ht}, one finds the induced map $S \rightarrow S'$ is an isomorphism.

\section{Transcendence}\label{s.transcendence}

In this section we prove \cref{main.transcendence} and \cref{main.G-structure-period-maps}, and give a refinement that applies outside of the basic case. Before giving these proofs, in \cref{ss.complex-transcendence} we describe in more detail the analogous transcendence results over $\mbb{C}$.  

\subsection{Conditional and unconditional results over $\mbb{C}$}\label{ss.complex-transcendence}

Let $A/\mbb{C}$ be an abelian variety. The kernel of the exponential $\Lie\+ A \rightarrow A(\mbb{C})$ is naturally identified with $H_1(A(\mbb{C}),\mbb{Z})$.  The Hodge filtration is the kernel of the induced map
\[ H_1(A(\mbb{C}), \mbb{C})=H_1(A(\mbb{C}), \mbb{Z}) \otimes \mbb{C} \rightarrow \Lie\+ A, \]
which can canonically be identified with $\omega_{A^\vee}$, the invariant differentials on the dual abelian variety $A^\vee$. This weight $-1$ $\mbb{Q}$-Hodge structure determines $A$ up to isogeny, and we say $A$ has complex multiplication if this weight $-1$ Hodge structure does --- this is consistent with our earlier definitions, since $\mbb{Q}$-Hodge structures form a Tannakian category, and is equivalent to the usual definition that $\End(A) \otimes \mbb{Q}$ contains a semisimple commutative subalgebra of dimension $2 \dim A$. We say $A$ is defined over $\overline{\mbb{Q}}$ if the equations defining $A$ can be chosen to have coefficients in $\overline{\mbb{Q}}$, i.e. if there is an abelian variety over $\overline{\mbb{Q}}$ whose base change to $\mbb{C}$ is $A$. 

\begin{theorem}[Cohen \cite{cohen:transcendence} and Shiga-Wolfart \cite{shiga-wolfart:transcendence}, generalizing Schneider \cite{schneider:j}]\label{theorem.complex-av-transcendence} An abelian variety $A/\mbb{C}$ has complex multiplication if and only if $A$ is defined over $\overline{\mbb{Q}}$ and its Hodge filtration is defined over $\overline{\mbb{Q}}$ (i.e. the subspace $\omega_{A^\vee} \subset H_1(A(\mbb{C}), \mbb{Q}) \otimes \mbb{C}$ admits a basis whose elements are $\overline{\mbb{Q}}$-linear combinations of elements in $H_1(A(\mbb{C}), \mbb{Q})$).  
\end{theorem}

Abelian varieties up to isogeny are precisely the weight one motives over $\mbb{C}$. The proof of this theorem makes use of the equivalence of this subcategory with the category of weight $-1$ $\mbb{Q}$-Hodge structures and the W\"ustholz analytic subgroup theorem, a linear transcendence result. \cref{main.complex} gives a conditional generalization of this result to all motives over $\mbb{C}$, and the assumptions match up with these ingredients: we assume the standard conjectures so that there is a good Tannakian category of motives, we assume the Hodge conjecture so that Hodge structures give a fully faithful realization, and we assume the Grothendieck period conjecture as a non-linear transcendence result to get us started. 

We now explain the setup in detail, then prove \cref{main.complex}: Assume the standard conjectures, and let $\Mot(\mbb{C})$ be the category of pure motives over $\mbb{C}$ with $\mbb{Q}$-coefficients. It is equipped with a Betti realization $\omega_{B}: \Mot(\mbb{C}) \rightarrow \Vect(\mbb{Q})$ and an algebraic de Rham realization $\omega_{\dR}: \Mot(\mbb{C}) \rightarrow \Vect(\mbb{C})$. The latter admits a Hodge filtration $F^\bullet \omega_{\dR}: \Mot(\mbb{C}) \rightarrow \Vect^f(\mbb{C})$, and there is a canonical comparison isomorphism $c: \omega_{\dR} \rightarrow \omega_{B} \otimes \mbb{C}$. The category $\Mot(\overline{\mbb{Q}})$ of motives over $\overline{\mbb{Q}}$ is a full subcategory, and the restriction of $F^\bullet \omega_\dR$ to $\Mot(\overline{\mbb{Q}})$ factors canonically through $\Vect^f(\overline{\mbb{Q}})$. Both $\Mot(\mbb{C})$ and $\Mot(\overline{\mbb{Q}})$ are Tannakian categories over $\mbb{Q}$, neutralized by $\omega_B$. Moreover,
\[ \Isom^\otimes(\omega_\dR|_{\Mot(\overline{\mbb{Q}})}, \omega_B \otimes \overline{\mbb{Q}}|_{\Mot(\overline{\mbb{Q}})}) \]
is a scheme over $\overline{\mbb{Q}}$ and the Grothendieck period conjecture says that 
\[ c \in \Isom^\otimes(\omega_\dR|_{\Mot(\overline{\mbb{Q}})}, \omega_B \otimes \overline{\mbb{Q}}|_{\Mot(\overline{\mbb{Q}})})(\mbb{C}) \]
is a generic point, i.e. is $\overline{\mbb{Q}}$-Zariski dense. There is also an exact tensor functor from $\Mot(\mbb{C})$ to the category $\mbb{Q}\dash\HS$ of $\mbb{Q}$-Hodge structures sending  $M \in \Mot(\mbb{C})$ to  $\left(\omega_B(M), c(M)(\Fil^\bullet\omega_\dR(M)) \right)$, and the Hodge conjecture implies this is fully faithful.  

\begin{proof}[Proof of \cref{main.complex}]
The algebraicity for CM-motives is well known, so let us assume both $M$ and the Hodge filtration are defined over $\overline{\mbb{Q}}$. Let $G$ be the motivic Galois group of $M$, i.e. $G=\Aut^\otimes(\omega_B|_{\langle M \rangle})$. It is a connected linear algebraic group over $\mbb{Q}$ (e.g., by the Hodge conjecture and the connectedness of Mumford-Tate groups), and we have a natural tensor equivalence $\Rep\+ G \rightarrow \langle M \rangle$ and an identification of $\omega_\std$ on $\Rep\+G$ with $\omega_B$ on $\langle M \rangle$. If we fix also a trivialization $\omega_\dR|_{\langle M \rangle}=\omega_\std \otimes \overline{\mbb{Q}}$ then we obtain an identification 
\[ \Isom^\otimes(\omega_\dR|_{\langle M \rangle}, \omega_B \otimes \overline{\mbb{Q}}|_{\langle M \rangle})=\Isom^\otimes(\omega_\std \otimes \overline{\mbb{Q}}, \omega_\std \otimes \overline{\mbb{Q}})=G_{\overline{\mbb{Q}}} \]
and, by the Grothendieck period conjecture, the de Rham comparison isomorphism $c$ is a generic point in $G_{\overline{\mbb{Q}}}(\mbb{C})$. Splitting the filtration of $\omega_\dR|_{\langle M \rangle}$ gives a cocharacter $\mu$ of $G_{\overline{\mbb{Q}}}$ so that the Hodge filtration is $F_\mu$. The classifying point for the induced fully faithful functor (by the Hodge conjecture) $\Rep\+ G \rightarrow \mbb{Q}\dash\HS$ is
\[ c \cdot F_\mu \in \Fl_{[\mu]}(\mbb{C}) \]
Since the orbit map for $F_\mu$, $G_{\overline{\mbb{Q}}} \rightarrow \Fl_{[\mu]}$ is dominant, this is a generic point of $\Fl_{[\mu]}/\overline{\mbb{Q}}$, so if the Hodge filtration is $\overline{\mbb{Q}}$-algebraic (with respect to the Betti rational structure), we deduce $c \cdot F_\mu = \Fl_{[\mu]}$ and hence the stabilizer $P_\mu$ of $F_\mu$ is $G$. If this is the case, then each element of $G(\mbb{Q})$ stabilizes the Hodge filtration thus induces an automorphism of the induced tensor functor from $\Rep\+ G$ to $\mbb{Q}$-Hodge structures. By Lemma \ref{lemma.tannakian-automorphisms} the only automorphisms are by $Z(G)(\mbb{Q})$ so we conclude $G(\mbb{Q}) \subseteq Z(G)$. Since the $\mbb{Q}$-points of a connected linear algebraic group over $\mbb{Q}$ are Zariski dense by \cite[Corollary 18.3]{borel:linear-algebraic-groups}, we conclude $G=Z(G)$, i.e. $G$ is abelian. It is thus a product of a torus and an additive group, but the additive part is trivial as there are no non-trivial extensions of the trivial Hodge structure by itself (this is already true at the level of filtrations)\footnote{Of course we could also exclude the possibility of an additive factor by using that $G$ is reductive because of the polarization on motives, but the given argument better mirrors the $p$-adic case.}. Thus, $G$ is a torus. 
\end{proof}

\subsection{Proof of \cref{main.transcendence} and \cref{main.G-structure-period-maps}} Given the setup of \cref{ss.periods}, \cref{main.G-structure-period-maps} is an immediate consequence of \cref{main.transcendence} since for any connected linear algebraic group $G$, $\Rep\+ G$ admits a single tensor generator. We now prove \cref{main.transcendence}. To that end, we fix an algebraically closed non-archimedean extension $C/\mbb{Q}_p$. Let $\kappa=\mc{O}_C/\mf{m}_C$ be the residue field, and let $C_0=W(\kappa)[1/p]$, so that any $p$-adic subfield of $C$ is contained in the algebraic closure $\overline{C}_0$ of $C_0$ in $C$. 

In \cref{prop.cm-has-rig-an-periods} we saw that any CM admissible pair over $C$ is $\overline{\breve{\mbb{Q}}}_p$-analytic and has $\overline{\mbb{Q}}_p$-analytic de Rham lattice. Suppose now that $M\in \AdmPair^{\mr{basic}}(\overline{C}_0)$, and that the de Rham lattice of $M$ is also $\overline{C}_0$-analytic. We may thus fix a finite extension $K/C_0$
such that $M$ is an admissible pair with good reduction over $K$ and such that the de Rham lattice is also defined over $K$. Let $\mf{I}=\Gal(\overline{C}_0/K)$. 

Let $G$ be the motivic Galois group of $M$, and let
\[ \mc{G}: \Rep\+ G \rightarrow \AdmPair^\goodred(K) \subseteq \AdmPair(C) \]
be the canonical $G$-structure for $M$, so that there is a canonical trivialization $\omega_\et \circ \mc{G}=\omega_\std$. Fix also a trivialization $\omega_\Isoc \otimes_{\breve{\mbb{Q}}_p} K=\omega_\std$ (viewed as a fiber functor to $\Vect(K)$). Then, as in \cref{ss.periods}, $c_\dR$ corresponds (via the chosen trivializations) to an element $g \in G(B_\dR)$, and the de Rham lattice is classified by 
\[ g^{-1} G(B^+_\dR) \in \Gr_G(C)=G(B_\dR)/G(B^+_\dR). \]
As in \cref{lemma:crystalline-galois-rep-etale-trivialized}, we find that $\mf{I}$ acts on $g$ by
\[ \sigma(g) = g \rho(\sigma) \qquad \sigma \in \mf{I}, \]
where $\rho: \mf{I} \rightarrow G(\mbb{Q}_p)$ is the associated crystalline Galois representation. The de Rham lattice period mapping is $\mf{I}$ equivariant (\cref{rmk:lattice-period-maps-G_K-equivariant}) so the Galois action on the de Rham lattice period multiplies it on the left $\rho(\sigma)^{-1}$. By assumption, the de Rham lattice period is preserved by this action (\cref{remark:de-Rham-lattice-Galois-equivariant}), so we conclude that $\rho(\sigma) \in G(\mbb{Q}_p)$ preserves the de Rham lattice. The induced functor from $\Rep\+ G$ to $p$-adic Hodge structures is fully faithful and $\rho(\sigma)$ is an automorphism of this functor so \cref{lemma.tannakian-automorphisms} shows that $\rho(\sigma) \in Z(G)(\mbb{Q}_p)$. However, the image of $\rho$ is Zariski dense in $G$ by \cref{corollary.galois-rep-dense-open}, so we conclude $Z(G)=G$, so $G$ is abelian. It is thus a product of a torus part and an additive part, but the additive part must be trivial because there is no non-trivial $\overline{C}_0$-analytic extension of the trivial admissible pair by itself (\cref{example.no-non-trivial-rigid-analytic-ext-of-trivial-by-itself}). Thus $G$ is a torus. 

\begin{remark}\label{remark.relation-with-one-dim-proof}
The proof is similar to the proof for one-dimensional formal groups given in \cite{howe:transcendence}, but at the time we only understood weaker tools --- indeed, in \cite{howe:transcendence} we used Tate's full faithfulness of the $p$-adic Tate module over a $p$-adic field in place of the crystalline comparison here, and the Scholze-Weinstein classification in place of the equivalence between basic admissible pairs and $p$-adic Hodge structures here. Happily, we have skipped over the analog of the abelian varieties step (where progress has halted in the archimedean theory) and gone straight to an unconditional analog of \cref{main.complex}. Nonetheless, it would still be interesting to find another proof for isoclinic formal groups mirroring the use of the W\"ustholz analytic subgroup theorem in \cref{theorem.complex-av-transcendence}.
\end{remark}

\begin{remark}\label{remark.real-hs-transcendence}
The same methods would give an unconditional theorem for extended real Hodge structures, but the result is not very interesting: in this case, the condition for an extended real Hodge structure analogous to having an $\overline{C}_0$-analytic de Rham lattice is simply to be a real Hodge structure, but every real Hodge structure has complex multiplication (indeed, the motivic Galois group of the category of real Hodge structures is the Deligne torus $\mr{Res}_{\mbb{C}/\mbb{R}} \mbb{G}_m$). Thus in the archimedean case one must use an underlying global structure on the coefficients to have an interesting transcendence theory, whereas in the non-archimedean case there is already a rich purely local transcendence theory. 
\end{remark}

\subsection{Beyond the basic case}\label{ss.beyond-basic}

In \cite[Conjecture 4.1]{howe:transcendence}, one of the authors made a conjecture based on the result for one-dimensional formal groups that does not really make sense as written (regretfully, we failed to noticed at the time that isogeny Breuil-Kisin-Fargues modules were not an abelian category!). The natural correction is to use rigidified Breuil-Kisin-Fargues modules. Doing so, the conjecture becomes that a $\overline{C}_0$-analytic admissible pair with $\overline{C}_0$-analytic Hodge-Tate filtration admits complex multiplication. \cref{main.transcendence} thus proves a weaker statement where we require the stronger condition that the de Rham lattice is $\overline{C}_0$-analytic and restrict ourselves to basic admissible pairs. Note that in the case of minuscule cocharacter, including all settings that apply to $p$-divisible groups, the Hodge-Tate filtration uniquely determines the de Rham lattice, so that the corrected conjecture is now proved in these basic minuscule cases. With a more complete understanding of the structures that allow us to prove \cref{main.transcendence}, it no longer seems reasonable to conjecture that $\overline{C}_0$-analyticity of the Hodge-Tate filtration alone would suffice beyond the minuscule case --- it would be interesting to have an example!

With the correction described above, the conjecture does not hold outside the basic case. Indeed, already the admissible pair attached to an ordinary elliptic curve without complex multiplication gives an example where it fails. This may appear at odds with the full results of \cite{howe:transcendence}, which also treat one-dimensional $p$-divisible groups with an \'{e}tale part, but note that in \cite{howe:transcendence} we did not take into account the rigidification and characterized CM purely in terms of endomorphisms; without the rigidification the \'{e}tale part of a $p$-divisible group can always be split off from the connected part over $\mc{O}_C$. 

In the non-basic counterexample given by an ordinary elliptic curve, note that the slope filtration still lifts, and each graded part for the slope filtration has complex multiplication. In fact, this holds in general, as we explain now.

The slope filtration on the category of isocrystals is the increasing $\mbb{Q}$-filtration $\Fil_\lambda(W)=\bigoplus_{\lambda' \leq \lambda} W_{\lambda'}$, where $W_{\lambda'}$ denotes the $D_{\lambda'}$-isotypic component of $W$. For $\mc{M} \subseteq \AdmPair(C)$ a Tannakian sub-category, we say the slope filtration lifts to $\mc{M}$ if $\Fil_\lambda(W)$ underlies a sub-object for any $(W, \mc{L}_\et) \in \mc{M}$. Note that it is equivalent to say that $\Aut^\otimes(\omega_\Isoc|_{\mc{M}})$, where $\omega_\Isoc$ is viewed as a fiber functor to $\Vect(\breve{\mbb{Q}}_p)$, preserves the slope filtration. 

\begin{example}
The slope filtration lifts to the subcategory of basic admissible pairs (because the slope grading does!). 
\end{example}

\begin{theorem}\label{theorem.transcendence-with-slope}
Let $M \in \AdmPair(\overline{C}_0)$ and suppose the de Rham lattice of $M$ is $\overline{C}_0$-analytic. Then the slope filtration lifts to $\langle M\rangle$, and the associated graded has complex multiplication. 
\end{theorem}
\begin{proof}
If the slope filtration lifts then the associated graded is basic, $\overline{C}_0$-analytic, and has $\overline{C}_0$-analytic de Rham lattice, so \cref{main.transcendence} implies it has complex multiplication. Thus it remains only to show that the slope filtration lifts. 

For this, we may proceed as in the proof of \cref{main.transcendence} to obtain $\mf{I}$ and $\rho$ such that, for $\sigma \in \mf{I}$, $\rho(\sigma)$ preserves the de Rham lattice. This implies that $\rho(\sigma)$ induces an automorphism of the vector bundle $\mc{E}(W)$ for any $(W,\mc{L}_\et) \in \langle M \rangle$. Such an automorphism preserves the slope filtration of $\mc{E}(W)$, since morphisms of semistable vector bundles on $\FF_{C^\flat}$ only go up in slope. Thus, after we use $c_\dR$ to identify $G(B_\dR)$ with $\Aut^\otimes(\omega_{\Isoc})(B_\dR)$, we find $g$ preserves the slope filtration of $W_{B_\dR}$. But since the image of $\rho$ is Zariski dense\footnote{Here we use that the image of $\rho$ is a Zariski dense set of $\mbb{Q}_p$-points of $G$, and that this implies the induced set of $L$-points in $G_L$ for any $L/\mbb{Q}_p$ is also Zariski dense --- to see this second point, fix a basis of $L$ as a $\mbb{Q}_p$-vector space to obtain a basis for $\mc{O}(G_L)=\mc{O}(G)\otimes_{\mbb{Q}_p} L$ as a free $\mc{O}(G)$-module, then argue separately with each term in any linear combination of basis vectors.} in $G_{B_\dR}$ by \cref{corollary.galois-rep-dense-open}, we conclude that $\Aut^\otimes(\omega_{\Isoc}|_{\langle M \rangle})$, viewed as a fiber functor to $\breve{\mbb{Q}}_p$, preserves the slope filtration. In other words, the slope filtration lifts to $\langle M \rangle$.  
\end{proof}

\bibliographystyle{plain}
\bibliography{refs}

\end{document}